\documentclass{article}
\usepackage{arxiv}
\usepackage[utf8]{inputenc} % allow utf-8 input
\usepackage[T1]{fontenc}    % use 8-bit T1 fonts
\usepackage{hyperref}       % hyperlinks
\usepackage{url}            % simple URL typesetting
\usepackage{booktabs}       % professional-quality tables
\usepackage{amsfonts}       % blackboard math symbols
\usepackage{nicefrac}       % compact symbols for 1/2, etc.
\usepackage{microtype}      % microtypography
\usepackage{lipsum}
\usepackage{amsthm}
\usepackage{graphicx}
\usepackage{authblk}
\def\standard{\oddsidemargin=0in
              \evensidemargin=0in
              \topmargin =-.5in
              \textheight=8.5in
              \textwidth=6.5in}

{\end{list}}
\def\bm{{\bf m}}

\def\b1{{\bf 1}}

\def\blot{\quad {$\vcenter{\vbox{\hrule height.4pt
             \hbox{\vrule width.4pt height.9ex \kern.9ex \vrule 
width.4pt}
             \hrule height.4pt}}$}}
\usepackage{mathtools}
\DeclarePairedDelimiter\ceil{\lceil}{\rceil}

\usepackage{comment}
\usepackage{multirow}
\usepackage[ruled,norelsize]{algorithm2e}

\SetAlFnt{\small}
\SetAlCapFnt{\small}
\SetAlCapNameFnt{\small}
\SetAlCapHSkip{0pt}
\IncMargin{-\parindent}
\usepackage{algorithmic}
\usepackage{wrapfig}
\usepackage{amsmath}
\allowdisplaybreaks

\newtheorem{theorem}{Theorem}

\newtheorem{lemma}{Lemma}

\newtheorem{proposition}{Proposition}

\standard
\usepackage{graphicx}
\usepackage{setspace}
\usepackage{verbatim} 
\usepackage[numbers]{natbib} 

\usepackage{latexsym}
\usepackage{amsmath}
\usepackage{nomencl}
\makenomenclature

\usepackage{etoolbox}
\renewcommand\nomgroup[1]{%
  \item[\bfseries
  \ifstrequal{#1}{F}{Functions}{%
  \ifstrequal{#1}{V}{Variables}{%
  \ifstrequal{#1}{S}{Sets}{%
  \ifstrequal{#1}{E}{Empirical Study}}}}%
]}
\usepackage{amsbsy}    % math symbols
\usepackage{bbm}
\usepackage{color}
\usepackage{soul}
\begin{document}

\title{Global-Local Metamodel Assisted Two-Stage Optimization via Simulation}

% \author{
%   Wei~Xie^ \\
% %   Department of Computer Science\\
%   Northeastern University\\
%   Boston, MA 02115 \\
%   \texttt{w.xie@northeastern.edu} 
%   %% examples of more authors 
%   \And
%  Yuan Yi \\
% %   Department of Electrical Engineering\\
%   Rensselaer Polytechnic Institute\\
%   Troy, NY 12180\\
%   \texttt{yiy2@rpi.edu} \AND
%  Hua Zheng\\
% %   Department of Electrical Engineering\\
%   Northeastern University\\
%   Boston, MA 02115 \\
%   \texttt{zheng.hua1@husky.neu.edu} 
%   %% \And
%   %% Coauthor \\
%   %% Affiliation \\
%   %% Address \\
%   %% \texttt{email} \\
%   %% \And
%   %% Coauthor \\
%   %% Affiliation \\
%   %% Address \\
%   %% \texttt{email} \\
% }
\author[ 1]{Wei Xie\thanks{Corresponding author: w.xie@northeastern.edu}
}
\author[2]{Yuan Yi}
\author[1]{Hua Zheng}
\affil[1]{Northeastern University, Boston, MA 02115}
\affil[2]{Rensselaer Polytechnic Institute, Troy, NY 12180}
\maketitle
\begin{abstract}

	To integrate strategic, tactical and operational decisions, the two-stage optimization has been widely used to guide dynamic decision making. In this paper, we study the two-stage stochastic programming for complex systems with unknown response estimated by simulation. 
	We introduce the global-local metamodel assisted two-stage optimization via simulation that can efficiently employ the simulation resource to iteratively solve for the optimal first- and second-stage  decisions. Specifically, at each visited first-stage decision, we develop a local metamodel to simultaneously solve a set of scenario-based second-stage optimization problems, which also allows us to estimate the optimality gap. Then, we construct a global metamodel accounting for the errors induced by: (1) using a finite number of scenarios to approximate the expected future cost occurring in the planning horizon, (2) second-stage optimality gap, and (3) finite visited first-stage decisions. Assisted by the global-local metamodel, we propose a new simulation optimization approach that can efficiently and iteratively search for the optimal first- and second-stage decisions. Our framework can guarantee the convergence of optimal solution for the discrete two-stage optimization with unknown objective, and the empirical study indicates that it achieves substantial efficiency and accuracy. 

\end{abstract}


% At a minimum you need to supply the author names, year and a title.
% IMPORTANT:
% Full first names whenever they are known, surname last, followed by a period.
% In the case of two authors, 'and' is placed between them.
% In the case of three or more authors, the serial comma is used, that is, all author names
% except the last one but including the penultimate author's name are followed by a comma,
% and then 'and' is placed before the final author's name.
% If only first and middle initials are known, then each initial
% is followed by a period and they are separated by a space.
% The remaining information (journal title, volume, article number, date, etc.) is 'auto-generated'.

%\begin{bottomstuff}
%Author's addresses: Y. Yi, email:~yiy2@rpi.edu; Department of Industrial and Systems Engineering,
%Rensselaer Polytechnic Institute, Troy, NY 12180-3590;
%W. Xie (corresponding author), email:~xiew3@rpi.edu; Department of Industrial and Systems Engineering,
%Rensselaer Polytechnic Institute, Troy, NY 12180-3590;  

%\end{bottomstuff}

%\renewcommand{\shortauthors}{Xie, Yi and Zheng}

	\section{Introduction}
	\label{intro}
	
	In many applications, such as high-tech manufacturing, bio-pharmaceutical supply chains, and smart power grids with renewable energy, we often need to integrate the strategic, tactical and operational decisions. For example, in the semiconductor manufacturing, the managers need to consider the facility planning and the ensuing production scheduling. The planning decision is made ``here and now", and the production scheduling is a ``wait and see" decision, which depends on the investment decision and also the realization of demand. \textit{To guide the dynamic decisions making, in this paper, we consider the two-stage stochastic programming,} % \citep{Dempster_etal_1981},
	\begin{equation}
	\min_{\mathbf{x} \in \mathcal{X}}  \mbox{  }G(\mathbf{x})\equiv{c_0(\mathbf{x})}+\mbox{E}_{\pmb{\xi}}
	\left[
	\min_{\mathbf{y} \in \mathcal{Y}(\mathbf{x})} {q}(\mathbf{x}, \mathbf{y}, \pmb{\xi})\right]
	\label{eq: example}
	\end{equation} 
	where $\mathbf{x}$ denotes the first-stage action, e.g., the investment decision, $\mathbf{y}$ denotes the second-stage decision, e.g., production scheduling, $\pmb{\xi}$ denotes the random inputs, e.g., demands, $\mathcal{X}$ and $\mathcal{Y}(\mathbf{x})$ represent the feasible decision sets. The overall cost includes the investment cost and the expected production cost occurring in the planning horizon.

	The existing two-stage optimization approaches often assume that the response function ${q}(\mathbf{x}, \mathbf{y}, \pmb{\xi})$ is known \cite{Beale_1955, Dantzig_1955, Shapiro_Philpott_2007, shapiro_2009}. For example, it can be a linear or mixed-integer function.
	However, for many complex real systems, the response function could be unknown. For example, in the semiconductor manufacturing, the production processes can involve thousands of steps and the production cost function is unknown \citep{Liu_etal_2011}.
	{We resort to simulation for the unknown response under different scenarios $\pmb{\xi}$ and decisions $(\mathbf{x},\mathbf{y})$.}  
	%Suppose that there are finite sets of feasible decisions, $\mathcal{X}$ and $\mathcal{Y}(\mathbf{x})$. 
	Thus, in this paper, we consider the \textit{two-stage optimization via simulation (OvS)}. 
	
	Compared to the classical two-stage optimization with the known response function, such as two-stage stochastic linear programming problems, there exist additional challenges to solve the two-stage OvS listed as follows.	
	\begin{enumerate}
		\setcounter{enumi}{0}
		
		\item %It is challenging to extract structural information for optimization. The objective function is unknown. Therefore, we cannot extend existing algorithms well developed in the stochastic optimization community that directly exploit the structural information to obtain the optimal solution for the two-stage stochastic DOvS. This is fundamentally different from the classical two-stage stochastic problem with a known objective. 
		When the response function is known, various algorithms exploiting the structural information haven been developed in the optimization community to search for the optimal solution,  including Benders decomposition algorithm and stochastic decomposition \citep{Ekin_Polson_Soyer_2014, Frauendorfer_1988, Sen_Higle_1991, Higle_Sen_1994}.  However, they cannot be employed and extended to the two-stage OvS of interest.

		\item  It is computationally demanding to solve the two-stage OvS. For complex stochastic systems, each simulation run could be computationally expensive. In addition, there could exist high prediction uncertainty and the number of scenarios used by the Sample Average Approximation (SAA) to approximate the expected future cost needs to be large \citep{Ahmed_Shapiro_2002,shapiro_2009}. Hence, there exists tremendous computational burden. 
		
		\item It is challenging to search for the optimal first-stage solution. Given limited computational resource, we often cannot find the true optimal second-stage decisions, which leads to the optimality gap. Thus, besides the finite sampling error introduced by SAA, this optimality gap can further lead to a biased estimate of the expected future cost.
		
	\end{enumerate}
	Hence, in this paper, we introduce a new simulation optimization approach that allows us to efficiently solve the complex two-stage OvS problems.

	Notice that the problem of interest is different with the existing black-box OvS problems studied in the simulation literature %i.e., OvS problems with unknown objectives 
	\citep{Spall_1992, Jones_1998, Hong_Nelson_2006, Huang_etal_2006G, Sun_etal_2014}.  Existing studies tend to focus on the stochastic \textit{one-stage} OvS, 
	\begin{equation}
	\min_{\mathbf{x} \in \mathcal{X}}  \mbox{  }  \mbox{E}_{\pmb{\xi}} \left[ {f}(\mathbf{x}, \pmb{\xi}) \right].
	\label{eq: sto1}
	\end{equation} 
	Given a feasible decision $\mathbf{x}$, the system unknown mean response, denoted by $\mbox{E}_{\pmb{\xi}} \left[ {f}(\mathbf{x}, \pmb{\xi}) \right]$, can be assessed by simulation.
	Various simulation optimization algorithms are proposed to solve  one-stage OvS; see Henderson and Nelson~\citep{Henderson_Nelson_2006} for a review. In particular, metamodel-assisted optimization approaches can efficiently employ the simulation resource for the search of optimal solution \citep{Henderson_Nelson_2006, Fu_2014}. When there is no strong prior information on the mean response surface, the Gaussian process (GP) can be used to characterize the remaining metamodel estimation uncertainty. To balance exploration and exploitation, 
	Sun et al.~\citep{Sun_etal_2014} proposed a GP based search (GPS) algorithm for discrete optimization problems, and it can efficiently use the simulation resource and guarantee global convergence.
	
	\textit{Inspired by those one-stage metamodel assisted approaches, in this paper, we propose a {global-local metamodel-assisted two-stage OvS}.}  Specifically, at each visited first-stage action $\mathbf{x}$, we construct a local GP metamodel for $q_{\mathbf{x}}(\mathbf{y},\pmb{\xi})\equiv q(\mathbf{x},\mathbf{y},\pmb{\xi})$ so that we can simultaneously solve a large set of second-stage optimization problems sharing the same first-stage decision $\mathbf{x}$. Then, built on the search results from the second-stage optimization problems, we further develop a global metamodel accounting for various sources of errors. % \textcolor{red}{errors induced by: (1) using a finite number of scenarios to approximate the expected future cost occurring in the planning horizon, (2) second-stage optimality gap, and (3) finite visited first-stage decisions}. 
	Assisted by the global-local metamodel, we introduce a two-stage optimization via simulation approach that can efficiently employ the limited simulation budget to iteratively search for the optimal first- and second-stage decisions. 
	{Here, suppose each simulation run could be computationally expensive, say taking about a few days.}

	Therefore, the main contributions of this paper are as follows.  
	\begin{itemize}
		\setcounter{enumi}{0}
		
		%\item To the authors' best knowledge, our approach is the first algorithm that solves the two-stage stochastic DOvS where the system response is assessed by simulation.

		\item We propose a global-local metamodel accounting for the finite sampling error introduced by using a finite number of scenarios to approximate the expected future cost and also the bias introduced by the optimality gap from the second-stage optimization. 
		
		\item Assisted by the global-local metamodel, we develop a two-stage optimization via simulation approach that can simultaneously control the impact from various sources of error %balance exploration and exploitation, 
		and efficiently employ the simulation budget to search for the optimal first- and second-stage decisions.

		\item 	Our approach can guarantee global convergence as the simulation budget increases. The empirical study also demonstrates that it has good and stable finite-sample performance.
		
	\end{itemize}

	The rest of the paper is organized as follows. Section~\ref{sec:liter} gives the literature review of relevant studies on two-stage stochastic programming and metamodel-assisted simulation optimization. We formally state the problem of interest in Section~\ref{sec:statement}. We develop {a global-local metamodel-assisted two-stage OvS approach for complex stochastic systems} in Section~\ref{sec: metamodel}. We study the finite sample performance of our approach in Section~\ref{sec:example}, and conclude the paper in Section~\ref{sec:conclusion}.

	% ==================== Section 2 background ==================================  

	\section{Background}
	\label{sec:liter}
	To integrate strategic, tactical and operational decisions, the classical two-stage stochastic optimization was introduced \citep{Beale_1955, Dantzig_1955}. 	
	\begin{comment}
	It considers two decision time points: 
	a ``here and now" decision for the first-stage action before the realization of the uncertainty and a ``wait and see" decision for the second-stage action after the revelation of the uncertainty; see the detailed description in Shapiro and Philpott~\citep{Shapiro_Philpott_2007} and Shapiro et al.~\citep{shapiro_2009}. 
	The objective is to minimize the total expected cost, including the first-stage cost and the second-stage expected cost.
	\end{comment}
	Since the introduction, it has been applied to a wide range of areas, including electricity marketing \citep{Takriti_etal_1996, Ruiz_etal_2009, Tuohy_etal_2009} and the capacity expansion problem \citep{Dempster_etal_1981, Birge_Dempster_1996, Bailey_Jensen_Morton_1999}.
	In the classical two-stage stochastic optimization, the second-stage response function is assumed to be known, e.g., a linear function \citep{Beale_1955, Dantzig_1955} or a mixed integer one \citep{Takriti_etal_1996, Ruiz_etal_2009, Tuohy_etal_2009}. 
	Various algorithms have been proposed in the stochastic optimization community to solve it, such as two-stage linear program and two-stage mixed-integer program. To obtain the optimal solution, they typically exploit the structural information. For example, the L-shaped algorithm was proposed by Van Slyke and Wets~\citep{Slyke_Wets_1969}, and further developed in \cite{Louveaux_1980, Birge_Louveaux_1988, Frauendorfer_1988}. The Stochastic Decomposition (SD) algorithm was proposed in Higle and Sen~\citep{Sen_Higle_1991, Higle_Sen_1994}.

	However, for many real-world complex systems, e.g., semiconductor production systems and global bio-pharma supply chains, the second-stage response function is often unknown and the system outputs can be predicted through simulation. Therefore, the problem becomes a \textit{two-stage stochastic optimization via Simulation (OvS)}. In these situations, the classical algorithms cannot be easily extended.
	
	In the simulation community, OvS is an active research area. With recent technology advances, OvS offers a convenient way to support the decision making for a variety of complex stochastic systems and it has gained ever increasing importance \citep{Amaran_etal_2016, Nelson_2016}. 
	The existing OvS studies extensively focus on the one-stage OvS problems, including the ranking and selection \citep{Henderson_Nelson_2006}, the random search algorithm \citep{Zabinsky_2009}, the COMPASS algorithm \citep{Hong_Nelson_2006, Xu_Nelson_Hong_2010}, the simultaneous perturbation stochastic approximation (SPSA) algorithm \citep{Spall_1992}; see Fu~\citep{Fu_2014} for a comprehensive review.
	
	\textit{However, existing simulation-based optimization algorithms are typically developed for one-stage OvS problems in (\ref{eq: sto1}).} To the best of our knowledge, there is no rigorous algorithm proposed for two-stage stochastic OvS in the general situations. Different from the one-stage optimization, two-stage stochastic programming exhibits some unique features: the first-stage  optimization problem is stochastic and it depends on the results from second-stage optimization, while the second-stage optimization problems are conditional on the realizations of random events and they are deterministic.

	For the two-stage OvS, the estimation of expected cost is required. In many situations, the possible scenarios representing the second-stage uncertainty are too many or even infinite. Then, the sampling approach, SAA, uses a set of scenarios to approximate the expected future cost \citep{DANTZIG_GLYNN_1990}; see more discussion about SAA in \cite{Shapiro_Homem-de-Mello_1998, Shapiro_Homem-de-Mello_2000, Shapiro_Nemirovski_2005}.
	However, the SAA often introduces the finite sampling error. To accurately estimate the second-stage expected cost and control the impact of finite sampling error, a large scenario size is required \citep{shapiro_2009}. In addition, the expected objective in (\ref{eq: example}) relies on the second-stage optimal solutions. Without strong prior information on the second-stage response function, such as linearity, given a tight computational resource, there exists the optimality gap.

	For real-world complex stochastic systems, each simulation run could be computationally expensive. Thus, it is important to efficiently employ the tight simulation budget to search for promising first-stage decisions, while controlling the finite sampling error introduced by SAA and the optimality gap induced in the second-stage black-box optimization.
	In this paper, we propose a metamodel-assisted framework for two-stage optimization in (\ref{eq: example}). 
	When there is no strong prior information on the system response function, the Gaussian process metamodel is often used to provide a global prediction. In the past decades, it has received great attention in both deterministic and one-stage stochastic simulation optimization; see for example \cite{Jones_1998, Huang_etal_2006G, Quan_etal_2013,Sun_etal_2014}. Jones et al.~\citep{Jones_1998} introduced Kriging for deterministic OvS problems.
	Sun et al.~\citep{Sun_etal_2014} developed the GPS algorithm that employs Stochastic Kriging models introduced by Ankenman et al.~\citep{ankenman_nelson_staum_2010} for stochastic discrete optimization via simulation (DOvS). It can guarantee the global convergence, and also demonstrates good finite sample performances.
	
Notice that this paper is fundamentally different with our previous study in \citep{Xie_Yi_2016}. Here, we propose a simulation optimization approach that can efficiently \textit{solve} the {discrete} two-stage optimization for complex stochastic systems with unknown second-stage response function. The study in \citep{Xie_Yi_2016} focused on a \textit{generalized} two-stage dynamic decision model with nested risk measures, such as the conditional value-at-risk (CVaR). For a given decision policy, a metamodel-assisted approach was introduced to efficiently \textit{assess} the nested system risk, and further delivered a credible interval (CrI)  quantifying the simulation estimation error.  
%In this paper, we propose the aforementioned metamodel-assisted optimization approach to efficiently solve (\ref{eq: example}) given finite computational resources.

	% =============== Section 3 Problem Statement ========================

	\section{Problem Description}
	\label{sec:statement}
	
	In this section, we describe the problem of interest. To guide the dynamic decision making for complex stochastic systems, we consider the two-stage stochastic programming with unknown response estimated by simulation,  
	\begin{equation}
	\min_{\mathbf{x} \in \mathcal{X}}  \mbox{  }G(\mathbf{x})\equiv{c_0(\mathbf{x})}+\mbox{E}_{\pmb{\xi}}
	\left[
	\min_{\mathbf{y} \in \mathcal{Y}(\mathbf{x})} {q}(\mathbf{x}, \mathbf{y}, \pmb{\xi})\right] 
	\label{eq.intro1}  
	\end{equation} 
	where $\mathbf{x}$ represents the first-stage decision with the feasible set, denoted by $\mathcal{X}$, and $\mathbf{y}$ represents the second-stage decision with the feasible set $\mathcal{Y}(\mathbf{x})$ depending on $\mathbf{x}$.  
	Suppose that both $c_0(\mathbf{x})$ and ${q}(\mathbf{x}, \mathbf{y}, \pmb{\xi})$ are continuous, and the sets $\mathcal{X}$ and $\mathcal{Y}(\mathbf{x})$ are discrete and finite.
	The random input variate $\pmb{\xi}$ follows the probability model, denoted by $F(\pmb{\xi})$, characterizing the \textit{prediction uncertainty}. The first-stage decision 
	$\mathbf{x}$ is made prior to the realization of $\pmb{\xi}$ and the second-stage decision $\mathbf{y}$ is made after the uncertainty is revealed. That means given any feasible first-stage decision $\mathbf{x}$ and a scenario $\pmb{\xi}$, we need to solve a second-stage optimization problem, 
	\begin{equation}
	Q(\mathbf{x},\pmb{\xi})\equiv \min_{\mathbf{y} \in \mathcal{Y}(\mathbf{x})} {q}(\mathbf{x}, \mathbf{y}, \pmb{\xi}),
	\label{eq.secondStage}
	\end{equation} 
	which leads to the optimal decision, denoted by $\mathbf{y}^\star(\mathbf{x},\pmb{\xi})$. 
	For simplification, suppose that there is a unique optimal first-stage decision, denoted by $\mathbf{x}^\star=\arg\min_{\mathbf{x}\in \mathcal{X}} G(\mathbf{x})$, with the objective $G(\mathbf{x}^\star)$.

	For complex stochastic systems, such as global biopharma supply chains and semiconductor manufacturing, both cost functions $G(\mathbf{x})$ and $q_{\mathbf{x}_i}(\mathbf{y}, \pmb{\xi})\equiv {q}(\mathbf{x}_i, \mathbf{y}, \pmb{\xi})$ at any $\mathbf{x}_i\in \mathcal{X}$ are often unknown, which can be estimated by simulation. 
	Here, we use a semiconductor production facility investment as an illustrative example. We consider the planning horizon with $T$ time periods, say $T$ quarters. 
	Given the first-stage decision $\mathbf{x}_i$, i.e., the capacity of production facility investment, and the realization of demands $\pmb{\xi}=({\xi}_1,\xi_2, \ldots,{\xi}_T)$ in the planning horizon, we want to find the production scheduling decision  $\mathbf{y}=({y}_1,y_2,\ldots,{y}_T)$ minimizing the operational cost $q_{\mathbf{x}_i}(\mathbf{y},\pmb{\xi})=\sum_{t=1}^Tc_t(\mathbf{x}_i,{y}_t,
	{\xi}_t)$, where ${\xi}_t$ and $c_t(\cdot)$ for $t=1,2,\ldots,T$ represent the \textit{accumulated} demand and the cost occurring in the $t$-th time period. The optimal second-stage decision depends on the investment decision $\mathbf{x}_i$ and also the realized demands $\pmb{\xi}$. Since the semiconductor production could involve thousands of processing steps, each simulation run for $q_{\mathbf{x}_i}(\mathbf{y},\pmb{\xi})$ could be computationally expensive. 
	{Notice that the scenario aggregation and decision discretization could be used to reduce the complexity of second-stage optimization; see the similar strategies recommended for approximate look-ahead models in \cite{WarrenPowell_2014}.}

	For some $\mathbf{x}_i \in \mathcal{X}$, the second-stage program could be infeasible, i.e., $\mathcal{Y}(\mathbf{x}_i)$ is empty \cite{Shapiro_Philpott_2007}. To avoid this issue, we assume that the problem of interest has relatively complete recourse, i.e., for every $\mathbf{x}_i \in \mathcal{X}$ and every scenario $\pmb{\xi} \in \Xi$, the second-stage recourse problem is always feasible, where $\Xi$ denotes the scenario set.  
	
	If the scenario set $\Xi$ only contains $J$ scenarios $\{\pmb{\xi}_1, \ldots, \pmb{\xi}_J \}$ with known probabilities $\{ p_1, \ldots, p_J \}$, then the expectation in Objective~(\ref{eq.intro1}) becomes
	$
	\mbox{E}_{\pmb{\xi}}[Q(\mathbf{x}_i, \pmb{\xi})] =\sum_{j=1}^{J} p_j Q(\mathbf{x}_i, \pmb{\xi}_j),
	$
	where $J$ is a finite integer. In this case, for each first-stage decision $\mathbf{x}_i$, only $J$ second-stage optimization problems in ~(\ref{eq.secondStage}) need to be solved.
	
	However, in many situations, $\pmb{\xi}$ is continuous or the number of possible scenarios is astronomical. At any visited $\mathbf{x}_i\in \mathcal{X}$, the SAA with $N(\mathbf{x}_i)$ number of scenarios can be used to approximate the expected future cost in (\ref{eq.intro1}),
	\begin{equation}
	\bar{G}^c(\mathbf{x}_i)\equiv {c_0(\mathbf{x}_i)}+ \frac{1}{N(\mathbf{x}_i)} \sum_{j=1}^{N(\mathbf{x}_i)} {q_{\mathbf{x}_i}( \mathbf{y}^\star_{ij},\pmb{\xi}_{ij})}
	\label{eq.SAA}  
	\end{equation} 	
	with $\pmb{\xi}_{ij}\stackrel{i.i.d.}\sim F(\pmb{\xi})$ for $j=1,2,\ldots,N(\mathbf{x}_i)$; see the introduction of SAA in \cite{Ahmed_Shapiro_2002, Shapiro_Homem-de-Mello_2000, Shapiro_2008}. The superscript $c$ in (\ref{eq.SAA}) indicates that the objective is estimated based on the {``correct"} optimal second-stage decisions $\mathbf{y}^\star_{ij}\equiv\mathbf{y}^\star(\mathbf{x}_i,\pmb{\xi}_{ij})$.
	Differing with the existing two-stage optimization studies in the literature that assume the response $q_{\mathbf{x}_i}(\mathbf{y},\pmb{\xi})$ known \citep{Dempster_etal_1981, Birge_Dempster_1996, Bailey_Jensen_Morton_1999, Ahmed_Shapiro_2002},  for each first-stage decision $\mathbf{x}_i$ and scenario $\pmb{\xi}_{ij}$, the second-stage optimization in our study is \textit{a black-box scenario-based deterministic optimization problem}. Given finite computational assignment, there often exists an \textit{optimality gap} introduced by using the estimated optimal solution, denoted by $\widehat{\mathbf{y}}^\star_{ij}$,	\begin{equation}
	{
	\delta(\mathbf{x}_i,\pmb{\xi}_{ij})\equiv q_{\mathbf{x}_i}(\widehat{\mathbf{y}}^\star_{ij},\pmb{\xi}_{ij}) - q_{\mathbf{x}_i}(\mathbf{y}^\star_{ij},\pmb{\xi}_{ij}).} 
	\label{eq.optGap}   
	\end{equation}

	%SAA is built on the search results from $N(\mathbf{x})$ second-stage optimization problems. Since each simulation run could be expensive, given a tight budget, it is computationally prohibitive to solve many \textit{black-box second-stage optimization problems} in Equation~(\ref{eq.secondStage}). 

	In addition, the forecast uncertainty of $\pmb{\xi}$ occurring in the planning horizon could be large. For example, in the semiconductor and biopharma manufacturing, we frequently introduce new products and it is challenging to precisely predict their demands \citep{Liu_etal_2011, Chien_etal_2011, Ma_2011, Kaminsky_Wang_2015}. For the power grids with high renewable energy penetration, it is challenging to provide an accurate forecast of the wind and solar power generation when we make the unit commitment scheduling decision for the day-ahead market
	\citep{Takriti_etal_1996, Ruiz_etal_2009, Tuohy_etal_2009}. 
	Thus, it could require a large number of scenarios, $N(\mathbf{x}_i)$, to accurately estimate the expected future cost. It is computationally prohibitive to solve a large number of {black-box second-stage optimization problems}, $\min_{\mathbf{y}_{ij} \in \mathcal{Y}(\mathbf{x}_i)} {q}_{\mathbf{x}_i}(\mathbf{y}_{ij}, \pmb{\xi}_{ij})$ for $j=1,2,\ldots,N(\mathbf{x}_i)$.

	When we solve the two-stage optimization in (\ref{eq.intro1}) for complex stochastic systems via simulation, there are some important {observations} described as follows.    	     
	\begin{enumerate}
		\setcounter{enumi}{0}
		\item[] \textbf{Observation (1):} The {objective function values $G(\mathbf{x})$ and $G(\mathbf{x}^\prime)$ tend to be} similar when the first-stage decisions $\mathbf{x}$ and $\mathbf{x}^\prime$ are close to each other. 
		
		\item[] \textbf{Observation (2):} For each first-stage action $\mathbf{x}_i$, when we use SAA to approximate the expected cost,  there are $N(\mathbf{x}_i)$ second-stage optimization problems needed to be solved. Under situations with high prediction uncertainty, $N(\mathbf{x}_i)$ is required to be large so that we can accurately estimate the expected cost. Each scenario-based second-stage optimization is a black-box deterministic optimization problem and it is time-consuming to solve them separately. Also each simulation run could be computationally expensive.

		%\item[] \textbf{Observation (3):} At any first-stage action $\mathbf{x}_i$, similar scenarios $\pmb{\xi}_{ij_1}$ and $\pmb{\xi}_{ij_2}$ could lead to the similar second-stage optimal decisions $\mathbf{y}^\star_{ij_1}$ and $\mathbf{y}^\star_{ij_2}$. In addition, the costs $q_{\mathbf{x}_i}(\mathbf{y},\pmb{\xi})$ and $q_{\mathbf{x}_i}(\mathbf{y}^\prime,\pmb{\xi}^\prime)$ are similar when $(\mathbf{y},\pmb{\xi})$ and $(\mathbf{y}^\prime,\pmb{\xi}^\prime)$ are close to each other.
		
		\item[] \textbf{Observation (3):} {When we search for the optimal first-stage decision, we need to consider errors induced by: (a) the finite sampling error introduced by SAA, (b) the second-stage optimality gap, and (c) finite visited first-stage decisions. %This information is critical for developing an efficient search strategy that can balance exploitation and exploration.
		}
		
	\end{enumerate}

\vspace{-0.07in}
	
	% =============== Section 4 Kriging Model ========================
	
	\section{{Global-Local Metamodel-Assisted Two-Stage Optimization}}
	\label{sec: metamodel}
	
	{Considering the unique properties of stochastic programming with unknown response, in this section, we introduce a global-local metamodel-assisted two-stage optimization via simulation}. It 
	can efficiently employ the simulation budget, denoted by $C$, to solve the two-stage optimization problem in (\ref{eq.intro1}) for complex stochastic systems. 
	We first present the local metamodel assisted second-stage optimization in Section~\ref{sec:KrigingModel}. At each visited first-stage decision $\mathbf{x}_i\in\mathcal{X}$, we construct a GP metamodel for $q_{\mathbf{x}_i}(\mathbf{y},\pmb{\xi})$. {The} metamodel uncertainty is characterized by the GP posterior distribution, denoted by $\mathcal{GP}_{\mathbf{x}_i}$ and let $\widetilde{q}_{\mathbf{x}_i}(\cdot,\cdot)\sim \mathcal{GP}_{\mathbf{x}_i}$. In the paper, the notation $\widetilde{\cdot}$ denotes the posterior sample or random function/variable. 
	To overcome the challenges stated in Observation~(2), we utilize the local metamodel $\mathcal{GP}_{\mathbf{x}_i}$ to \textit{simultaneously} solve the $N(\mathbf{x}_i)$ second-stage optimization problems, $\min_{\mathbf{y}_{ij} \in \mathcal{Y}(\mathbf{x}_i)} {q}_{\mathbf{x}_i}( \mathbf{y}_{ij}, \pmb{\xi}_{ij})$ for $j=1,2,\ldots,N(\mathbf{x}_i)$. 
	%Let $\mathbf{y}^\star_{ij}$ and $\widehat{\mathbf{y}}_{ij}^\star$ denote the unknown true and estimated optimal second-stage decisions. 
	The local metamodel also allows us to estimate the  optimality gap from the second-stage optimization.
	%and the \textcolor{red}{estimation} uncertainty of $\delta(\mathbf{x}_i,\pmb{\xi}_{ij})$ comes from the remaining uncertainty of the local metamodel for $q_{\mathbf{x}_i}(\cdot,\cdot)$.

	Then, in Section~\ref{sec:Procedure}, based on the search results for the second-stage optimization, we develop a global GP metamodel for $G(\mathbf{x})={c_0(\mathbf{x})}+\mbox{E}_{\pmb{\xi}}
	\left[
	\min_{\mathbf{y} \in \mathcal{Y}(\mathbf{x})} {q}(\mathbf{x}, \mathbf{y}, \pmb{\xi})\right]$. {It accounts for: (1) finite sampling error induced by SAA, (2) the bias caused by the second-stage optimality gap, and (3) prediction error induced by only finite decision points visited.
	It can capture the spatial dependence of the response surface $G(\cdot)$ stated in Observation~(1). %Thus, this metamodel can balance the computational budget allocated for the first- and second-stage optimal search.
	}

	Assisted by the global-local metamodel developed in Sections~\ref{sec:KrigingModel} and \ref{sec:Procedure}, we propose the two-stage optimization {via simulation} in Section~\ref{subsec:OptimizationProcedure}  
	that can balance exploration and exploitation through simultaneously controlling all sources of errors.
	It can efficiently employ the computational resource to \textit{iteratively} solve for the optimal first- and second-stage decisions.
	As the simulation budget goes to infinity, in Section~\ref{subsec: prove},
	we can show that the proposed optimization procedure can guarantee the global convergence to $G(\mathbf{x}^\star)$.

	\subsection{Local Metamodel Assisted Second-Stage Optimization}
	\label{sec:KrigingModel}

	At any visited design point $\mathbf{x}_i$, in this section, we introduce a local metamodel assisted algorithm for solving the second-stage simulation optimization problems. Specifically, we first construct a GP metamodel for ${q}_{\mathbf{x}_i}(\cdot, \cdot)$ in Section~\ref{subsec:Kmetamodel}. Then, 
	{for each scenario $\pmb{\xi}_{ij}$, we plug it into the local metamodel and the posterior sample path $\widetilde{q}_{\mathbf{x}_i}(\mathbf{y}, \pmb{\xi}_{ij})$ is utilized to predict the response at any $\mathbf{y}\in\mathcal{Y}(\mathbf{x}_i)$. Thus, we solve
	the scenario-based second-stage optimization problems, $\min_{\mathbf{y}_{ij} \in \mathcal{Y}(\mathbf{x}_i)} {q}_{\mathbf{x}_i}(\mathbf{y}_{ij}, \pmb{\xi}_{ij})$  and further estimate the optimality gap $\delta(\mathbf{x}_i,\pmb{\xi}_{ij})$ for $j=1,2,\ldots,N(\mathbf{x}_i)$ in Section~\ref{subsec: secondstage}. Since the local metamodel $q_{\mathbf{x}_i}(\cdot,\cdot)$ can leverage the information collected from all scenarios at $\mathbf{x}_i$, it can efficiently 
	solve the second-stage optimization problems.}

	\subsubsection{Local Metamodel Construction}
	\label{subsec:Kmetamodel}
	
	For notational simplicity, let $\mathbf{z}\equiv (\mathbf{y},\pmb{\xi})$. The cost occurring in the planning horizon can be modeled as a realization of GP,
	\begin{equation}
	{q}_{\mathbf{x}_i}(\mathbf{z})=\beta_0+ M(\mathbf{z}).
	\label{eq.functionK}   \nonumber 
	\end{equation} 	
	It consists of two parts: a global trend $\beta_0$ 
	(note that $\beta_0$ can be replaced by a more general trend term $\mathbf{f}(\mathbf{z})^\top \pmb{\beta}$ without affecting our method)
	and a zero-mean GP, denoted by $M(\mathbf{z})$, modeling the spatial dependence of the response function $q_{\mathbf{x}_i}(\mathbf{z})$. 
	The covariance function of the GP is $\mbox{Cov}(M(\mathbf{z}),M(\mathbf{z}^\prime))=\sigma^2R(\mathbf{z},\mathbf{z}^\prime)
	$,
	where $\sigma^2$ is the variance and $R(\cdot)$ is the correlation function. {Our previous study \cite{xie_nelson_staum_2010} demonstrates that  the product-form Gaussian correlation function has the good performance and also easy to implement. %Similar observations are obtained in the other literature; see Sacks et al.~\citep{Sacks_Welch_Mitchell_Wynn_1989}, Hertog et al.~ \citep{Hertog_Kleijnen_Siem_2006}. 
	Thus, it is used} in the empirical study
	\begin{equation}
	R(\mathbf{z}-\mathbf{z^\prime} | \pmb{\phi})=\exp \left[-\sum_{j=1}^d \phi_j(z_j-z^\prime_j)^2 \right] 
	\label{eq.GaussianCorFun}
	\end{equation}
	where $d$ is the dimension of $\mathbf{z}$ and the parameters $\pmb{\phi}=(\phi_1,\ldots,\phi_d)$ control the spatial dependence. 
	%The smaller $\phi_j$ is, the greater influence impacted by the distance in $jth$-dimension.
	
	%\begin{sloppypar}
	
	Denote $K_1$ design points at $\mathbf{x}_i$ by $\mathcal{P}_{\mathbf{x}_i} \equiv\{\bm z_1,\bm z_2,\ldots,\bm z_{K_1}\}$, and the corresponding simulation outputs by 
	$
	\mathbf{Q}_{\mathcal{P}_{\mathbf{x}_i}}=(q_{\mathbf{x}_i}(\bm z_1), q_{\mathbf{x}_i}(\bm z_2),\ldots,q_{\mathbf{x}_i}(\bm z_{K_1}))^\prime. 
	$        	
	Define $\bm Z$ as the matrix of $K_p$ prediction points.    
	Let $R(\bm Z,\cdot)$ represent the $K_1 \times K_p$ spatial correlation matrix between $K_1$ design points and $K_p$ prediction points $\bm Z$. Let $R$
	be the $K_1 \times K_1$ correlation matrix across all $K_1$ design points and let $R(\bm Z,\bm Z)$ represent the $K_p \times K_p$ spatial correlation matrix between $K_p$ prediction points .
	Then, given the simulation outputs $\mathbf{Q}_{\mathcal{P}_{\mathbf{x}_i}}$, the remaining uncertainty of the local metamodel at $\bm Z$ is characterized by the updated Gaussian process, denoted by $\mathcal{GP}_{\mathbf{x}_i}(\widehat{q}_{\mathbf{x}_i}(\bm Z), s^2_{\mathbf{x}_i}(\bm Z))$, with mean
	\begin{equation}
	\widehat{q}_{\mathbf{x}_i}(\bm Z)=\widehat{\beta}_0 \cdot  \mathbf{1}_{K_p \times 1} +R(\bm Z,\cdot)^\prime R^{-1}(\mathbf{Q}_{\mathcal{P}_{\mathbf{x}_i}}-\mathbf{1}_{K_1 \times 1}\widehat{\beta}_0),
	\label{eq.KBLUP} 
	\end{equation}
	and variance-covariance matrix	
	\begin{eqnarray}
	\label{eq.undervar} 
	\lefteqn{
		s^2_{\mathbf{x}_i}(\bm Z ,\bm Z ) = \sigma^2  \{ R(\bm Z,\bm Z)-R(\bm Z,\cdot)^\prime R^{-1} R(\bm Z,\cdot) } \\
	&& + [\mathbf{1}_{1\times K_p}-\mathbf{1}_{K_1 \times 1}^\prime R^{-1}R(\bm Z,\cdot)]^\prime [\mathbf{1}_{K_1\times 1}^\prime R^{-1}
	\mathbf{1}_{K_1 \times 1}]^{-1} 
	[\mathbf{1}_{1\times K_p}-\mathbf{1}_{K_1\times 1}^\prime R^{-1} R(\bm Z,\cdot) ] \},
	\nonumber
	\end{eqnarray}	
	where  $\widehat{\beta}_0=(\mathbf{1}^\prime_{K_1 \times 1} R^{-1} \mathbf{1}_{K_1 \times 1} )^{-1}  \mathbf{1}^\prime_{K_1 \times 1}  R^{-1}\mathbf{Q}_{\mathcal{P}_{\mathbf{x}_i}}$. 
	The unknown parameters $\sigma$ and $\pmb{\phi}$ are estimated by the maximum likelihood approach; see Sacks et al.~\citep{Sacks_Welch_Mitchell_Wynn_1989} and Jones  et al.~\citep{Jones_1998}. % denoted by $\widehat{\sigma}$ and  $\widehat{\pmb{\phi}}$, respectively . 

	\subsubsection{Second-Stage Optimization and Optimality Gap Estimation}
	\label{subsec: secondstage}

	\begin{sloppypar}
		At any visited first-stage candidate $\mathbf{x}_i\in\mathcal{X}$, when SAA is used to estimate the expected future cost, it could be computationally expensive to solve $N(\mathbf{x}_i)$ black-box second-stage optimization problems separately. 
		{Assisted by the local metamodel built for $q_{\mathbf{x}_i}(\cdot,\cdot)$ in Section~\ref{subsec:Kmetamodel}, we can simultaneously solve \textit{all} $N(\mathbf{x}_i)$ second-stage optimization problems:  $\mathbf{y}_{ij}^\star = \arg\min_{\mathbf{y} \in \mathcal{Y}(\mathbf{x}_i)} {q}_{\mathbf{x}_i}(\mathbf{y}, \pmb{\xi}_{ij})$ with $j=1,2,\ldots,N(\mathbf{x}_i)$.}
		Specifically, by plugging in each scenario $\pmb{\xi}_{ij}$, %\textcolor{red}{into $\mathcal{GP}_{\mathbf{x}_i}(\widehat{q}(\mathbf{y},\pmb{\xi}),s^2_{\mathbf{x}_i}(\mathbf{y},\pmb{\xi}))$}
		the metamodel 
	$\mathcal{GP}_{\mathbf{x}_i}(\widehat{q}_{\mathbf{x}_i}(\cdot,\pmb{\xi}_{ij}),s_{\mathbf{x}_i}^2(\cdot,\pmb{\xi}_{ij}))$ provides the posterior prediction of the system response for any untried decision $\mathbf{y}\in \mathcal{Y}(\mathbf{x}_i)$, which is used to guide the search for $\mathbf{y}_{ij}^\star$. 
	Given the current estimated optimal solution, denoted by $\widehat{\mathbf{y}}^\star_{ij}$, we want to efficiently employ the simulation resource to reduce {the optimality gap}, 
	$
	\delta(\mathbf{x}_i,\pmb{\xi}_{ij})= q_{\mathbf{x}_i}(\widehat{\mathbf{y}}^\star_{ij},\pmb{\xi}_{ij})
	-q_{\mathbf{x}_i}(\mathbf{y}^\star_{ij},\pmb{\xi}_{ij}).
	\label{eq.delta1}
$
	Thus, in the next search iteration, we want to find the promising point $\mathbf{y}_{ij}\in \mathcal{Y}(\mathbf{x}_i)$ that can reduce the optimal gap the most and run simulation there.
		\end{sloppypar}

	Since the response surface $q_{\mathbf{x}_i}(\cdot,\pmb{\xi}_{ij})$ is unknown, motivated by the Efficient Global Optimization (EGO) algorithm \citep{Jones_1998}, the criteria used to guide the sequential search for the optimal solution $\mathbf{y}_{ij}^\star$ is based on minimizing the expected optimality gap or maximizing the \textit{expected improvement (EI)}. 
	{The unknown second-stage response surface is modeled by the posterior sample path $\widetilde{q}_{\mathbf{x}_i}(\cdot,\cdot)\sim
		\mathcal{GP}_{\mathbf{x}_i}(\widehat{q}_{\mathbf{x}_i}(\cdot,\cdot),s_{\mathbf{x}_i}^2(\cdot,\cdot))$. 
		%In this paper, we use $\widetilde{\cdot}$ to denote any posterior sample. 
	Following \citep{Jones_1998}, given the current optimal solution $\widehat{\mathbf{y}}^\star_{ij}$, the EI at any untried point $\mathbf{y}_{ij}\in \mathcal{Y}(\mathbf{x}_i)$ can be defined as
	\begin{eqnarray} 
	\lefteqn{ \mbox{E}_{\widetilde{q}_{\mathbf{x}_i}(\mathbf{y}_{ij},\pmb{\xi}_{ij})}\left[\mathrm{I}(\mathbf{y}_{ij}, \pmb{\xi}_{ij})\right] \equiv	\mathrm{E}_{\widetilde{q}_{\mathbf{x}_i}(\mathbf{y}_{ij},\pmb{\xi}_{ij})}\left[ \max \left(
{q}_{\mathbf{x}_i}(\widehat{\mathbf{y}}^\star_{ij},\pmb{\xi}_{ij})-\widetilde{q}_{\mathbf{x}_i}(\mathbf{y}_{ij},\pmb{\xi}_{ij})
	,0\right)\right] } \nonumber \\
&=&
\Delta \cdot \Phi\left(\frac{\Delta}{s_{\mathbf{x}_i}(\mathbf{y}_{ij},\pmb{\xi}_{ij})} \right)+ 
s_{\mathbf{x}_i}(\mathbf{y}_{ij},\pmb{\xi}_{ij})\cdot \varphi\left(\frac{\Delta}{s_{\mathbf{x}_i}(\mathbf{y}_{ij},\pmb{\xi}_{ij})}\right) 
		    	\label{eq. formEI}
		\end{eqnarray} 
	where $\mbox{E}_{\widetilde{q}_{\mathbf{x}_i}(\mathbf{y}_{ij},\pmb{\xi}_{ij})}\left[ \cdot \right]$ denotes the expectation over the remaining metamodel uncertainty at $(\mathbf{y}_{ij},\pmb{\xi}_{ij})$ with
	$\widetilde{q}_{\mathbf{x}_i}(\mathbf{y}_{ij},\pmb{\xi}_{ij})\sim
		N(\widehat{q}_{\mathbf{x}_i}(\mathbf{y}_{ij},\pmb{\xi}_{ij}),s_{\mathbf{x}_i}^2(\mathbf{y}_{ij},\pmb{\xi}_{ij}))$. Here, we set $\Delta\equiv q_{\mathbf{x}_i}(\widehat{\mathbf{y}}^\star_{ij},\pmb{\xi}_{ij})-\widehat{q}_{\mathbf{x}_i}(\mathbf{y}_{ij}, \pmb{\xi}_{ij})$,
	 and represent the PDF, CDF of standard normal distribution with $\varphi(\cdot)$, $\Phi(\cdot)$ respectively.} The greater $\mbox{E}_{\widetilde{q}_{\mathbf{x}_i}(\mathbf{y}_{ij},\pmb{\xi}_{ij})}\left[\mathrm{I}(\mathbf{y}_{ij}, \pmb{\xi}_{ij})\right]$ is, the more promising the untried point $\mathbf{y}_{ij}$ {could be}. Hence, to efficiently search for the true optimal $\mathbf{y}_{ij}^\star$, we find the point which gives the maximum EI, 
	\begin{equation}
	\mathbf{y}_{ij}^{EI}={\arg
	\max}_{\mathbf{y}_{ij} \in \mathcal{Y}(\mathbf{x}_i )} \mbox{E}_{\widetilde{q}_{\mathbf{x}_i}(\mathbf{y}_{ij},\pmb{\xi}_{ij})}\left[\mathrm{I}(\mathbf{y}_{ij}, \pmb{\xi}_{ij})\right]
	\label{eq: EIsearch}
	\end{equation}
		and then run simulation there; see the  second-stage optimization search procedure in Section~\ref{subsec:OptimizationProcedure}.

	{In addition,
	the local metamodel can be used to esimate the second-stage optimality gap,
	$
	\delta(\mathbf{x}_i,\pmb{\xi}_{ij}) = q_{\mathbf{x}_i}(\widehat{\mathbf{y}}^\star_{ij}, \pmb{\xi}_{ij})-q_{\mathbf{x}_i}(\mathbf{y}^\star_{ij}, \pmb{\xi}_{ij})
	$ for $j=1,2,\ldots,N(\mathbf{x}_i)$.
	%This information could improve the construction of global metamodel so that we can balance the computational resource allocated to first- and second-stage optimization.
	Specifically, the unknown response surface $q_{\mathbf{x}_i}(\cdot,\cdot)$ is modeled by a sample path of Gaussian Process $\mathcal{GP}_{{\mathbf{x}_i}}$. For each scenario $\pmb{\xi}_{ij}$, the local metamodel $\mathcal{GP}_{{\mathbf{x}_i}}$ can be used to estimate the expected optimality gap $\mbox{E}[\delta(\mathbf{x}_i,\pmb{\xi}_{ij})|\pmb{\xi}_{ij}]$, where the conditional expectation is over the metmodel uncertainty. \textit{Basically, as the metamodel uncertainty increases especially at the promising area, the expected optimality gap also increases.}
To estimate $\mbox{E}[\delta(\mathbf{x}_i,\pmb{\xi}_{ij})|\pmb{\xi}_{ij}]$, we first generate a posterior sample path $\widetilde{q}_{\mathbf{x}_i}^{(b)}(\cdot,\cdot)\sim \mathcal{GP}_{\mathbf{x}_i}(\widehat{q}_{\mathbf{x}_i}(\cdot, \cdot),s_{\mathbf{x}_i}^2(\cdot,\cdot))$
for $b=1,\ldots,B$, where $\widehat{q}_{\mathbf{x}_i}$ and $s_{\mathbf{x}_i}^2$ are specified by (\ref{eq.KBLUP}) and (\ref{eq.undervar}). Then, we can use $\widetilde{q}_{\mathbf{x}_i}^{(b)}(\mathbf{y},\pmb{\xi}_{ij})$ to predict the response for any $\mathbf{y}\in \mathcal{Y}(\mathbf{x}_i)$, and calculate the $b$-th realization of the optimality gap,
	\[
	\delta^{(b)}(\mathbf{x}_i,\pmb{\xi}_{ij}) = \max\left(0, {q}_{\mathbf{x}_i}(\widehat{\mathbf{y}}^\star_{ij}, \pmb{\xi}_{ij})- \widehat{q}^\star_{ijb} \right) 
	\mbox{ with } \widehat{q}^\star_{ijb} = \min_{\mathbf{y}\in \mathcal{Y}(\mathbf{x}_i)} \widetilde{q}^{(b)}_{\mathbf{x}_i}(\mathbf{y},\pmb{\xi}_{ij}).
	\] 
	Thus, we can estimate $\mbox{E} [\delta(\mathbf{x}_i,\pmb{\xi}_{ij}) | \pmb{\xi}_{ij}]$ with % and $\mbox{Var}[\delta(\mathbf{x}_i,\pmb{\xi}_{ij})| \pmb{\xi}_{ij}]$ by the sample mean and variance,
	\begin{equation}
	\widehat{\mbox{E}} [\delta(\mathbf{x}_i,\pmb{\xi}_{ij}) | \pmb{\xi}_{ij}]= \frac{1}{B} 
	\sum_{b=1}^{B} \delta^{(b)}(\mathbf{x}_i,\pmb{\xi}_{ij})
	\label{eq: meanog}.
	\end{equation}
	}

	% ============================= Section 5 ===================================
	
	\subsection{Global Metamodel Development}
	\label{sec:Procedure}

	{
	Based on Observation~(1) described in Section~\ref{sec:statement}, suppose that the unknown response surface $G(\mathbf{x})={c_0(\mathbf{x})}+\mbox{E}_{\pmb{\xi}}
	\left[
	\min_{\mathbf{y} \in \mathcal{Y}(\mathbf{x})} {q}_{\mathbf{x}}(\mathbf{y}, \pmb{\xi})\right] 
	$ is a realization of GP. Given the results from the second-stage optimization, we develop a global GP metamodel for $G(\mathbf{x})$, which will be used to guide the search for the optimal decision $\mathbf{x}^\star$ in Section~\ref{subsec:OptimizationProcedure}.}

	\begin{sloppypar}
	For any $\mathbf{x}_i\in \mathcal{X}$, the optimal response from the $j$-th scenario  
	${G}_j(\mathbf{x}_i) \equiv c_0(\mathbf{x}_i)+\min_{\mathbf{y}_{ij} \in \mathcal{Y}(\mathbf{x}_i)} {q}_{\mathbf{x}_i}(\mathbf{y}_{ij}, \pmb{\xi}_{ij}) $ 
	can be written as,
	\begin{equation}
	{G}_j(\mathbf{x}_i)=G(\mathbf{x})+\epsilon(\mathbf{x}_i,\pmb{\xi}_{ij}),
	\label{eq: skdef} 
	\end{equation}
 where the mean zero random error $\epsilon(\mathbf{x}_i,\pmb{\xi}_{ij})$ represents the deviation of $G_j(\mathbf{x}_i)$ from the expected cost $G(\mathbf{x}_i)=\mbox{E}[G_j(\mathbf{x}_i)]$.
However, the optimal solution for the second-stage optimization, 
	$\mathbf{y}^{\star}_{ij}=\arg\min_{\mathbf{y}_{ij}\in \mathcal{Y}(\mathbf{x}_i)} q_{\mathbf{x}_i}(\mathbf{y}_{ij},\pmb{\xi}_{ij})$, is unknown. Since only the current optimal estimate $\widehat{\mathbf{y}}^\star_{ij}$ is available, we can observe $\breve{G}_j(\mathbf{x}_i)\equiv c_0(\mathbf{x}_i)+q_{\mathbf{x}_i}(\widehat{\mathbf{y}}^\star_{ij},\pmb{\xi}_{ij})$. Thus, we rewrite the optimal response from the $j$-th scenario as $
	{G}_j(\mathbf{x}_i) = \breve{G}_j(\mathbf{x}_i) - \delta(\mathbf{x}_i,\pmb{\xi}_{ij})
	$
	with the optimality gap, $\delta(\mathbf{x}_i,\pmb{\xi}_{ij})=q_{\mathbf{x}_i}(\widehat{\mathbf{y}}^\star_{ij},\pmb{\xi}_{ij})
	-q_{\mathbf{x}_i}(\mathbf{y}^\star_{ij},\pmb{\xi}_{ij})$.
	Combining with (\ref{eq: skdef}), we have
	\begin{equation}
	\breve{G}_j(\mathbf{x}_i)-\delta(\mathbf{x}_i,\pmb{\xi}_{ij})=G(\mathbf{x}_i)+\epsilon(\mathbf{x}_i,\pmb{\xi}_{ij}).
	\label{eq: unbiased}
	\end{equation}
	As noted in Section~\ref{subsec: secondstage}, given any scenario $\pmb{\xi}_{ij}$, since $q_{\mathbf{x}_i}(\cdot,\cdot)$ has the remaining metamodel uncertainty, our belief on unknown optimality gap $\delta(\mathbf{x}_i,\pmb{\xi}_{ij})$ is treated as a random variable with conditional expectation $\mbox{E} [\delta(\mathbf{x}_i,\pmb{\xi}_{ij})|\pmb{\xi}_{ij}]$. 
	Thus, $\delta(\mathbf{x}_i,\pmb{\xi}_{ij})$ can be expressed as 
	\begin{equation}
	\delta(\mathbf{x}_i,\pmb{\xi}_{ij})=\mbox{E} [\delta(\mathbf{x}_i,\pmb{\xi}_{ij})|\pmb{\xi}_{ij}]+\epsilon^\prime(\mathbf{x}_i, \pmb{\xi}_{ij}),
	\label{eq: defop} 
	\end{equation}
	where $\epsilon^\prime(\mathbf{x}_i, \pmb{\xi}_{ij})$ is defined as a zero-mean random variable characterizing the variability of $\delta(\mathbf{x}_i,\pmb{\xi}_{ij})$ given $\pmb{\xi}_{ij}$ and it is induced by the local metamodel uncertainty of $q_{\mathbf{x}_i}(\cdot,\pmb{\xi}_{ij})$. 
	By plugging (\ref{eq: defop}) into (\ref{eq: unbiased}) and then rearranging terms, we obtain 
	\begin{equation}
	\breve{G}_j(\mathbf{x}_i)-\mbox{E}[\delta(\mathbf{x}_i,\pmb{\xi}_{ij})|\pmb{\xi}_{ij}]=G(\mathbf{x}_i)+\epsilon(\mathbf{x}_i,\pmb{\xi}_{ij})+
	\epsilon^\prime (\mathbf{x}_i, \pmb{\xi}_{ij}).
	\label{eq: skBias} 
	\end{equation}
	\end{sloppypar}
	
	\begin{sloppypar}
	{
	At each $\mathbf{x}_i$, given $N(\mathbf{x}_i)$ scenario results from solving the second-stage optimization problems, we obtain the expected optimality gap adjusted SAA estimate for $G(\mathbf{x}_i)$,
	\begin{eqnarray}
	\label{eq.Gbar}
	\lefteqn{
		\bar{G}(\mathbf{x}_i)= \frac{1}{N(\mathbf{x}_i)} \sum_{j=1}^{N(\mathbf{x}_i)} \{ \breve{G}_j(\mathbf{x})-\mbox{E}[\delta(\mathbf{x}_i,\pmb{\xi}_{ij})|\pmb{\xi}_{ij}] \}
	} \nonumber \\
	&= c_0(\mathbf{x}_i)+\frac{1}{N(\mathbf{x}_i)}
	\sum_{j=1}^{N(\mathbf{x}_i)} \left\{
	q(\mathbf{x}_i,\widehat{\mathbf{y}}^\star_{ij},\pmb{\xi}_{ij}) -  \mbox{E}[\delta(\mathbf{x}_i,\pmb{\xi}_{ij})|\pmb{\xi}_{ij}] \right\}
	\label{eq.mid101}
	\end{eqnarray}
	with $\pmb{\xi}_{ij}\stackrel{i.i.d.}\sim F(\pmb{\xi})$ for $i=1,2,\ldots,K$ and $j=1,2,\ldots,N(\mathbf{x}_i)$.
	Let $\bar{e}(\mathbf{x}_i,\mathcal{D}_i)=\frac{1}{N(\mathbf{x}_i)}\sum_{j=1}^{N(\mathbf{x}_i)}[\epsilon(\mathbf{x},\pmb{\xi}_{ij})+\epsilon^\prime(\mathbf{x}_i,\pmb{\xi}_{ij})]$, where $\mathcal{D}_i=\{\pmb{\xi}_{ij}: j=1,\ldots,N(\mathbf{x}_i) \}$ represents the finite scenarios set at $\mathbf{x}_i$.
	%For the unknown conditional expected optimality gap, $\mbox{E}[\delta(\mathbf{x}_i,\pmb{\xi}_{ij})|\pmb{\xi}_{ij}]$, we can plug in the estimate obtained from (\ref{eq: meanog}).
	Then, from Equation~(\ref{eq: skBias}), we get
	$
    \bar{G}(\mathbf{x}_i) = G(\mathbf{x}_i) + 
    \bar{e}(\mathbf{x}_i,\mathcal{D}_i).$
    }

	{Suppose the unknown response surface $G(\mathbf{x})$ is a random sample path of a GP, denoted by $\mathcal{W}(\mathbf{x})=\mu_0+W(\mathbf{x})$, characterizing the uncertainty of our belief on the underlying $G(\mathbf{x})$.
	Thus, at any $\mathbf{x}_i\in\mathcal{X}$, we can model the summary simulation output $\bar{G}(\mathbf{x}_i)$ with
	\begin{equation}
	    \bar{G}(\mathbf{x}_i) = [\mu_0+W(\mathbf{x}_i)] + \bar{e}(\mathbf{x}_i,\mathcal{D}_i).
	    \label{eq.barG_GP}
	\end{equation}
Since the mean zero error $\bar{e}(\mathbf{x}_i,\mathcal{D}_i)$ considers the aggregated impact from many factors or results from the $N(\mathbf{x}_i)$ scenarios, we can assume that it follows the normal distribution by applying the general central limit theory (CLT). Note that $\mu_0$ in (\ref{eq.barG_GP}) can be replaced by a more general trend term, i.e., $\mathbf{f}(\mathbf{x}_i)^\top \pmb{\mu}$. 
	}
	
   {
	Denote the design points with
	$\mathcal{P}_o\equiv\{\mathbf{x}_1,\mathbf{x}_2,\ldots,\mathbf{x}_{K}\}$.
	Let 
	$\bar{\mathbf{G}}_{\mathcal{P}_o}=(\bar{G}(\mathbf{x}_1),\bar{G}(\mathbf{x}_2),\ldots,\bar{G}(\mathbf{x}_{K}))^\prime$.
    The bootstrap can be used to quantify the estimation variance of $\bar{G}(\mathbf{x}_i)$ for $i=1,\ldots,K$. Specifically, in each $t$-th iteration, we draw with replacement
$N(\mathbf{x}_i)$ outputs from  $\{\breve{G}_j(\mathbf{x})-\mbox{E}[\delta(\mathbf{x}_i,\pmb{\xi}_{ij})|\pmb{\xi}_{ij}] \mbox{  with }j=1,\ldots, N(\mathbf{x}_i)\}$ and calculate the sample mean, denoted by $\bar{G}^{(t)}(\mathbf{x}_i)$. Repeat this procedure for $T$ times. 
We can estimate the variance of $\bar{G}(\mathbf{x}_i)$, denoted by $V_{ii}$, by using the sample variance from these $T$ bootstrap samples $\bar{G}^{(1)},\ldots,\bar{G}^{(T)}.$ 
For unknown $\mbox{E}[\delta(\mathbf{x}_i,\pmb{\xi}_{ij})|\pmb{\xi}_{ij}]$, we plug in the estimate $\widehat{\mbox{E}}[\delta(\mathbf{x}_i,\pmb{\xi}_{ij})|\pmb{\xi}_{ij}]$ obtained from (\ref{eq: meanog}).
Then, without CRN, the covariance matrix of $\bar{\mathbf{G}}_{\mathcal{P}_o}$ is a diagonal matrix $V$ with the $i$-th diagonal term equal to $V_{ii}$. Given $\bar{\mathbf{G}}_{\mathcal{P}_o}$ and $V$, we can construct the global metmaodel to guide the search for $\mathbf{x}^\star$.
}

	Represent the spatial covariance function as $\mbox{Cov}(W(\mathbf{x}),W(\mathbf{x}^\prime))=\tau^2r(\mathbf{x},\mathbf{x}^\prime)
	$, where $\tau^2$ denotes the variance and 
	$r(\mathbf{x}-\mathbf{x^\prime}; \pmb{\phi})$ denotes the correlation function with parameters $\pmb{\phi}$. The Gaussian correlation function is used in our empirical study.
	For any prediction point $\bm X_\star\in\mathcal{X}$, denote the spatial covariance vector between $\bm X_\star$ and design points by 
	$\Sigma(\bm X_\star,\cdot)$, and denote the variance-covariance matrix between design points by $\Sigma$. 
	Given the simulation outputs, the remaining uncertainty of the cost function  $G(\bm X_\star)$ can be characterized by a GP, denoted by $\widetilde{G}(\bm X_\star) \sim \mathcal{GP}_o(\widehat{G}(\bm X_\star), s^2(\bm X_\star ))$, with mean 
	\begin{equation}
	\widehat{G}(\bm X_\star)=\widehat{\mu}_0
	+ \Sigma(\bm X_\star,\cdot)^\prime [\Sigma+V]^{-1}
	(\bar{\mathbf{G}}_{\mathcal{P}_o}-\widehat{\mu}_0\cdot \mathbf{1})
	\label{eq.predictor}
	\end{equation}
	and variance
	\begin{equation}
	\begin{aligned}
	s^2(\bm X_\star ) = \tau^2 - \Sigma(\bm X_\star,\cdot)^\prime [\Sigma+V]^{-1} \Sigma(\bm X_\star,\cdot)
	+\mathbf{\eta}^\prime [\mathbf{1}^\prime(\Sigma+V)^{-1}
	\mathbf{1}]^{-1}\mathbf{\eta}  
	\label{eq.MSE}
	\end{aligned}
	\end{equation}
	where $\widehat{\mu}_0=[\mathbf{1}^\prime (\Sigma+V)^{-1}
	\mathbf{1}]^{-1}\mathbf{1}^\prime (\Sigma+V)^{-1}\bar{\mathbf{G}}_{\mathcal{P}_o}$ and $\mathbf{\eta}=1-\mathbf{1}^\prime
	(\Sigma+V)^{-1} \Sigma(\bm X_\star,\cdot)$; see the detailed information about the stochastic kriging in Ankenman et al.~\citep{ankenman_nelson_staum_2010}. The maximum likelihood estimates (MLEs) for $\tau^2$ and $\pmb{\phi}$ are used for prediction.
	\end{sloppypar}

	% =======================================================================
	
	\subsection{Global-Local Metamodel Assisted Two-Stage Optimization via Simulation}
	\label{subsec:OptimizationProcedure}

	{Built on the global-local metamodel developed in Sections~\ref{sec:KrigingModel} and \ref{sec:Procedure}, we propose a two-stage optimization via simulation. Notice that before using either local or global metamodel for the optimal search, we need to verify the GP assumption; See Jones et al.~\citep{Jones_1998} and Huang et al.~\citep{Huang_etal_2006G} for the detailed discussion. The procedure of proposed global-local metamodel assisted two-stage OvS includes the main steps described in Algorithm~\ref{Global}.} In Step~(1), we specify the simulation budget $C$, the initial relative EI threshold $\alpha_0$, the growth factor $g$, the number of replications $n_0$, and set the iteration index $k=1$. Denote $D_{\mathbf{x}}^{(k)}$ as the set of first-stage design points visited at the $k$-th iteration and denote $\mathcal{S}_{\mathbf{x}}^{(k)}$ as all design points that have been visited until the $k$-th iteration,  
	$
	\mathcal{S}_{\mathbf{x}}^{(k)} =\bigcup_{s=1}^{k} D_{\mathbf{x}}^{(s)}.
	$
	We generate the initial set of design points for the global GP metamodel $D_{\mathbf{x}}^{(k)}$ with $k=1$ by using the maximin Latin Hypercube Design (LHD)  \citep{Jones_1998, Huang_etal_2006G}.  In the empirical study, we use the ``$10d$" rule for the number of initial design points \citep{Huang_etal_2006G}.
	For each $\mathbf{x}_i \in D_{\mathbf{x}}^{(k)}$, we allocate $n_0$ scenarios. Following Kim and Nelson~\citep{Kim_Nelson_2007} and Tsai et al.~\citep{Tsai_Barry_Staum_2009}, we set $n_0=10$ in the empirical study. 
	\begin{algorithm}
		\textbf{Step~(1):} Initialization.  
		%\begin{itemize}
		
		%\item[]
		Step (1.1) Specify the simulation budget $C$ and the growth factor $g$. Initialize the iteration index $k=1$, the EI threshold $\alpha_0$ and the number of replications $n_0$. %, and the first-stage design point set $\mathcal{S}_{\mathbf{x}}^{(0)}= \emptyset$.
		
		%	\item[] 
		Step (1.2) Use the maximin LHD to generate the {initial }first-stage design point set $D_{\mathbf{x}}^{(k)}$ evenly covering the decision space $\mathcal{X}$ and set $\mathcal{S}_{\mathbf{x}}^{(k)}=D_{\mathbf{x}}^{(k)}$. Let $N^{(k)}(\mathbf{x}_i)=n_0$ for $\mathbf{x}_i \in D_{\mathbf{x}}^{(k)}$.
		%\end{itemize}
		
		\vspace{0.1in}
		\textbf{{Step~(2):}} Solve the second-stage optimization {problems }for design points {visited at the $k$-th iteration,} $\mathbf{x}_i \in {D_{\mathbf{x}}^{(k)}}$. {Reset the 
		collection} of finished second-stage scenarios, $\mathcal{U}_{\mathbf{x}_i} = \emptyset$. 
		
		%	\begin{itemize}
		%	\item[] 
		Step~(2.1): Update the relative EI threshold $\alpha^{(k)}(\mathbf{x}_i)$ and the number of scenarios $N^{(k)}(\mathbf{x}_i)$ for $\mathbf{x}_i\in D_{\mathbf{x}}^{(k)}$. %, where $\mathcal{S}_{\mathbf{x}}^{(k)}=\mathcal{S}_{\mathbf{x}}^{(k-1)}\bigcup D_{\mathbf{x}}^{(k)}$.
		
		\eIf{ $\mathbf{x}_i \notin  \mathcal{S}_{\mathbf{x}}^{(k-1)}$ {(first visited design points)}} {
			
			Set $\alpha^{(k)}(\mathbf{x}_i)=\alpha_0$ {and $N^{(k)}(\mathbf{x}_i)=n_0$. Generate new} scenarios $\{\pmb{\xi}_{ij}\}_{j=1}^{N(\mathbf{x}_i)}$. %Following \cite{Jones_1998} and \cite{Huang_etal_2006G}, we 
			Use LHD to generate the second-stage design point set $\mathcal{P}_{\mathbf{x}_i}$ evenly covering the space of $(\mathbf{y},\pmb{\xi})$, denoted by
			$\mathcal{Y}(\mathbf{x}_i)\times \Xi$,
			and run simulations there. Construct the initial local metamodel {$\mathcal{GP}_{\mathbf{x}_i}(\widehat{q}_{\mathbf{x}_i}(\cdot,\cdot),s_{\mathbf{x}_i}^2(\cdot,\cdot))$} by using (\ref{eq.KBLUP}) and (\ref{eq.undervar}).
			%$\mathcal{GP}_I(\widehat{q}_{\mathbf{x}_i}(\cdot,\cdot),s_{\mathbf{x}_i}^2(\cdot,\cdot))$ 
		}
		{
			
			Set $\alpha^{(k)}(\mathbf{x}_i)=\alpha^{(k-1)}(\mathbf{x}_i)\cdot g^{-\frac{1}{2}}$. Generate  $N^{(k-1)}(\mathbf{x}_i)(g-1)$ new scenarios and let $N^{(k)}(\mathbf{x}_i)=N^{(k-1)}(\mathbf{x}_i)\cdot g$. 
		}

		%\item[]
		Step~(2.2): {For each new generated scenario $\pmb{\xi}_{ij}$, use the local metamodel  $\mathcal{GP}_{\mathbf{x}_i}$ to search  $\widehat{\mathbf{y}}_{ij}^\star$ by using (\ref{eq: initialopt}).} 
		%Add $(\widehat{\mathbf{y}}^\star_{ij}, \pmb{\xi}_{ij})$ to $\mathcal{P}_{\mathbf{x}_i}$, run a simulation, and update the local metamodel for ${q}_{\mathbf{x}_i}(\cdot, \cdot)$ by using (\ref{eq.KBLUP}) and (\ref{eq.undervar}).
		
		%\item[] 
		Step~(2.3): {Check the stopping criteria (\ref{eq.condAlpha}) for each scenario $\pmb{\xi}_{ij}$ with $j=1,\ldots,N^{(k)}(\mathbf{x}_i)$ and move the terminated ones to the set $\mathcal{U}_{\mathbf{x}_i}$. If $ |\mathcal{U}_{\mathbf{x}_i} | =  N^{(k)}(\mathbf{x}_i)$, stop the second-stage optimization for $\mathbf{x}_i$. Otherwise, continue the optimal search for each remaining scenario $\pmb{\xi}_{ij} \notin \mathcal{U}_{\mathbf{x}_i}$.}
		
		\While {$ |\mathcal{U}_{\mathbf{x}_i} | < N^{(k)}(\mathbf{x}_i)$}	
		{
			\For {$1 \leq \ j \leq N^{(k)}(\mathbf{x}_i) \mbox{ and } \ \pmb{\xi}_{ij} \notin \mathcal{U}_{\mathbf{x}_i}$}{
				%	\begin{itemize}
				
				%	\item[a.] 
				(a) Find $\mathbf{y}_{ij}^{EI}$ by using (\ref{eq: EIsearch2}).  
				
				%\item[b.]
				(b) If $\frac{  \mbox{E}_{\widetilde{q}_{\mathbf{x}_i}(\mathbf{y}_{ij},\pmb{\xi}_{ij})}\left[\mathrm{I}(\mathbf{y}_{ij}^{EI}, \pmb{\xi}_{ij})\right] }
				{ |     q_{\mathbf{x}_i}(\widehat{\mathbf{y}}^\star_{ij},\pmb{\xi}_{ij})            |} \leq \alpha^{(k)}(\mathbf{x}_i) $, then
			   {terminate the optimal search for scenario $\pmb{\xi}_{ij}$ and set} $\mathcal{U}_{\mathbf{x}_i}=\mathcal{U}_{\mathbf{x}_i} \bigcup \{ \pmb{\xi}_{ij}\} $. Otherwise, 
				run simulation at $(\mathbf{y}_{ij}^{EI}, \pmb{\xi}_{ij})$ and add it to $\mathcal{P}_{\mathbf{x}_i}$. {Update the local metamodel for ${q}_{\mathbf{x}_i}(\cdot, \cdot)$ based on $\mathcal{P}_{\mathbf{x}_i}$ by using (\ref{eq.KBLUP})--(\ref{eq.undervar})} and update $(\widehat{\mathbf{y}}^\star_{ij}, \pmb{\xi}_{ij})$.}
				%	\end{itemize}
			}
			{(c) {Estimate  ${\mbox{E}} [\delta(\mathbf{x}_i,\pmb{\xi}_{ij}) | \pmb{\xi}_{ij}]$ by (\ref{eq: meanog}).}
		}
		
		\vspace{0.1in}
		\textbf{Step~(3):}  Solve the first-stage optimization.
		%	\begin{itemize}
		
		%\item [] Step~(3.1): Update $\delta(\mathbf{x}_i,\pmb{\xi}_{ij})$ for $\mathbf{x}_i\in \mathcal{S}_{\mathbf{x}}^{(k)}$ and $j=1,2, \ldots,N^{(k)}(\mathbf{x}_i)$.
		
		%\item[] 
		Step~(3.1): Construct/Update the global metamodel $\mathcal{GP}_o(\widehat{G}(\cdot),s^2(\cdot))$ by using (\ref{eq.predictor}) and (\ref{eq.MSE}). Find the current optimal decision, {$\widehat{\mathbf{x}}^{\star(k)} =\arg \min_{\mathbf{x} \in \mathcal{S}_{\mathbf{x}}^{(k)}}\bar{G}(\mathbf{x})$}.
		
				%\item[] 
		{Step~(3.2): If the simulation budget $C$ is exhausted, stop the search procedure and report $\widehat{\mathbf{x}}^{\star(k)}$. Otherwise, 
		construct the sampling distribution $f^{(k+1)}(\mathbf{x})$ for $\mathbf{x}\in \mathcal{X}$ by (\ref{eq: sampling}) and use it to generate a new set of design points $D_{\mathbf{x}}^{(k+1)}$ with {size} $|D_{\mathbf{x}}^{(k+1)}|=s$. {Update the set of design points $\mathcal{S}_{\mathbf{x}}^{(k+1)}=\mathcal{S}_{\mathbf{x}}^{(k)}\bigcup D_{\mathbf{x}}^{(k+1)}$. Set the iteration index $k=k+1$ and return to Step~(2)} }

		\caption{{The Procedure for Global-Local Metamodel-Assisted Two-Stage {OvS}}.}
		\label{Global}
	\end{algorithm}

	Then, we iteratively solve the second- and first-stage optimization. At the $k$-th iteration, we solve the second-stage {optimization problems $\mathbf{y}_{ij}^\star = \arg\min_{\mathbf{y} \in \mathcal{Y}(\mathbf{x}_i)} {q}_{\mathbf{x}_i}(\mathbf{y}, \pmb{\xi}_{ij})$ with $j=1,2,\ldots,N^{(k)}(\mathbf{x}_i)$ for} each $\mathbf{x}_i \in D_{\mathbf{x}}^{(k)}$ in Step~(2). Here, $\alpha^{(k)}(\mathbf{x}_i) $ is the relative EI threshold which controls the stopping criteria for second-stage {optimal search}.
	%The smaller $\alpha^{(k)}(\mathbf{x}_i)$ is, the smaller the optimality gap $\delta(\mathbf{x}_i,\pmb{\xi}_{ij})$ is. 
	For any design point that is visited for the first-time,  $\mathbf{x}_i \in D_{\mathbf{x}}^{(k)}$ and $\mathbf{x}_i\notin \mathcal{S}_{\mathbf{x}}^{(k-1)}$, we set  $\alpha^{(k)}(\mathbf{x}_i)=\alpha_0$ and generate $N^{(k)}(\mathbf{x}_i)=n_0$ scenarios. Since too large $\alpha_0$ could lead to unreliable optimization results and too small $\alpha_0$ could cause wasting the effort on unpromising candidates, we set $\alpha_0 =0.1$ in the empirical study. %\textcolor{red}{The results in Section~\ref{subsec:sensitivity} also indicate that the performance of our approach depends on the choice of $\alpha_0$.}
% 	The results in Section~\ref{sec:example} also indicate that the performance of our approach is robust to the selection of $\alpha_0$. 
	Then, we use LHD to generate the design points of $(\mathbf{y},\pmb{\xi})$ evenly covering the space 
	$\mathcal{Y}(\mathbf{x}_i)\times \Xi$, run simulations, and construct the initial local metamodel  {$\mathcal{GP}_{\mathbf{x}_i}(\widehat{q}_{\mathbf{x}_i}(\cdot,\cdot),s_{\mathbf{x}_i}^2(\cdot,\cdot))$} by using (\ref{eq.KBLUP}) and (\ref{eq.undervar}).
	%\textcolor{red}{To obtain the initial local metamodel by the minimax LHD, we notice that each dimension of random samples generated by the standard maximin LHS design follows the standard uniform distribution, i.e. the outputs are uniformly distributed between $0$ and $1$ for each dimension \citep{Stein_1987,Morris_Mitchell_1995, Kleijnen_etal_2010}. Hence, we can regard the generated standard LHD random samples as \textit{quantiles}. Then, we can apply the corresponding marginal inverse cumulative distribution function (CDF) to obtain the corresponding LHD samples covering the random inputs $\pmb{\xi}=({\xi}_1,\xi_2, \ldots,{\xi}_T)$.  Suppose for ${\xi}_t$, $t=1, \ldots,T$, its marginal distribution is $F(\xi_t)$, i.e. $\xi_t \sim F(\xi_t)$, then we can apply its inverse $F^{-1}(\xi_t)$ to the $t$-dimension of standard LHD samples and obtain the corresponding LHD output on the $t$-th dimension. }
	For each revisited design point, i.e., {$\mathbf{x}_i \in\{ D_{\mathbf{x}}^{(k)} \cap  \mathcal{S}_{\mathbf{x}}^{(k-1)}$\}}, we increase the number of scenarios $N^{(k)}(\mathbf{x}_i)$ and reduce the threshold $\alpha^{(k)}(\mathbf{x})$ to simultaneously control {the impact from finite sampling error and second-stage optimality gap}. 
	Inspired by Lesnevski et al.~\citep{Lesnevski_Nelson_Staum_2007,Lesnevski_Nelson_Staum_2008}, for the $t$-th visited first-stage action $\mathbf{x}\in\mathcal{X}$ with $t>1$, the relative EI threshold for second-stage optimization is set to be $\alpha_0  g^{-\frac{t-1}{2}}$ and the accumulated number of scenarios is set to be $\ceil{n_0g^{t-1}}$.
	Thus, for the design point $\mathbf{x}$ revisited in the $k$-th iteration, we add $N^{(k-1)}(\mathbf{x})(g-1)$ additional scenarios.  Following Lesnevski et al.~\citep{Lesnevski_Nelson_Staum_2007}, we set $g=1.5$.

	In Step~(2.2), we search the optimal solutions for those new generated scenarios $\pmb{\xi}_{ij}$ by using the local metamodel,
	\begin{equation}
	\widehat{\mathbf{y}}_{ij}^\star={\arg
	\min}_{\mathbf{y}_{ij} \in \mathcal{Y}(\mathbf{x}_i )} \widehat{q}_{\mathbf{x}_i}(\mathbf{y}_{ij},\pmb{\xi}_{ij}).
	\label{eq: initialopt}
	\end{equation}
	{In Step~(2.3), we first check the stopping criteria for each scenario $\pmb{\xi}_{ij}$ with $j=1,\ldots,N^{(k)}(\mathbf{x}_i)$,
	\begin{equation}
	\frac{  \max_{\mathbf{y}_{ij} \in \mathcal{Y}(\mathbf{x}_i )} \mbox{E}_{\widetilde{q}_{\mathbf{x}_i}(\mathbf{y}_{ij},\pmb{\xi}_{ij})}\left[\mathrm{I}(\mathbf{y}_{ij}, \pmb{\xi}_{ij})\right] }
	{ |     q_{\mathbf{x}_i}(\widehat{\mathbf{y}}^\star_{ij},\pmb{\xi}_{ij})            |}  \leq \alpha^{(k)}(\mathbf{x}_i) 
	\label{eq.condAlpha}
	\end{equation} 
	with $\mbox{E}_{\widetilde{q}_{\mathbf{x}_i}(\mathbf{y}_{ij},\pmb{\xi}_{ij})}\left[\mathrm{I}(\mathbf{y}_{ij}, \pmb{\xi}_{ij})\right]$ obtained by (\ref{eq. formEI}). We move those scenarios satisfying the stopping criteria (\ref{eq.condAlpha}) to the set $\mathcal{U}_{\mathbf{x}_i}$ that inlcudes all terminated scenarios. For the remaining scenario $\pmb{\xi}_{ij}$ with $\pmb{\xi}_{ij} \notin \mathcal{U}_{\mathbf{x}_i}$, to efficiently reduce the optimality gap, we find the point $\mathbf{y}_{ij}^{EI}$ giving the maximum EI,
	\begin{equation}
	\mathbf{y}_{ij}^{EI}=\arg
	\max_{\mathbf{y}_{ij} \in \mathcal{Y}(\mathbf{x}_i )} \mbox{E}_{\widetilde{q}_{\mathbf{x}_i}(\mathbf{y}_{ij},\pmb{\xi}_{ij})}\left[\mathrm{I}(\mathbf{y}_{ij}, \pmb{\xi}_{ij})\right].
	\label{eq: EIsearch2}
	\end{equation}
We add the point $(\mathbf{y}_{ij}^{EI}, \pmb{\xi}_{ij})$ to the set $\mathcal{P}_{\mathbf{x}_i}$ that includes all design points for the local metamodel $\mathcal{GP}_{\mathbf{x}_i}$ and run simulation there. Then, we update the metamodel by using (\ref{eq.KBLUP})--(\ref{eq.undervar}) and also update $\widehat{\mathbf{y}}^\star_{ij}$.  We repeat this procedure for all scenarios $\pmb{\xi}_{ij}$ that have not been terminated. 
	If all $N(\mathbf{x}_i)$ optimization problems meet the stopping criteria, $ |\mathcal{U}_{\mathbf{x}_i} | = N(\mathbf{x}_i)$, we terminate the second-stage optimal search for $\mathbf{x}_i$ in the $k$-th iteration and estimate  ${\mbox{E}} [\delta(\mathbf{x}_i,\pmb{\xi}_{ij}) | \pmb{\xi}_{ij}]$ by using (\ref{eq: meanog}).
}

	{In Step~(3), we solve the first-stage optimization. Given the results from second-stage optimization, we construct/update the global GP metamodel $\mathcal{GP}_o$ by using (\ref{eq.predictor})--(\ref{eq.MSE}) and find the current optimal decision, $\widehat{\mathbf{x}}^{\star(k)} =\arg \min_{\mathbf{x} \in \mathcal{S}_{\mathbf{x}}^{(k)}}\bar{G}(\mathbf{x})$. 
	The optimal search terminates when the simulation budget $C$ is exhausted. Otherwise, we generate a new set of design points $D_{\mathbf{x}}^{(k+1)}$, set the number of iteration $k=k+1$ and then loop back to Step~(2) for solving the second-stage optimization problems at any $\mathbf{x}_i\in D_{\mathbf{x}}^{(k+1)}$.}
	The EI criterion used in the second-stage optimization was originally introduced for the deterministic simulation and it is appropriate for the stochastic cases; see Huang et al.~\citep{Huang_etal_2006G} and Quan et al.~\citep{Quan_etal_2013}. 
	Thus, following Sun et al.~\citep{Sun_etal_2014}, we construct a sampling distribution,
	{
	\begin{equation}
	f^{(k+1)}(\mathbf{x})=\frac{\mbox{Pr}\{\widetilde{G}(\mathbf{x})
		<\bar{G}(\widehat{\mathbf{x}}^{\star(k)})\}}{\sum_{\mathbf{x}_i \in \mathcal{X}} \mbox{Pr}\{\widetilde{G}(\mathbf{x}_i)<\bar{G}(\widehat{\mathbf{x}}^{\star(k)})\}},
	\label{eq: sampling}
	\end{equation}
	where $\widetilde{G}(\cdot) \sim \mathcal{GP}_o(\widehat{G}(\cdot), s^2(\cdot))$.}
	Since $\mbox{Pr}\{\widetilde{G}(\mathbf{x})
	<\bar{G}(\widehat{\mathbf{x}}^{\star(k)}) \}$ is the posterior possibility that the point $\mathbf{x}\in\mathcal{X}$ achieves a better objective value than the current optimal, the normalized probability mass function $f^{(k+1)}(\mathbf{x})$ reflects the potential of point $\mathbf{x}$. The sampling distribution $f^{(k+1)}(\mathbf{x})$ is used for generating a new set of first-stage design points $D_{\mathbf{x}}^{(k+1)}$ for the next $(k+1)$-th iteration {with the size $|D_{\mathbf{x}}^{(k+1)}|=s$}. Following Sun et al.~\citep{Sun_etal_2014}, we set $s=5$ in the empirical study. 
{
Notice that if the first-stage solutions $\mathbf{x}$ drawn in the early iterations are promising, they are more likely to be selected again and we invest more simulation budget there to get more accurate estimation on $G(\mathbf{x})$. On the other hand, if such first-stage candidate solutions are inferior, they are less likely to be selected again. Thus, the proposed algorithm can efficiently utilize the simulation budget to search for the optimal solution $\mathbf{x}^\star$.}

	\subsection{Convergence of Global-Local Metamodel Assisted Two-Stage OvS}
	\label{subsec: prove}
	
	{In this section, we provide the theoretical guarantee of the convergence for the proposed global-local metamodel assisted two-stage OvS.} {The results are given under following assumptions:
	\begin{enumerate}
	    \item \label{assumption:bounded} The response 
		$\theta(\mathbf{x},\mathbf{y},\pmb{\xi}) \equiv c_0(\mathbf{x})+{q}(\mathbf{x}, \mathbf{y}, \pmb{\xi}) $
		is bounded for any $\mathbf{x}\in\mathcal{X}$, $\mathbf{y}\in\mathcal{Y}(\mathbf{x})$ and $\pmb{\xi}\sim F(\pmb{\xi})$. There exists a bound $\overline{M}$ such that $\left\Vert q\right\Vert_{\mathcal{H}_\phi(\mathcal{X})}\leq \overline{M}$, where $\mathcal{H}_{\pmb{\phi}}(\mathcal{X})$ denotes the reproducing kernel Hilbert space (RKHS) of correlation function $R_{\pmb{\phi}}(\mathbf{z},\mathbf{z}^\prime)$ ($\mathbf{z}\equiv (\mathbf{y},\pmb{\xi})$) on $(\mathcal{Y}(\mathbf{x}), \Xi)$.
		\item \label{assumption:discrete}  The first- and second-stage feasible sets are finite, i.e., $|\mathcal{X}| <\infty$ and $|\mathcal{Y}(x)|<\infty$ for any $\mathbf{x}\in\mathcal{X}$.
		\item \label{assumption:spatialvariane}  The spatial variance for first- and second-stage GP is strictly positive, i.e., $\sigma^2>0$ and $\tau^2 >0$.
		\item \label{assumption:consistency} The MLEs for $\sigma^2$ and $\tau^2$ and $\phi$ are consistent under some regularity conditions \cite{Wald1949}.
	\end{enumerate}
	 The main result is shown in Theorem~\ref{thm:conv} and detailed proofs are provided in the appendix.}
		
% 			\textcolor{red}{For simplicity, we present without detailed proof of Proposition~\ref{proposition1:revisit}, ~\ref{proposition2:samplepath_revisit} and ~\ref{lemma1:revisit} because we use the same way to construct the sampling distribution $f^{(k)}(\mathbf{x})$ to generate and update the first stage design points as Sun et al.~\citep{Sun_etal_2014}. And also these three propositions mainly discuss the guarantee of repeat visit of all first-stage decision $\mathbf{x}$ as number of iteration $k$ grows without considering second-stage optimization. For Proposition~\ref{proposition3:liminf}, Lemma~\ref{lemma2:gapconv0}, since they are related to the second-stage optimality gap, we give detailed proof.}

% 	\begin{proposition} \label{proposition0:convergence}
% 	Given $X_1,X_2,\ldots$ an infinite sequence of $i.i.d.$ random variables with finite expected value $\mbox{E}(X_k)=c_k$ and $\lim_{k\rightarrow\infty}c_k = c$, then $\frac{1}{n}\sum_{k=1}^n X_k \rightarrow c.$ in probability as $k\rightarrow\infty$.
% 	\end{proposition}

	{%Since first- and second-stage decision spaces are finite, as long as our algorithm can reduce the finite sampling error and make sure the second-stage optimality gap converges to $0$ as $k$ grows, the gloabl convergence will naturally be reached by the strong law of large numbers. Inspired by the above ideas 
	We start with Lemma~\ref{lemma1:revisit} showing that the number scenarios allocated to any $\mathbf{x}\in\mathcal{X}$ goes to infinity as the number of iterations or the simulation budget goes to infinity, $N^{(k)}(\mathbf{x})\rightarrow \infty$ w.p.1 as $k\rightarrow \infty$. Then, the finite sampling error reduces to zero.
	}
	{
	\begin{lemma} \label{lemma1:revisit}
	Suppose that the metamodel-assisted optimization approach proposed in Algorithm~\ref{Global} is used to solve the two-stage optimization problem~(\ref{eq.intro1}) and Assumptions~\ref{assumption:bounded}--\ref{assumption:consistency} holds. Then $N^{(k)}(\mathbf{x}) \rightarrow \infty$ w.p.1 $k\rightarrow \infty$ $\forall x \in \mathcal{X}$.
	\end{lemma}}
	\raggedbottom
{Furthermore, since the second-stage optimization search is based on the expected improvement. In the proposed algorithm, we gradually reduce the second-stage optimality gap through controlling the threshold of relative expected improvement. Then, in Lemma~\ref{lemma2:gapconv0}, we can show that the expected improvement for any unobserved second-stage decision $\mathbf{y}$ is positive and bounded. % at first-stage decision $\mathbf{x}$ for any scenario $\pmb{\xi}$ and b
By letting the threshold $\alpha^{(k)}(\mathbf{x})$ gradually decreasing to zero as $k\rightarrow \infty$, all untried points will eventually be simulated. It implies that we eventually visit all possible solutions in $\mathcal{Y}(\mathbf{x})$ when $k$ is large and the optimality gap becomes zero.}
		{
	\begin{lemma} \label{lemma2:gapconv0}
	Suppose that the metamodel-assisted optimization approach proposed in Algorithm~\ref{Global} is used to solve the two-stage optimization problem~(\ref{eq.intro1}) and Assumptions~\ref{assumption:bounded}--\ref{assumption:consistency} holds. Then $\delta(\mathbf{x},\pmb{\xi}_{j}) \rightarrow 0$ w.p.1 as $k\rightarrow \infty$ $\forall \mathbf{x} \in \mathcal{X}$.
	\end{lemma}}
	\raggedbottom
		Theorem~\ref{thm:conv} shows that our two-stage optimization approach can guarantee the global convergence as the simulation budget $C$ goes to infinity. Since it simultaneously reduces the finite sampling error introduced by SAA and the error induced by the second-stage optimality gap to zero, we have a consistent performance estimator as the simulation budget goes to infinite, $\bar{G}(\mathbf{x}) \rightarrow G(\mathbf{x})$ as $N(\mathbf{x})	\rightarrow \infty$ for any $\mathbf{x} \in \mathcal{X}$. 
	Then, following the proof in Sun et al.~\citep{Sun_etal_2014}, we can show that $\bar{G}(\widehat{\mathbf{x}}^{\star}) {\rightarrow} G(\mathbf{x}^\star)$ w.p.1 as the budget $C \rightarrow \infty$ or the number of iteration $k\rightarrow \infty$.
	
	\begin{theorem} 
		\label{thm:conv}
		Suppose that Assumptions~\ref{assumption:bounded}--\ref{assumption:consistency} holds and the metamodel-assisted optimization approach proposed in Algorithm~\ref{Global} is used to solve the two-stage optimization problem~(\ref{eq.intro1}). Denote $\mathbf{x}^\star$ as the true optimal solution, i.e., $G(\mathbf{x}^\star)=\arg\min_{\mathbf{x} \in \mathcal{X}} G(\mathbf{x})$. Let $\widehat{\mathbf{x}}^{\star (k)}$ be the optimal decision obtained in the $k$-th iteration with the objective estimate $\bar{G}(\widehat{\mathbf{x}}^{\star (k)})$. Then, $\bar{G}(\widehat{\mathbf{x}}^{\star (k)}) {\rightarrow} G(\mathbf{x}^\star)$ w.p.1 as the simulation budget $C \rightarrow \infty$ or the iteration $k\rightarrow \infty$.

	\end{theorem}

	% =================== Empirical Study ============================================
	
	\section{Empirical Study}
	\label{sec:example}

	{In this paper, we consider two-stage stochastic programming for complex systems with unknown second-stage response surface. %To the best of our knowledge, there is no algorithm developed for two-stage stochastic programming in the simulation literature. 
	The existing simulation optimization approaches typically consider one-stage optimization. 
    In addition, to the best of our knowledge, the two-stage stochastic programming approaches typically assume the response surface $q_{\mathbf{x}}(\mathbf{y},\pmb{\xi})$ known; see the literature review in Section~\ref{sec:liter}. Thus, in this section, we compare 	
	the finite sample performance of proposed global-local metamodel assisted two-stage OvS  with two methods, including a random sampling SAA approach and the deterministic look-ahead (DLH) policy model solved by using the state-of-art simulation optimization approach, called Gaussian process-based search approach (GPS) proposed in \cite{Sun_etal_2014}. } 
%The empirical results demonstrate that our approach can efficiently use the simulation resource to solve two-stage optimization, and it has stable and good finite sample performance.  

	{
	For the \textit{random sampling SAA approach}, 
	without any prior information about the optimal first- and second-stage decisions, suppose that all first-stage solutions have equal probability to be optimal. Also at any given first-stage decision, all second-stage solutions have equal probabilities to be the best. Thus, in {this approach}, we randomly generate $N_1$ first-stage candidate solutions, $\{\mathbf{x}_1,\mathbf{x}_2,\ldots,\mathbf{x}_{N_1}\}$. At each solution $\mathbf{x}_i$ with $i=1,2,\ldots,N_1$, we generate $N_2$ scenarios $\{\pmb{\xi}_{i1},\pmb{\xi}_{i2},\ldots,\pmb{\xi}_{iN_2}\}$. For each scenario $\pmb{\xi}_{ij}$ with $j=1,2,\ldots,N_2$, we randomly {generate} ${C}/{(N_1 N_2)}$ second-stage solutions and select the best one as the optimal $\widehat{\mathbf{y}}_{ij}^\star$. Then, for each visited first-stage solution $\mathbf{x}_i$ with $i=1,2,\ldots,N_1$, we aggregate the results from all $N_2$ second-stage optimization problems, and obtain an estimate of the objective value, $\bar{G}(\mathbf{x}_i)=c_0(x_i)+\sum_{j=1}^{N_2}q_{\mathbf{x}_i}(\widehat{\mathbf{y}}_{ij}^\star,\pmb{\xi}_{ij})/N_2$. The first-stage solution giving the best objective is selected, $\widehat{\mathbf{x}}^\star = \arg\min_{\mathbf{x}_i\in \{\mathbf{x}_1,\mathbf{x}_2,\ldots,\mathbf{x}_{N_1}\}}\bar{G}(\mathbf{x}_i)$.} 
	
	{
For the \textit{deterministic look-ahead policy model} (see the description in \cite{WarrenPowell_2014}), we consider	
\begin{equation}
	\min_{\mathbf{x} \in \mathcal{X}, \mathbf{y} \in \mathcal{Y}(\mathbf{x})}  \mbox{  } G_D(\mathbf{x},\mathbf{y}) \equiv c_0(\mathbf{x})+\mbox{E}_{\pmb{\xi}}
	\left[ 
	 {q}(\mathbf{x}, \mathbf{y}, \pmb{\xi})\right], 
	\label{eq.singlestage}  
	\end{equation} 
which can be used for stochastic control {\citep{Birge:2011:ISP:2031490}}. Then, the Gaussian process-based search approach (GPS) proposed in \citep{Sun_etal_2014} is used to efficiently solve the optimization problem in (\ref{eq.singlestage}). 
Specifically, we model the unknown mean response surface $G_D(\cdot)$ with a GP metamodel having the spatial variance denoted as $\sigma^2_{GPS}$. Then, following the simulation optimization proposed in \citep{Sun_etal_2014}, we develop a sampling distribution to efficiently guide the search for the optimal decisions of $(\mathbf{x},\mathbf{y})$. 
	In each iteration, the sampling distribution is used to generate $m$ promising decisions of $(\mathbf{x},\mathbf{y})$. At each selected  $(\mathbf{x},\mathbf{y})$, we assign $r$ independent scenarios of $\pmb{\xi}$. We repeat this search procedure until reaching to the computational budget. 
}

\begin{comment}
To save the space for describing the notations, we follow the same notation and parameter setting used in \citep{Sun_etal_2014}.
We apply the GPS with the following model:
\begin{equation}
    G_D(\pmb{w}) = M(\pmb{w}) + \pmb{\lambda}^T(\pmb{w})(\bar{\mathbf{G}}_{\mathcal{P}_o} - \mathbb{M}) + \pmb{\lambda}^T\mathcal{E}
\end{equation}
where the stationary Gaussian process $M(\cdot)$ with mean 0
and covariance function $\sigma_{GPS}^2\gamma(\pmb{w}, \pmb{w}^\prime)$ models our belief of underlying response surface $G_D(\pmb{w})$,  $\pmb{\lambda}(\pmb{w})=(\lambda_1(\pmb{w}),\ldots, \lambda_K(\pmb{w}))^T$ is a vector of weight functions, $\mathbb{M}=(M(\pmb{w}_1),\ldots,M(\pmb{w}_K))$ is a vector of $M(\pmb{w})$ evaluated at design points $\pmb{w}_1,\ldots, \pmb{w}_K$ and
$\mathcal{E}$ is an $K$-dimensional simulation errors at design points. %random vector following a multivariate normal distribution with mean 0 and covariance matrix $\Sigma_{\mathcal{E}}=\diag\{\frac{\sigma^2(\pmb{w}_1)}{n_1},\ldots,\frac{\sigma^2(\pmb{w}_K)}{n_K}\}$
The correlation function is
	$\gamma(\pmb{w}_1, \pmb{w}_2) = 
	\exp( - a\lVert \pmb{w}_1, \pmb{w}_2 \rVert)$ and the vector of weight functions is
	\begin{equation}
	\lambda_i(\pmb{w})=\begin{cases}
	\frac{(1 - \gamma(\pmb{w}, \pmb{w}_i))^{-1}}{\sum^{K}_{j=1}(1 - \gamma(\pmb{w}, \pmb{w}_j))^{-1}}, & \text{$\pmb{w} \neq \pmb{w}_i$},\\
	1, & \text{$\pmb{w} = \pmb{w}_i$},
	\end{cases}
	\end{equation}
	where we set $a = 1$.
\end{comment}

{To study the finite sample performance of proposed framework and compare it with the random sampling SAA and the deterministic look-ahead approach with GPS, we consider two examples, including a simple two-stage linear optimization problem and a supply chain management example.}
	
\subsection{A Two-Stage Linear Optimization Problem}
	\label{subsec: toyexample}

	We first consider a two-stage linear stochastic optimization example from Ekin et al.~\citep{Ekin_Polson_Soyer_2014},
	\begin{eqnarray}
		\min_{0 \leq x \leq 3} \mbox{  } G(x)=-3x+\mbox{E} [Q(x, \xi) ]  
		\quad \mbox{   with  } \quad Q(x, \xi)= \min_{y \geq 0} \{\xi y: 0.5x+y \leq 5\} 
		\label{eq: toyobj} 
	\end{eqnarray}  
	where $\xi$ follows the lognormal distribution with mean $0$ and variance $1$. It is easy to see that the first-stage optimal solution is $x^\star=3$.  
	Notice that the objective function in (\ref{eq: toyobj}) is monotonic in $x$, which means that the closer a solution $x$ is to the true optimal
	$x^\star=3$, the better quality it has. Thus, we can examine the algorithm's performance by directly checking the value of optimal first-stage decision $\widehat{x}^\star$.   
	To study the performance of our approach, we pretend that the objective function is unknown, and it is estimated by simulation. 
	%Since our framework considers the discrete solution space, 
	We discretize the solution spaces of $x$ and $y$ with an increment $0.01$.
	In our optimization procedure, we set the initial threshold for the second-stage relative optimality gap $\alpha_0=0.1$ and set the initial number of scenarios  $n_0=10$.

	We study the performance of proposed global-local metamodel assisted two-stage OvS, random sampling SAA approach, and deterministic look-ahead policy with GPS (DLH-GPS) under different simulation budget $C=600,1000,2000$. For the random sampling SAA approach, we consider two representative settings for $N_1$ and $N_2$: $N_1=10, N_2=10$ and $N_1=10,N_2=20$. 
	For the GPS, we set $m=10, r=10$ when the budget $C=600,1000$ and set $m=13,r=10$ when $C=2000$. Following the setting of Section~(5.1) in \citep{Sun_etal_2014}, we set the spatial variance of the GP to be $\sigma_{GPS}=5$.
	Table~\ref{table: toy} records mean and standard deviation (SD) of $\widehat{x}^\star$ obtained by using these approaches. The results are estimated based on 100 macro-replications. Given the same simulation budget, our method and deterministic look-ahead with GPS provide much higher quality solutions in terms of means and standard deviations of $\widehat{x}^\star$. They deliver $\widehat{x}^\star$ very close to $x^\star=3$ with all three budget levels. 
	By contrast, the optimal decision obtained by the random sampling SAA approach has low quality and high estimation uncertainty, and it shows only a small improvement as $C$ increases. %Thus, our approach provides a reliable and high quality performance.  
	
		\begin{table}[h]
		\centering
		\caption{Mean and SD of $\widehat{x}^\star$ obtained by three candidate approaches
		\label{table: toy} }{
		\begin{tabular}{|c|c|c|c|c|c|c|c|c|}
			\hline
			& \multicolumn{2}{|c|}{Our approach} &
			\multicolumn{2}{|c|}{DLH-GPS} &
			\multicolumn{2}{|c|}{\begin{tabular}[c]{@{}c@{}}Random sampling SAA\\ $(N_1=10,N_2=10)$\end{tabular}} & \multicolumn{2}{|c|}{\begin{tabular}[c]{@{}c@{}}Random sampling SAA\\ $(N_1=10,N_2=20)$\end{tabular}}
			\\  \cline{2-7}
			\hline
			& mean          & SD   & mean          & SD          & mean         & SD        & mean            & SD                                           \\ \hline
			$C=600$     & 2.91         & 0.07   & 2.91 &0.07    & 2.69         & 0.28      & 2.67            & 0.27                                      \\ \hline
			$C=1000$     & 2.92         & 0.05   &2.92 & 0.06    & 2.68         & 0.29      & 2.67            & 0.31                                         \\ \hline
			$C=2000$     & 2.94         & 0.04   & 2.93 & 0.05      & 2.73        & 0.22      & 2.73            & 0.22                                         \\ \hline
		\end{tabular}
		}
	\end{table}

	We also examine the estimation accuracy of $\bar{G}(\widehat{x}^\star)$, the estimator of  corresponding objective $G(\widehat{x}^\star)$ of the obtained optimal solution $\widehat{x}^\star$. Since for each $\widehat{x}^\star$ with $G(\widehat{x}^\star)=-3 \widehat{x}^\star$, we can calculate the relative estimation error 
	$rE \equiv \frac{|\bar{G}(\widehat{x}^\star)-G(\widehat{x}^\star)|}{G(\widehat{x}^\star)} \cdot 100\%$. Table ~\ref{table: toybias} documents mean and SD of $rE$ in $\%$  obtained by using these three approaches. Our method can provide reliable estimates of $G(\widehat{x}^\star)$. The random sampling SAA method shows its deficiency in the objective value estimation accuracy and it has only a small improvement as $C$ increases.
				\begin{table}[h]
	\centering
		\caption{Mean and SD of relative estimation error $rE\equiv {|\bar{G}(\widehat{x}^\star)-G(\widehat{x}^\star)|}/{G(\widehat{x}^\star)} \cdot 100\%$ (in $\%$).
		\label{table: toybias}}{
		\begin{tabular}{|c|c|c|c|c|c|c|c|c|}
			\hline
			& \multicolumn{2}{|c|}{Our approach} &
			\multicolumn{2}{|c|}{DLH-GPS} &
			\multicolumn{2}{|c|}{\begin{tabular}[c]{@{}c@{}}Random sampling SAA\\ $(N_1=10,N_2=10)$\end{tabular}} & \multicolumn{2}{|c|}{\begin{tabular}[c]{@{}c@{}}Random sampling SAA\\ $(N_1=10,N_2=20)$\end{tabular}}
			\\  \cline{2-7}
			\hline
			& mean          & SD   & mean          & SD          & mean         & SD        & mean            & SD                                           \\ \hline
			$C=600$     & 5.4         & 3.0     & 5.5 & 3.2   & 24.2         & 5.0      & 23.1            & 5.6                                     \\ \hline
			$C=1000$     & 5.0         & 2.9       & 5.2 & 2.8 & 24.1         & 6.0      & 21.7            & 6.2                                         \\ \hline
			$C=2000$     & 4.0   & 3.0  & 4.0 & 3.0   & 22.1        & 4.0      & 22.0            & 6.0   \\\hline
		\end{tabular}}
	\end{table}
	
In addition, we further study the two-stage linear optimization problem in (\ref{eq: toyobj}) and note that $\xi y$ is monotonically increasing in $y$ for $\xi > 0$. Since $\xi$ follows the lognormal distribution, it means that the optimal value $y^\star$ doesn't depend on $\xi$. That explains why the proposed global-local metamodel assisted two-stage OvS demonstrates the similar performance with the deterministic look-ahead approach with GPS. % It also indicates we can use this Therefore except for two-stage algorithm, we can also use the cutting-edge single-stage algorithm to the equivalent objective (\ref{eq: toyobjequivalent}).}
% 	\begin{table}
% 		\caption{Mean and SD of relative estimation error $rE$ by our approach and the naive approach in $\%$
% 		\label{table: toybias}}{
% 		\begin{tabular}{|c|c|c|c|c|c|c|}
% 			\hline
% 			& \multicolumn{2}{|c|}{Our approach} & \multicolumn{2}{|c|}{Naive $(N_1=10,N_2=10)$} & \multicolumn{2}{|c|}{Naive $(N_1=10,N_2=20)$}
% 			\\  \cline{2-7}
% 			\hline
% 			& mean          & SD             & mean         & SD        & mean            & SD                                           \\ \hline
% 			$C=600$     & 5.4         & 3.0           & 24.2         & 5.0      & 23.1            & 5.6                                     \\ \hline
% 			$C=1000$     & 5.0         & 2.9           & 24.1         & 6.0      & 21.7            & 6.2                                         \\ \hline
% 			$C=2000$     & 4.0         & 3.0           & 22.1        & 4.0      & 22.0            & 6.0                                         \\ \hline
% 		\end{tabular}}
% 	\end{table}

% 	\vspace{-0.in}

	\subsection{A Supply Chain Management Example}
	\label{subsec:supplyChainExample}
	In this section, we use a supply chain management example to study the performance of our approach. It is inspired by our research collaboration with a bio-pharmaceutical manufacturing company.
	The company produces various commercial and clinical products, which requires some common vital raw materials, including soy and other chemical raw materials. For simplification, we only consider the soy raw material and one type of raw chemical material  used for producing the key clinical product. The company orders soy and chemical raw material from outside vendors. While the chemical raw material can be fast-delivered, due to the regulations and long testing cycles, soy has long lead time. Since the clinical demand has high prediction uncertainty \citep{Kaminsky_Wang_2015}, the company faces high fluctuations in the total cost.  Thus, %to reduce the production cost, 
	{the company
	is interested in finding the first-stage soy ordering decision $x$ and the second-stage decisions $\mathbf{y}=(u,s,S)$, including production scheduling $u$ and inventory control for the raw chemical material (specified by $(s,S)$ review policy), to minimize the expected  overall cost.}  
	%%%%%%%%%%%%

	{
	Considering the long lead time of soy delivery, the company first forecasts the clinical demand and places the soy order $x$ in advance. Suppose that $x$ is within the range $[0, 5000]$ with an increment $20$. Then, after the clinical demand is the realized, the company needs to make two types of decisions: the inventory control for the chemical material and the daily production decision $u$.
	Due to the fast-delivery nature, the company excises the daily review $(s,S)$ policy for the chemical raw material satisfying $100 \leq s < 400$, $200 \leq S < 500$ and $s < S$. Here, we consider a variety of choices: $(s,S)=(100, 200)$, $(100, 300)$, $(100,400)$,$(100,500)$, $(200,300)$,$(200,400)$,$(200,500)$,$(300,400)$, $(300,500)$, and $(400, 500)$. For simplification, suppose that the chemical raw material orders have zero lead time.}

	{The production planning horizon has four weeks, and each week has five work days. Let $\mathcal{D}=(d_1, d_2, d_3, d_4)$, where $d_i$ denotes the aggregated clinical demand occurring in the $i$-th week with $i=1, 2, \ldots, 4$. 
	If the production can not fully meet the demand $d_i$, the unmet demand will be subcontracted at a much higher price $P_c$ per unit.   
	If the company produces more than needed, the additional products will be stored with the holding cost as
	$P_e$ per unit. The goal is to minimize the expected total cost,
	}
	{
		\begin{eqnarray}
		 \min_{x \in \mathcal{X}} G(x) &\equiv  P_s x + \mbox{E}_{\mathcal{D}} \left[ \min_{(s,S,u)} \sum_{i=1}^{4} \left(P_r \cdot \sum_{j=1}^5 o_{ij} + P_c \cdot h_i^- + P_e \cdot h_i^+ \right)\right] \label{eq: 2stg genobj} \nonumber  \\
		\mbox{S.t.}
		& h_i^- = \max (d_i- h_{i-1}^+- 5 u, 0)  \quad \forall i
		\label{eq: 2stg gen2} \nonumber \\
		& h_i^+ = \max ( h_{i-1}^++ 5 u-d_i, 0 ) \quad \forall i
		\label{eq: 2stg gen3} \nonumber \\ 
	    & I_{ij}= \max(I_{i,j-1} - u + o_{ij}, 0) \nonumber \\
		& o_{ij}=\begin{cases}
    S - I_{ij}, & \text{$I_{ij} \leq s$},\\
    0, & \text{$I_{ij} > s$}.
  \end{cases} \nonumber \\
  	  		& 0 \leq u \leq {x}/{20}, \quad u  \mbox { is an integer},\nonumber  %, \quad u_i\geq 0, \quad h_i \geq 0\nonumber 
		\label{eq: 2stg gen5} 
	\end{eqnarray} 
	where $P_s$ and $P_r$ are the unit ordering costs for soy and chemical raw material, $h_i^+$ denotes the inventory left at week $i$, $h_i^-$ denotes the unmet demand for week $i$, $o_{ij}$ and $I_{ij}$ denote the raw chemical material ordering decision and the inventory in the $j$-th day of $i$-th week. %Let $O_i \equiv P_r \cdot \sum_{j=1}^5 o_{ij}$ denote the raw chemical material ordering cost in week $i$, where $I_{ij}$ for $i = 1,2,3,4$ and $j =1,2,...,5$ is current chemical raw material inventory on $j$-th day of week $i$, $o_{ij}$ is the amount of ordering for the chemical raw material on $j$-th day of week $i$
	Let starting chemical raw material inventory $I_{10} = 100$.
	Thus, the second-stage production cost at the $i$-th week consists of the ordering cost for the fast-delivery chemical, the subcontract and inventory costs. %Since $(s,S)$ does not have a closed form expression, the entire objective can only be assessed by simulation. 
	Let the soy ordering price $P_s=10$, the chemical ordering price $P_r=5$, the inventory cost $P_e=5$ and the penalty $P_c=100$. 
For the deterministic look-ahead policy, we solve the optimization problem,
		\begin{eqnarray}
		 \min_{x,s,S,u} G_D(x,s,S,u) &\equiv  P_s x + \mbox{E}_{\mathcal{D}} \left[\sum_{i=1}^{4} \left(P_r \cdot \sum_{j=1}^5 o_{ij} + P_c \cdot h_i^- + P_e \cdot h_i^+ \right)\right] .
		\label{eq: 1stg genobj} \nonumber 
	\end{eqnarray} 
We apply Gaussian process-based search approach (GPS) with same setting as \ref{subsec: toyexample} except the standard deviation of the Gaussian process $\sigma_{GPS} = 15$ following Section~(5.2) in \citep{Sun_etal_2014}.
	}

	{Given any $x$, since there is no closed-form of $G(x)$, we use SAA with $N_B$ scenarios to correctly estimate the mean response, which will be used to assess the performance of optimal solutions obtained by the different candidate approaches.}
	\begin{comment}
		\begin{eqnarray}
		\bar{G}({x}) &= P_s {x} + \frac{1}{N_B} \sum_{b=1}^{N_B} \left[ \min_{\mathbf{y} \in \mathcal{Y}({x}), (\mathbf{s},\mathbf{S})} \sum_{i=1}^{4} \left(P_r \cdot \sum_{j=1}^5 o_{ij})+ P_c \cdot u_i + P_e \cdot h_i \right)\right] 
		\label{eq: SAAobj} \\
		\mbox{S.t. } 
		& u_i \geq  d_i^{(b)}- h_{i-1}-5 y_i \quad \forall i,b
		\nonumber \\
		& h_i  \geq h_{i-1}+ 5 y_i-d_i^{(b)} \quad \forall i,b
		\nonumber \\ 
	    & I_{ij}= \max(I_{i,j-1} - y + o_{ij}, 0) \nonumber \\
		& o_{ij}=\begin{cases}
    S - I_{ij}, & \text{$I_{ij} \leq s$},\\
    0, & \text{$I_{ij} > s$}.
  \end{cases} \nonumber \\
  		& 0 \leq y \leq {x}/{20}, \quad y  \mbox { is an integer}, \quad u_i\geq 0, \quad h_i \geq 0\nonumber 
	\end{eqnarray}
	where $d^{(b)}_i$ denotes the $b$-th scenario of demand occurring in the $i$-th week. 	
	\end{comment}
	To determine the proper sample size $N_B$ so that we can accurately estimate the objective $G({x})$, we did a side experiment by running $10$ macro-replications. In each macro-replication, we randomly generate a first-stage action $x$ (with equal probability at each solution). Then, we generate $N_B$ second-stage problems. For each second-stage problem, we exhaustively go through every combination of $\mathbf{y}$ and thus find the corresponding second-stage optimal decisions. For various choices of $N_B$, we record the relative difference error ${|\bar{G}(x)-G(x)|}/{G(x)}$, where $G(x)$ denotes the objective value by $10^5$ second-stage samples. Suppose $10^5$ is large enough and the finite sample estimation error is negligible. The maximum relative
	error for different choices of $N_B$ obtained from $10$ macro-replications is recorded in Table~\ref{table:absolute}.  We observe that $N_B=5000$ has the maximum relative error not exceeding $1\%$. Balancing the computational cost and the accuracy, we use $N_B=5000$ to evaluate the objective value for ${x}$.

	\begin{table}%
	\centering
		\caption{Maximum absolute relative difference for $G(x)$ estimation.
			\label{table:absolute}}{
			\begin{tabular}{|c|c|c|c|c|c|c|} 
				\hline
				$N_B$  & $10^2$ & $5\times10^2$ & $ 10^3$  & $5 \times10^3$ &$ 10^4$ &$5\times10^4$ \\  \hline
				relativeError       & 9.1\% & 3.5\% & 3.2\% & 0.6\% & 0.3\% & 0.2\%  \\ \cline{1-7}
			\end{tabular}}
		\end{table}%
		
			{Moreover, by conducting the side experiment, we obtain the true response surface $G(x)$ against soy order $x$ for cases with $\sigma = 10, 20, 30$. %as shown in Figure \ref{fig_est_optimal_cost_x}. The solid lines denote estimated optimal costs against $x$ for $\sigma= 10, 20, 30$. 
		We get the optimal cost: $G({x}^\star) = 9912$ for $\sigma=10$, $G({x}^\star) = 10625$ for $\sigma=20$, and $G({x}^\star) = 11484$ for $\sigma=30$.}
	
	\begin{comment}		
		\begin{figure}[h!]
			\centering
			\includegraphics[width=0.8\textwidth]{est_opt_cost_vs_x.png}
			\caption{Estimated optimal cost against soy order $x$ for $\sigma = 10, 20, 30$. %The solid lines denote estimated optimal costs against $x$ for $\sigma= 10, 20, 30$ and the dashed lines denote the optimal cost $G({x}^\star)$ for $\sigma= 10, 20, 30$ where $G({x}^\star) = 9912$ for $\sigma=10$, $G({x}^\star) = 10625$ for $\sigma=20$, $G({x}^\star) = 11484$ for $\sigma=30$} \label{fig_est_optimal_cost_x
			}
		\end{figure}
	\end{comment}
	
		\begin{sloppypar}
			{Then, given the same simulation budget, we compare the results obtained from proposed global-local metamodel assisted two-stage OvS, the random sampling SAA approach, and DLH-GPS.} Denote mean and SE of $G(\widehat{x}^\star)$ obtained by our approach as $\mbox{E}(G^\star_g)$ and $\mbox{SE}(G^\star_g)$, respectively. Let $n_m$ represent the number macro-replications, $\mbox{E}(G^\star_g)	\equiv\mbox{E}[G(\widehat{x}^\star)]$ and $\mbox{SE}(G^\star_g)\equiv\mbox{SD}(G^\star_g) / \sqrt{n_m}$, where SD represents the standard deviation (SD) of optimal objective estimate obtained from each macro-replication. Denote mean and SE of $G(\widehat{x}^\star)$ obtained by the random sampling SAA approach
			as $\mbox{E}(G^\star_n)$ and $\mbox{SE}(G^\star_n)$. {Denote those obtained by the DLH-GPS approach
			as $\mbox{E}(G^\star_s)$ and $\mbox{SE}(G^\star_s)$}. 
			In Tables~\ref{table: supplychainsd10}--\ref{table: supplychainsd30}, we record the results obtained from these approaches when $C=600,1000,2000$ and
			$d \sim N(150, 10^2)$, $N(150, 20^2)$ and $N(150, 30^2)$. 
			{They are based on $n_m=100$ macro-replications. 
			%We set $\alpha_0=0.1$ and $n_0=10$. 
			The mean of $G(\widehat{x}^\star)$ tends to decrease as the budget $C$ increases.
			The results in Tables~\ref{table: supplychainsd10}--\ref{table: supplychainsd30} show that our approach significantly outperforms the DLH-GPS and random sampling SAA approaches. It leads to much smaller expected cost and SE.  Moreover, we also see the mean of $G(\widehat{x}^\star)$ increases as the variance of demand increases for all approaches. It agrees with the real-world fact that the higher prediction uncertainty of demand results in higher cost of production scheduling and inventory control. %Notice that when the simulation error is smaller and the design points become more clustered, the numerical issue of GP may impact the performance of proposed approach and DLH-GPS. %Another interesting insight is that the standard error of estimated cost decreases as the variance of demand $d$ increases. Given the similar budget, we could have relative small simulation error in the GP and close design points when the variance of demand is small. Thus, the numerical issue of GP becomes more serious under this situation.
			}

			\end{sloppypar}
					\begin{table}[!ht]
			\centering \small
			\caption{Performance statistics of $G(\widehat{x}^\star)$ when $d_i \sim N(150, 10^2)$
			\label{table: supplychainsd10}}{
				\begin{tabular}{|c|c|c|c|c|c|c|c|c|}
					\hline
					& \multicolumn{2}{c|}{Our approach} &
					\multicolumn{2}{c|}{DLH-GPS} &
					\multicolumn{2}{c|}{\begin{tabular}[c]{@{}c@{}}Random Sampling SAA \\ $(N_1=10,N_2=10)$\end{tabular}} & \multicolumn{2}{c|}{\begin{tabular}[c]{@{}c@{}}Random Sampling SAA \\ $(N_1=10,N_2=20)$\end{tabular}} \\ \cline{2-9}
					& $\mbox{E}(G^\star_g)$  & $\mbox{SE}(G^\star_g)$ &$\mbox{E}(G^\star_s)$  & $\mbox{SE}(G^\star_s)$ & $\mbox{E}(G^\star_n)$  & $\mbox{SE}(G^\star_n)$ & $\mbox{E}(G^\star_n)$  & $\mbox{SE}(G^\star_n)$ \\ \hline
					$C=600$ &   12004 & 151 & 13087 & 224 & 16083 & 428 & 15917 & 428                                                \\ \hline
					$	C=1000$ & 11836 &  148 &12217 & 227 & 15789 & 369 &  15408 & 368  \\ \hline
					$C=2000$ & 11290 & 104 & 11906 & 295 & 15380 & 405 & 14605 &   313 \\ \hline
				\end{tabular}}
			\end{table}
			
			\begin{table}[!ht]
			\centering \small
				\caption{Performance statistics of $G(\widehat{x}^\star)$ when $d_i \sim N(150, 20^2)$
				\label{table: supplychainsd20}}{
					\begin{tabular}{|c|c|c|c|c|c|c|c|c|}
						\hline
						& \multicolumn{2}{c|}{Our approach} &
						\multicolumn{2}{c|}{DLH-GPS} &
						\multicolumn{2}{c|}{\begin{tabular}[c]{@{}c@{}}Random Sampling SAA\\ $(N_1=10,N_2=10)$\end{tabular}} & \multicolumn{2}{c|}{\begin{tabular}[c]{@{}c@{}}Random Sampling SAA\\ $(N_1=10,N_2=20)$\end{tabular}} \\ \cline{2-9}
						& $\mbox{E}(G^\star_g)$  & $\mbox{SE}(G^\star_g)$ &$\mbox{E}(G^\star_s)$  & $\mbox{SE}(G^\star_s)$ & $\mbox{E}(G^\star_n)$  & $\mbox{SE}(G^\star_n)$ & $\mbox{E}(G^\star_n)$  & $\mbox{SE}(G^\star_n)$ \\ \hline
						$C=600$ &  12135 & 173 & 13066 & 217 & 16462 & 427 & 16188 & 421                                                  \\ \hline
						$C=1000$ & 11905 & 97 & 12965& 163& 16342 & 371 &  15753 & 371  \\ \hline
						$C=2000$ & 11771 & 82 & 12272& 199& 16316 & 416 & 15391 & 324  \\ \hline
					\end{tabular}}
				\end{table}
				
				\begin{table}[!ht]
				\centering \small
					\caption{Performance statistics of $G(\widehat{x}^\star)$ when $d_i \sim N(150, 30^2)$
					\label{table: supplychainsd30}}
				    {
						\begin{tabular}{|c|c|c|c|c|c|c|c|c|}
							\hline
							& \multicolumn{2}{c|}{Our approach} &
							\multicolumn{2}{c|}{DLH-GPS} &
							\multicolumn{2}{c|}{\begin{tabular}[c]{@{}c@{}}Random Sampling SAA\\ $(N_1=10,N_2=10)$\end{tabular}} & \multicolumn{2}{c|}{\begin{tabular}[c]{@{}c@{}}Random Sampling SAA\\ $(N_1=10,N_2=20)$\end{tabular}} \\ \cline{2-9}
							& $\mbox{E}(G^\star_g)$  & $\mbox{SE}(G^\star_g)$& $\mbox{E}(G^\star_s)$  & $\mbox{SE}(G^\star_s)$ & $\mbox{E}(G^\star_n)$  & $\mbox{SE}(G^\star_n)$ & $\mbox{E}(G^\star_n)$  & $\mbox{SE}(G^\star_n)$ \\ \hline
							$C=600$ & 12601 & 97 &13507& 217& 17103 & 438 & 16657 & 423 \\ \hline
    							$C=1000$ & 12522 & 125 & 13378& 155& 16684 & 368 & 16063 & 357  \\ \hline
							$C=2000$ & 12236 & 66 & 12929 & 189& 16815 & 415 & 15995 & 322 \\ \hline
						\end{tabular}}
					\end{table}
	
		\begin{sloppypar}			
				{In addition, we define the relative estimation error as $r\Delta G_{\gamma} \equiv \{\mbox{E}[G(\widehat{x}^\star_{\gamma})] -G({x}^\star)\}/{G({x}^\star)}
			= [\mbox{E}(G^\star_{\gamma}) -G({x}^\star)]/{G({x}^\star)}$ with $\gamma=g,s,n$ representing the results obtained by proposed global-local metamodel assisted two-stage OvS, DLH-GPS and random sampling SAA approach respectively. The true optimal solution $x^\star$ is obtained by the side experiments. For the case with $d \sim N(150, 20^2)$,  even with a very tight budget $C=600$, our approach can identify the promising solutions. 
			The average objective value obtained by our approach is {$12,135$} and we have {$r\Delta G_g =14.2\%$, $r\Delta G_{s}=23.0\%$ and $r\Delta G_{n}=54.9\%,52.4\%$ for two random sampling SAA settings} respectively. It shows our approach outperform DLH-GPS by $9\%$ and random sampling SAA by about $40\%$.
			As $C$ grows, the performance of our approach significantly improves. When $C=1000$, our method delivers the optimal solution with the objective value equal to {$11,905$}. We get {$r\Delta G_g =12.0\%$, $r\Delta G_{s}=22.0\%$ and $r\Delta G_{n}=53.8\%,48.3\%$ for two random sampling SAA settings. The results} suggest that our approach outperforms the deterministic look-ahead approach with GPS by $10\%$ and the random sampling SAA method by about $40\%$. 
			The similar performance is also observed when $d \sim N(150, 10^2)$ and $ N(150, 30^2)$.  
			When the variability of the demand increases, it requires a larger simulation budget to search for the optimal solution. 
			{According to Tables~\ref{table: supplychainsd10}--\ref{table: supplychainsd30}, we see the estimate $\mbox{E}(G^\star_g)$ obtained by our algorithm converges to the true optimum faster than the other competitors.}}
		
		\end{sloppypar}

		{Furthermore, in order to better understand the convergence behavior obtained from the proposed approach, we plot $G(\widehat{\mathbf{x}}^{\star(k)})$ with respect to the number of iterations $k$. We record the results from 10 representative runs in Figure~\ref{fig_est_optimal_cost_k}. They indicate that our algorithm can quickly search for the optimal solutions.}
		In addition, the average overhead computational cost from our approach is $0.8$ seconds per simulation run when $C=600,1000$, and $2$ seconds when $C=2000$. Since we consider the situations where each simulation run could be computationally expensive, the overhead is negligible.
		
			\begin{figure}[h!]
	\centering
	\includegraphics[width=0.8\textwidth]{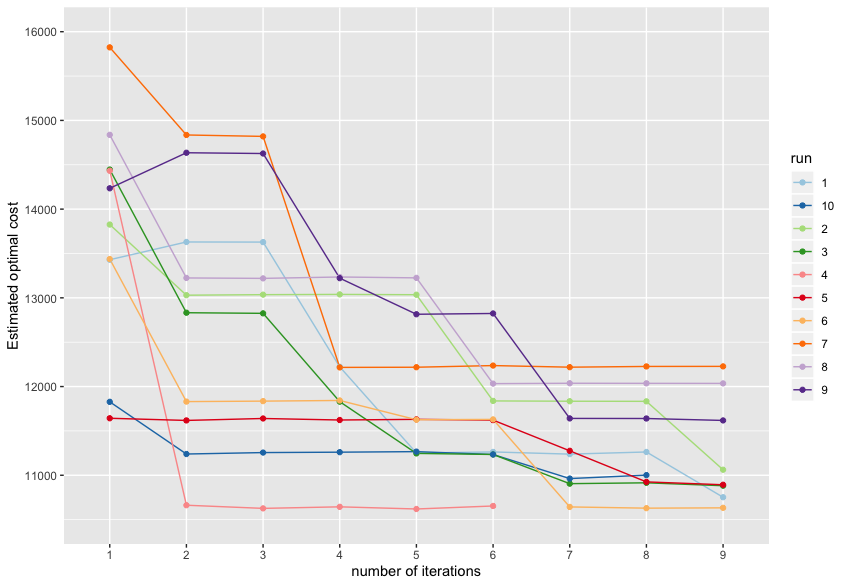}
	\caption{Ten representative sample path plots of $G(\widehat{x}^{\star(k)})$ versus the index of iteration $k$ for the supply chain management example with $\sigma = 20$} \label{fig_est_optimal_cost_k}
\end{figure}

				{In addition, we study the performance of proposed approach as the dimension of second-stage decisions $\mathbf{y}$ increases. 
				Specifically, %instead of making production decision $y$ every month, 
				we modify the supply chain management example by allowing more flexible production and inventory decision making. Let $u_1, u_2$ and $(s_1, S_1), (s_2,S_2)$ represent the daily production and inventory control decisions made in the first and second two weeks. Then we replace the constraint $0\leq u \leq x/20$ in (\ref{eq: 2stg genobj}) with}
{$$0 \leq u_1 + u_2 \leq x/10.$$ We consider the daily inventory review policy $(s_i, S_i)$ for the chemical raw material satisfying $100 \leq s_i < 400$, $200 \leq S_i < 500$ and $s_i < S_i$ with $i =1,2$.
Thus, as we change from $\mathbf{y}=(u,s,S)$ to $\mathbf{y}=(u_1,u_2,s_1,S_1,s_2,S_2)$, the dimension and complexity of second-stage optimization problem also increase. Under the same parameter setting as Section~\ref{subsec:supplyChainExample}, the results obtained from our algorithm, DLH-GPS and random sampling SAA approaches are summarized in Table~\ref{table: extended supplychainsd20}. We can see that the proposed global-local metamodel assisted two-stage OvS still delivers the optimal solution with smaller expected cost in all budget levels than the other two algorithms. %It indicates that the performance of our approach is robust to the size of problem. 
%Further comparing with Table~\ref{table: supplychainsd20},  we see our approach is robust to the size of problem. % because our algorithm results under extended second-stage setting still have near optimal results about mean and SE of $G(\widehat{x}^\star)$ in all three budget levels although the second-stage dimension is greatly increased.
}
								\begin{table}[!h]
			\centering
				\caption{Performance statistics of $G(\widehat{x}^\star)$ with the increased size of second stage when $d_i \sim N(150, 20^2)$
				\label{table: extended supplychainsd20}}{
					\begin{tabular}{|c|c|c|c|c|c|c|c|c|}
						\hline
						& \multicolumn{2}{c|}{Our approach} &
						\multicolumn{2}{c|}{DLH-GPS} &
						\multicolumn{2}{c|}{\begin{tabular}[c]{@{}c@{}}Random Sampling SAA \\ $(N_1=10,N_2=10)$\end{tabular}} & \multicolumn{2}{c|}{\begin{tabular}[c]{@{}c@{}}Random Sampling SAA \\ $(N_1=10,N_2=20)$\end{tabular}} \\ \cline{2-9}
						& $\mbox{E}(G^\star_g)$  & $\mbox{SE}(G^\star_g)$ &$\mbox{E}(G^\star_s)$  & $\mbox{SE}(G^\star_s)$ & $\mbox{E}(G^\star_n)$  & $\mbox{SE}(G^\star_n)$ & $\mbox{E}(G^\star_n)$  & $\mbox{SE}(G^\star_n)$ \\ \hline
						$C=600$ &  12948 & 221 & 13532 & 362 & 17007 & 443 & 16730 & 440                                             \\ \hline
    						$C=1000$ & 12621 & 292 & 13339& 353& 16687 & 405 & 16184   & 408  \\ \hline
    						$C=2000$ & 12606 & 193 & 13282 &297 & 16510 & 385 & 16160 &  439 \\ \hline
					\end{tabular}}
				\end{table}

\subsection{The effect of $\alpha_0$ and $n_0$}	
\label{subsec:sensitivity}
Comparing with one-stage optimization studied in the simulation literature, in two-stage stochastic programming, the optimality gap from the second-stage optimization  impacts the search performance. Thus, we study the effect of $\alpha_0$ in this section. The relative EI threshold $\alpha_0$ impacts the exploitation and exploration trade-off. When $\alpha_0$ is large, our proposed algorithm turns to put less effort for second-stage optimization for each scenario. 
When $\alpha_0$ is small, we could spend more efforts to search for the second-stage optimal solution for each scenario.

	{Here, we empirically study the effect of $\alpha_0$.
		Tables~\ref{table: toytol} and \ref{table: biotol} provide the results for the examples in Sections~\ref{subsec: toyexample} and \ref{subsec:supplyChainExample} when $\alpha_0=0.05,0.1,0.2$. Table~\ref{table: toytol} indicates the choice of $\alpha_0$ has relative less impact on the optimization for the two-stage linear optimization problem.
	Table~\ref{table: biotol} provides mean and SE of $G(\widehat{x}^\star)$ for the supply chain management example when $d_i\sim N(150,20^2)$.
	 All three settings can deliver near-optimal $\widehat{x}^\star$ while the SD of $\widehat{x}^\star$ obtained from $\alpha_0=0.1$ is slightly smaller than the other two settings. %Hence, $\alpha_0=0.1$ is recommended.
	%It suggests $\alpha_0 = 0.1$ since its estimate $\mbox{E}(G^\star_g)$ and $\mbox{SE}(G^\star_g)$ are almost always smallest for $\alpha_0=0.05,0.1,0.2$ and different budget level. 
	%In general, the choice of $\alpha_0$ depends more on the specific problem, such as  problem complexity and simulation budget. %For more complex second-stage optimization, we recommend a smaller $\alpha_0$ while for more complex first-stage, a larger $\alpha_0$ is preferred.
}
	
	%Then, we examine the impact of $n_0$. Since we mainly consider tight budget cases, we record the results with $n_0=10,20$ in Table~\ref{table: toynpath} and the objective value estimation error in Table~\ref{table: toynpathacc}. The results are similar, while the estimation accuracy is slightly worse when $n_0=20$. %In this case, it has a slightly larger relative estimation error. 	\textcolor{red}{In the supply chain management example, however, we see slightly better result for $n_0=20$, shown in Table~\ref{table: bionpath}.}
	
	\begin{table}[!h]
	\centering
		\caption{For the two-stage linear optimization problem, mean and SD of $\widehat{x}^\star$
			when $\alpha_0=0.05,0.1,0.2$.
		\label{table: toytol}}{
		% Please add the following required packages to your document preamble:
		% \usepackage{multirow}
		% Please add the following required packages to your document preamble:
		% \usepackage{multirow}
		\begin{tabular}{|c|c|c|c|c|c|c|}
			\hline
			\multirow{2}{*}{} & \multicolumn{2}{c|}{\multirow{2}{*}{\begin{tabular}[c]{@{}l@{}}$\alpha_0=0.05$ \\ \end{tabular}}} & \multicolumn{2}{c|}{\multirow{2}{*}{\begin{tabular}[c]{@{}l@{}}$\alpha_0=0.1$\\ \end{tabular}}} & \multicolumn{2}{c|}{\multirow{2}{*}{\begin{tabular}[c]{@{}l@{}}$\alpha_0=0.2$\\ \end{tabular}}} \\
			& \multicolumn{2}{l|}{}                                                                         & \multicolumn{2}{l|}{}                                                                       & \multicolumn{2}{l|}{}                                                                       \\ \hline
			& mean          & SD             & mean         & SD        & mean            & SD                                           \\ \hline
			$C=600$     & 2.90         & 0.07      & 2.91         & 0.07           & 2.90            & 0.07                                       \\ \hline
			$C=1000$    & 2.91         & 0.07       & 2.92         & 0.05           & 2.93            & 0.06                                         \\ \hline
			$C=2000$   & 2.93        & 0.06     & 2.94         & 0.04              & 2.93            & 0.07                                         \\ \hline
		\end{tabular}}
	\end{table}
	
	\begin{comment}
	\begin{table}[!ht]
	\centering
		\caption{Mean and SD of relative estimation error $rE$ when $\alpha_0=0.05,0.1,0.2$ in $\%$
		\label{table: toytolacc}} {
		\begin{tabular}{|c|c|c|c|c|c|c|}
			\hline
			& \multicolumn{2}{c|}{$\alpha_0=0.05$} 	& \multicolumn{2}{c|}{$\alpha_0=0.1$} & \multicolumn{2}{c|}{$\alpha_0=0.2$} \\ \hline
			& mean & SD	& mean & SD & mean & SD \\ \hline
			$C=600$ & 5.1 & 2.9 & 5.4   & 3.0 & 5.2 & 3.0 \\ \hline
			$C=1000$ & 4.7 & 2.8 & 5.0 &  2.9 & 4.6 & 3.0 \\ \hline
			$C=2000$ &4.3 & 3.0 & 4.0 & 3.0 & 4.2 & 3.1 \\ \hline
		\end{tabular} }
	\end{table}
\end{comment}
	\begin{comment}
	\begin{table}[!h]
	\centering
		\caption{Mean and SD of relative estimation error $rE$ when $n_0=10,20$ in $\%$
		\label{table: toynpathacc}}{
		\begin{tabular}{|c|c|c|c|c|}
			\hline
			& \multicolumn{2}{c|}{$n_0=10$} & \multicolumn{2}{c|}{$n_0=20$} \\ \hline
			& mean & SD & mean & SD \\ \hline
			$C=600$ & 5.4   & 3.0 & 6.9 & 1.8 \\ \hline
			$C=1000$ & 5.0 &  2.9 & 6.8 & 2.0 \\ \hline
			$C=2000$ & 4.0 & 3.0 & 6.0 & 2.0 \\ \hline
		\end{tabular}}
	\end{table}
	\end{comment}
	\begin{table}[!t]
	    \centering
		\caption{For the supply chain management example, mean and SE of $\widehat{x}^\star$ for $\alpha_0=0.05,0.1,0.2$ when $n_m=100$, $n_0=10$ and $d_i \sim N(150, 20^2)$
		\label{table: biotol}}{
		% Please add the following required packages to your document preamble:
		% \usepackage{multirow}
		% Please add the following required packages to your document preamble:
		% \usepackage{multirow}
		\begin{tabular}{|c|c|c|c|c|c|c|}
			\hline
			\multirow{2}{*}{} & \multicolumn{2}{c|}{\multirow{2}{*}{\begin{tabular}[c]{@{}l@{}}$\alpha_0=0.05$ \\ \end{tabular}}} & \multicolumn{2}{c|}{\multirow{2}{*}{\begin{tabular}[c]{@{}l@{}}$\alpha_0=0.1$\\ \end{tabular}}} & \multicolumn{2}{c|}{\multirow{2}{*}{\begin{tabular}[c]{@{}l@{}}$\alpha_0=0.2$\\ \end{tabular}}} \\
			& \multicolumn{2}{l|}{}                                                                         & \multicolumn{2}{l|}{}                                                                       & \multicolumn{2}{l|}{}                                                                       \\ \hline
			& $\mbox{E}(G^\star_g)$          & $\mbox{SE}(G^\star_g)$             & $\mbox{E}(G^\star_g)$         & $\mbox{SE}(G^\star_g)$        & $\mbox{E}(G^\star_g)$            & $\mbox{SE}(G^\star_g)$                                           \\ \hline
			$C=600$ & 13230   & 269        &12135         & 173           & 12343            & 180                                     \\ \hline
			$C=1000$  & 13040 &   133    & 11905        & 97           &   12063       &    158                                    \\ \hline
			$C=2000$   &   12093      &  153    & 11771          & 82              &   11724        &   137                                      \\ \hline
		\end{tabular}}
	\end{table}
	
	We also conduct the experiments studying the impact of $n_0$ by using both examples. Tables~\ref{table: toynpath} and~\ref{table: bionpath} provide the corresponding results when $n_0=10,20$ under the same parameter setting. Table~\ref{table: toynpath} doesn't show significant difference between $n_0=10 \mbox{ and } 20$ in all three budget levels for the two-stage linear optimization problem. In the supply chain management example, however, we see better result for $n_0=20$ according to Table~\ref{table: bionpath}. {In general, $n_0$ plays an important role in the trade-off between exploitation and exploration. When $\pmb{\xi}$ has higher uncertainty and the complexity of the response surface $G(\mathbf{x})$ is lower, we could require a larger number of initial scenarios, $n_0$. %When random input variate $\pmb{\xi}$ has less uncertainty, it could require a smaller number of initial scenarios, $n_0$. 
Without strong prior information, we would recommend starting with $n_0=10$.}
	%when either the uncertainty induced by $\pmb{\xi}$ is relatively small or the uncertainty has less impact on optimizing the expected cost. Otherwise a larger number of initial scenarios $n_0$ could be expected.
	
	\begin{table}[!t]
	\centering
		\caption{For the two-stage linear optimization problem,  Mean and SD of $\widehat{x}^\star$ when $n_0=10,20$
			\label{table: toynpath}}{
			\begin{tabular}{|c|c|c|c|c|}
				\hline
				& \multicolumn{2}{c|}{$n_0=10$} & \multicolumn{2}{c|}{$n_0=20$} \\ \hline
				& mean & SD & mean & SD \\ \hline
				$C=600$ & 2.91  & 0.07 & 2.91 & 0.07 \\ \hline
				$C=1000$ & 2.92 & 0.05 & 2.94 & 0.07 \\ \hline
				$C=2000$ & 2.94 & 0.04 & 2.94 & 0.08 \\ \hline
		\end{tabular}}
	\end{table}

% 	\begin{table}[h]
% 		\caption{For the two-stage linear optimization problem, Mean and SD of relative estimation error $rE$ when $n_0=10,20$ in $\%$
% 		\label{table: toynpathacc}}{
% 		\begin{tabular}{|c|c|c|c|c|}
% 			\hline
% 			& \multicolumn{2}{c|}{$n_0=10$} & \multicolumn{2}{c|}{$n_0=20$} \\ \hline
% 			& mean & SD & mean & SD \\ \hline
% 			$C=600$ & 5.4   & 3.0 & 6.9 & 1.8 \\ \hline
% 			$C=1000$ & 5.0 &  2.9 & 6.8 & 2.0 \\ \hline
% 			$C=2000$ & 4.0 & 3.0 & 6.0 & 2.0 \\ \hline
% 		\end{tabular}}
% 	\end{table}
	
		\begin{table}[!t]
		\centering
		\caption{For the supply chain management example, Mean and SD of $\widehat{x}^\star$ for $n_0=10,20$ when $n_m=100$, $\alpha_0=0.1$ and $d_i \sim N(150, 20^2)$.
		\label{table: bionpath}}{
		\begin{tabular}{|c|c|c|c|c|}
			\hline
			& \multicolumn{2}{c|}{$n_0=10$} & \multicolumn{2}{c|}{$n_0=20$} \\ \hline
			& $\mbox{E}(G^\star_g)$ & $\mbox{SE}(G^\star_g)$ & $\mbox{E}(G^\star_g)$ & $\mbox{SE}(G^\star_g)$ \\ \hline
			$C=600$ & 12135  & 173 & 11711& 83 \\ \hline
			$C=1000$ & 11905 & 97 & 11666 & 132 \\ \hline
			$C=2000$ & 11771 & 82 & 11192& 66  \\ \hline
		\end{tabular}}
	\end{table}

					\section{Conclusion}
				
					\label{sec:conclusion}

					In this paper, we {propose a global-local metamodel assisted two-stage OvS that can efficiently employ the tight simulation budget to solve the stochastic programming for complex systems}. In particular, for each visited first-stage decision, we construct a local metamodel which allows us to simultaneously solve all second-stage optimization problems sharing the same first-stage decision. Then, based on the second-stage optimization results, we construct a global metamodel accounting for the finite sampling error from SAA and the second-stage optimality gap. Assisted by the global-local metamodel, we develop a two-stage optimization approach that can efficiently employ the simulation budget to iteratively solve for the optimal first- and second-stage decisions. The empirical studies demonstrate that our algorithm delivers superior and stable optimal decisions. The proposed methodology may lay the groundwork for future research on two-stage risk-averse stochastic simulation optimization.

\section*{Acknowledgments}
%\begin{acks}
	The authors acknowledge helpful discussion with Barry L. Nelson.
%\end{acks}

% Appendix
% Appendix
\appendix
\section{Nomenclature}
\nomenclature[E]{$n_m$}{number of macro replications.}
\nomenclature[E]{$C$}{simulation budget.}
\nomenclature[E]{$N_1$}{number of observed first-stage solutions for the random sampling SAA approach.}
\nomenclature[E]{$N_2$}{number of scenarios generated at each visited first-stage decision for the random sampling SAA approach.}
\nomenclature[E]{$N_B$}{number of scenarios used to accurately estimate the objective $G(\mathbf{x})$ at any $\mathbf{x}$.}
\nomenclature[E]{$r\Delta G$}{the relative error between estimated optimal cost and and true optimum.}
\nomenclature[E]{$d_i$}{the aggregated clinical demand occurring in the $i$-th week with $i=1,2,3,4$.}
\nomenclature[E]{$(s,S)$}{the daily review $(s,S)$ policy for the raw chemical material.}
\nomenclature[E]{$P_s$}{unit ordering costs for soy.}
\nomenclature[E]{$P_r$}{unit ordering costs for chemical raw material.}
\nomenclature[E]{$P_c$}{unit penalty cost for unmet demand.}
\nomenclature[E]{$P_e$}{unit inventory holding cost.}
\nomenclature[E]{$I_{ij}$}{the raw chemical material inventory in the $j$-th day of $i$-th week.}
\nomenclature[E]{$o_{ij}$}{raw chemical material ordering decision in the $j$-th day of $i$-th week.}
\nomenclature[E]{$h^-_{i}$}{unmet demand for the $i$-th week.}
\nomenclature[E]{$h^+_i$}{the inventory left at week $i$.}
\nomenclature[E]{$\sigma_{GPS}$}{spatial standard deviation of the Gaussian process metamodel in the DLH-GPS algorithm.}
%\nomenclature[E]{$\lambda_{i}(\mathbf{x})$}{the vector of weight functions in DLH-GPS algorithm.}
%\nomenclature[E]{$\gamma(\mathbf{x}_1,\mathbf{x}_2)$}{correlation function used by DLH-GPS algorithm.}
\nomenclature[E]{$\mbox{E}(G^\star_g)$}{mean of $G(\widehat{x}^\star)$ obtained by the proposed global-local metamodel assisted two-stage OvS algorithm.}
\nomenclature[E]{$\mbox{SE}(G^\star_g)$}{standard error of $G(\widehat{x}^\star)$ obtained by the proposed global-local metamodel assisted two-stage OvS algorithm.}
\nomenclature[E]{$\mbox{E}(G^\star_s)$}{mean of $G(\widehat{x}^\star)$ obtained by the DLH-GPS.}
\nomenclature[E]{$\mbox{SE}(G^\star_s)$}{standard error of $G(\widehat{x}^\star)$ obtained by the DLH-GPS.}
\nomenclature[E]{$\mbox{E}(G^\star_n)$}{mean of $G(\widehat{x}^\star)$ obtained by the random sampling SAA approach.}
\nomenclature[E]{$\mbox{SE}(G^\star_n)$}{standard error of $G(\widehat{x}^\star)$ obtained by the random sampling SAA approach.}

\nomenclature[F]{$G(x)$}{objective function of the two-stage stochastic programming.}
\nomenclature[F]{$q(\mathbf{x}, \mathbf{y},\pmb{\xi})$}{second-stage response function}
\nomenclature[F]{$F(\pmb{\xi})$}{probability distribution function of $\pmb{\xi}$, characterizing the prediction uncertainty.}
\nomenclature[F]{$\bar{G}^c(\mathbf{x})$}{SAA estimator for any visited $\mathbf{x}$ with $N(\mathbf{x})$ number of scenarios.}
\nomenclature[F]{$\delta(\mathbf{x}_i,\pmb{\xi}_{ij})$}{optimality gap $\delta(\mathbf{x}_i,\pmb{\xi}_{ij})\equiv q_{\mathbf{x}_i}(\widehat{\mathbf{y}}^\star_{ij},\pmb{\xi}_{ij}) - q_{\mathbf{x}_i}(\mathbf{y}^\star_{ij},\pmb{\xi}_{ij})$}
\nomenclature[F]{$M(\mathbf{z})$}{GP model with mean zero.}
\nomenclature[F]{$\mathcal{W}(\cdot)$}{the global GP metamodel, defined as $\mathcal{W}(\cdot)\equiv \mu_0+W(\cdot)$}
\nomenclature[F]{$\mathcal{GP}_{\mathbf{x}_i}(\widehat{q}_{\mathbf{x}_i}(\bm Z), s^2_{\mathbf{x}_i}(\bm Z))$}{the local Gaussian process metamodel at the prediction points $\mathbf{Z}$.}
\nomenclature[F]{$\mbox{E}_{\widetilde{q}_{\mathbf{x}_i}(\mathbf{y}_{ij},\pmb{\xi}_{ij})}\left[\mathrm{I}(\mathbf{y}_{ij}, \pmb{\xi}_{ij})\right]$}{the expected improvement (EI) for any untried point $\mathbf{y}_{ij} \in \mathcal{Y}(\mathbf{x}_i)$ given the current optimum $\widehat{\mathbf{y}}^\star_{ij}$.}
\nomenclature[F]{$\mathcal{GP}_o(\widehat{G}(\bm X_\star), s^2(\bm X_\star))$}{the global Gaussian process metamodel at prediction points $\bm X_\star$ with mean $\widehat{G}(\bm X_\star)$ and variance $s^2(\bm X_\star)$.}
\nomenclature[F]{$f^{(k+1)}(\mathbf{x})$}{sampling distribution used for generating a new set of first-stage design points for the next $(k+1)$-th iteration.}
\nomenclature[F]{$R(\mathbf{z}-\mathbf{z^\prime} | \pmb{\phi})$}{product-form Gaussian correlation function of $(\mathbf{z}, \mathbf{z^\prime})$ for local metamodel.}
\nomenclature[F]{$r(\mathbf{x}-\mathbf{x^\prime} | \pmb{\phi})$}{product-form Gaussian correlation function of $(\mathbf{z}, \mathbf{z^\prime})$ for global metamodel.}
\nomenclature[F]{$\Phi$}{CDF of standard normal distribution}
\nomenclature[F]{$\varphi$}{PDF of standard normal distribution}
% \nomenclature[F]{$R(\mathbf{z}-\mathbf{z^\prime} | \pmb{\phi})$}{the product-form Gaussian correlation function, where $d$ is the dimension of $\mathbf{z}$ and the parameters $\pmb{\phi}=(\phi_1,\ldots,\phi_d)$ control the spatial dependence.}

\nomenclature[V]{$\mathbf{x}$}{first-stage decision}
\nomenclature[V]{$\mathbf{y}$}{second-stage decision}
\nomenclature[V]{$\mathbf{x^\star}$}{optimal first-stage decision}
\nomenclature[V]{$\mathbf{y^\star}$}{optimal second-stage decision}
\nomenclature[V]{$\mathbf{\pmb{\xi}}$}{random input/scenarios}
\nomenclature[V]{$N(\mathbf{x})$}{number of scenarios assigned to design point at $\mathbf{x}$}
\nomenclature[V]{$\pmb{\phi}$}{parameter of correlation function controlling the spatial dependence.}
\nomenclature[V]{$\widetilde{q}_{\mathbf{x}_i}(\cdot,\cdot)$}{the posterior sample from the estimated GP, i.e. $\widetilde{q}_{\mathbf{x}_i}(\cdot,\cdot) \sim
		\mathcal{GP}_{\mathbf{x}_i}(\widehat{q}_{\mathbf{x}_i}(\cdot),s_{\mathbf{x}_i}^2(\cdot))$.}
\nomenclature[V]{$\beta_0$}{the global trend in local metamodel.}
\nomenclature[V]{$\mu_0$}{the global trend in global metamodel.}
\nomenclature[V]{$\tau^2$}{the spatial variance of correlation function of local metamodel.}
\nomenclature[V]{$\sigma^2$}{the spatial variance of correlation function of global metamodel.}
\nomenclature[V]{$\alpha_0$}{the initial relative EI threshold.}
%\nomenclature[V]{$\epsilon(\mathbf{x}_i,\pmb{\xi}_{ij})$}{the random error representing the deviation of $G_j(\mathbf{x}_i)$ from the expected cost $G(\mathbf{x}_i)$.}
%\nomenclature[V]{$e(\mathbf{x}_i, \pmb{\xi}_{ij})$}{the zero-mean random variable characterizing the variability of $\delta(\mathbf{x}_i,\pmb{\xi}_{ij})$ given $\pmb{\xi}_{ij}$, which is induced by the local metamodel uncertainty}
\nomenclature[V]{$V$}{$ K\times K$ diagonal covariance matrix of $\bar{\mathbf{G}}_{\mathcal{P}_o}=(\bar{G}(\mathbf{x}_1),\bar{G}(\mathbf{x}_2),\ldots,\bar{G}(\mathbf{x}_{K}))^\prime$, where $\mathcal{P}_o\equiv\{\mathbf{x}_1,\mathbf{x}_2,\ldots,\mathbf{x}_{K}\}$ denotes the global metamodel design points.}
\nomenclature[V]{$\alpha^{(k)}(\mathbf{x}_i)$}{the relative EI threshold for any $\mathbf{x}_i$ controlling the stopping criteria for second-stage optimization. %For the $t$-th visited first-stage action $\mathbf{x} \in \mathcal{X}$ with $t>1$, the relative EI threshold for second-stage optimization is set to be $\alpha_0 g^{-\frac{t-1}{2}}$.
}
\nomenclature[V]{$n_0$}{initial number of scenarios allocated to each visited design point.}
\nomenclature[V]{$\mathbf{Q}_{\mathcal{P}_{\mathbf{x}_i}}$}{the corresponding simulation outputs at given design points $\mathcal{P}_{\mathbf{x}_i}$.}

\nomenclature[S]{$\mathcal{X}$}{feasible set of first-stage decisions}
\nomenclature[S]{$\mathcal{Y}(\mathbf{x})$}{feasible set of second-stage decisions depending on $\mathbf{x}$}
\nomenclature[S]{$\Xi$}{the scenario set containing all possible $\pmb{\xi}$.}
\nomenclature[S]{$D_{\mathbf{x}}^{(k)}$}{new first-stage design point set generated at $k$-th iteration.}
\nomenclature[S]{$S_{\mathbf{x}}^{(k)}$}{first-stage design point set accumulated until the $k$-th iteration, i.e., $\mathcal{S}_{\mathbf{x}}^{(k)}=\mathcal{S}_{\mathbf{x}}^{(k-1)}\bigcup D_{\mathbf{x}}^{(k)}$.}
\nomenclature[S]{$\mathcal{U}_{\mathbf{x}_i}$}{the set of finished second-stage scenarios at $\mathbf{x}_i$.}
\nomenclature[S]{$\mathcal{P}_{\mathbf{x}_i}$}{the set of second-stage design points at $\mathbf{x}_i$.}
\nomenclature[S]{$\mathcal{P}_o$}{the set of first-stage design points.}

\printnomenclature
% Electronic Appendix
%\elecappendix
\section{Proof of Proposition~\ref{proposition:errorlowerbound}}

\begin{proposition}
\label{proposition:errorlowerbound}
Let $\underline{d}(\mathbf{x}) \equiv \min_{\mathbf{x}^\prime\in \mathcal{X}}\{ |\mathbf{x}- \mathbf{x}^\prime |\}$. Then, we have the prediction variance $s^2(\mathbf{x})\geq  \tau^2[1-r(\underline{d}(\mathbf{x}))]^2$ and the lower bound is zero if $\underline{d}(\mathbf{x})=0$.
\end{proposition}
\begin{proof}
Let $\pmb{\omega} = (\Sigma+V)^{-1}\left[\Sigma(\mathbf{x},\cdot) + \mathbf{1}\left(\frac{1-\mathbf{1}^\prime(\Sigma +V)^{-1}\Sigma(\mathbf{x},\cdot) }{\mathbf{1}^\prime(\Sigma+V)^{-1}
	\mathbf{1}}\right)\right]$, $\lambda=\frac{1-\mathbf{1}^\prime(\Sigma +V)^{-1}\Sigma(\mathbf{x},\cdot) }{\mathbf{1}^\prime(\Sigma+V)^{-1}
	\mathbf{1}}$ and $\mathbf{\eta}=1-\mathbf{1}^\prime
	(\Sigma+V)^{-1} \Sigma(\mathbf{x},\cdot).$ We first show that $\pmb{\omega}$ has nice property $\sum_i^K \omega_i = 1$. 
	Since $\Sigma(\mathbf{x},\cdot)^\prime(\Sigma+V)^{-1}\mathbf{1}$ and $\mathbf{1}^\prime(\Sigma+V)^{-1}
	\mathbf{1}$ are scalars, we have
	\begin{eqnarray}
\pmb{\omega}^\prime\mathbf{1}&=&\left[\Sigma(\mathbf{x},\cdot)^\prime + \mathbf{1}^\prime\left(\frac{1-\mathbf{1}^\prime(\Sigma +V)^{-1}\Sigma(\mathbf{x},\cdot) }{\mathbf{1}^\prime(\Sigma+V)^{-1}
	\mathbf{1}}\right)\right](\Sigma+V)^{-1}\mathbf{1}
	\nonumber\\
	&=& \Sigma(\mathbf{x},\cdot)^\prime(\Sigma+V)^{-1}\mathbf{1} + \left(\frac{1-\mathbf{1}^\prime(\Sigma +V)^{-1}\Sigma(\mathbf{x},\cdot) }{\mathbf{1}^\prime(\Sigma+V)^{-1}
	\mathbf{1}}\right) \mathbf{1}^\prime(\Sigma+V)^{-1}\mathbf{1}\nonumber\\
	&=&1.\nonumber
	\end{eqnarray}
	
	Let $\omega_i$ denote $i$-th element of $\pmb{\omega}$. Then, for the prediction variance, we derive the lower bound as follows,
\begin{align}
	s^2(\mathbf{x})  & =  \tau^2 - \Sigma(\mathbf{x},\cdot)^\prime [\Sigma+V]^{-1} \Sigma(\mathbf{x},\cdot)
 	+\mathbf{\eta}^\prime [\mathbf{1}^\prime(\Sigma+V)^{-1}
 	\mathbf{1}]^{-1}\mathbf{\eta}
	\nonumber\\
	&= \tau^2 - \Sigma(\mathbf{x},\cdot)^\prime(\Sigma+V)^{-1}\Sigma(\mathbf{x},\cdot) -\lambda\left[1-\mathbf{1}^\prime(\Sigma + V)^{-1}\Sigma(\mathbf{x},\cdot)\right] \nonumber\\
	&=\tau^2 +\lambda - \Sigma(\bm x,\cdot)^\prime(\Sigma+V)^{-1}\Sigma(\mathbf{x},\cdot) -\lambda\mathbf{1}^\prime(\Sigma + V)^{-1}\Sigma(\bm x,\cdot)\nonumber\\
 	&=\tau^2 +\lambda - \pmb{\omega}^\prime\Sigma(\mathbf{x},\cdot)\nonumber\\
 	&=\tau^2+\left[\Sigma(\mathbf{x},\cdot)^\prime+\lambda\mathbf{1}^\prime\right]\pmb{\omega} - 2\pmb{\omega}^\prime\Sigma(\mathbf{x},\cdot)
	\nonumber\\
	&=\tau^2+\pmb{\omega}^\prime(\Sigma+V)\pmb{\omega}-2\pmb{\omega}^\prime\Sigma(\mathbf{x},\cdot)\nonumber\\
	&\geq\tau^2+\pmb{\omega}^\prime\Sigma\pmb{\omega}-2\pmb{\omega}^\prime\Sigma(\mathbf{x},\cdot) \nonumber\\
	& =
	\tau^2 \left(1 + \sum_{i=1}^K\sum_{j=1}^K\omega_i\omega_j r(\mathbf{x}_i,\mathbf{x}_j)-2\sum_{i=1}^K\omega_i r(\mathbf{x},\mathbf{x}_i)\right) \nonumber\\
	&\geq\tau^2 \left(1 + \sum_{i=1}^K\sum_{j=1}^K\omega_i\omega_j r(\mathbf{x},\mathbf{x}_i)r(\mathbf{x},\mathbf{x}_j)-2\sum_{i=1}^K\omega_i r(\mathbf{x},\mathbf{x}_i)\right) \label{eq.mid201}\\
	&\geq \tau^2\left[1-\sum_{i=1}^K\omega_i r(\mathbf{x},\mathbf{x}_i)\right]^2\nonumber\\
	&\geq\tau^2[1-r(\underline{d}(\mathbf{x}))]^2.
\label{eq.reformulation} 
	\end{align}
Step~(\ref{eq.mid201}) follows because %Note that Gaussian 
	the correlation function $r(\cdot)$ satisfies $r(\mathbf{x}_1-\mathbf{x}_2)\geq r(\mathbf{x}-\mathbf{x}_1)r(\mathbf{x}-\mathbf{x}_2)$ for any $\mathbf{x},\mathbf{x}_1, \mathbf{x}_2 \in \mathcal{X}$. Step~(\ref{eq.reformulation}) follows by applying $\sum_i^K \omega_i = 1$. 
Notice that $\tau^2[1-r(\underline{d}(\mathbf{x}))]^2 =0$ if $\mathbf{x}$ is one of design points, i.e., $\underline{d}(\mathbf{x}) = \min_{\mathbf{x}^\prime\in \mathcal{X}}\{| \mathbf{x}- \mathbf{x}^\prime |\} = 0$.
% \begin{equation}
% 	    \tau^2+\pmb{\omega}^\prime\Sigma\pmb{\omega}-2\pmb{\omega}^\prime\Sigma(\bm X,\cdot) = \tau^2 - \Sigma(\bm X_\star,\cdot)^\prime [\Sigma+V]^{-1} \Sigma(\bm X_\star,\cdot)
%  	+\mathbf{\eta}^\prime [\mathbf{1}^\prime(\Sigma+V)^{-1}
%  	\mathbf{1}]^{-1}\mathbf{\eta}  
% 	\end{equation}
\end{proof}

\section{Proof of Lemmas~\ref{proposition1:revisit},~\ref{proposition2:samplepath_revisit} and~\ref{proposition3:liminf}}
%\medskip
%\begin{screenonly}
Lemma~\ref{proposition1:revisit},~\ref{proposition2:samplepath_revisit} and~\ref{proposition3:liminf} will facilitate the following proofs of convergence.
	\begin{lemma} \label{proposition1:revisit}
		Suppose that the global-local metamodel-assisted OvS approach proposed in Algorithm~\ref{Global} is used to solve the two-stage optimization problem~(\ref{eq.intro1}). For any $\mathbf{x} \in \mathcal{X}$, if $f^{(k)}(\mathbf{x}) \geq c$ infinitely often (i.o.) for some constant $c>0$, then $N^{(k)}(\mathbf{x}) \rightarrow \infty$ w.p.1 as $k\rightarrow \infty$.
	\end{lemma}
\begin{proof} The proof follows by applying Lemma~(1) in \citep{Sun_etal_2014}.
\end{proof}
	\begin{lemma} \label{proposition2:samplepath_revisit}
		Suppose that the global-local metamodel-assisted OvS approach proposed in Algorithm~\ref{Global} is used to solve the two-stage optimization problem~(\ref{eq.intro1}). Then $N^{(k)}(\widehat{\mathbf{x}}^{\star(k)})\rightarrow \infty$ w.p.1 as $k\rightarrow \infty$.
	\end{lemma}
\begin{proof} The proof follows by applying Lemma~(2) in \citep{Sun_etal_2014}.
\end{proof}

	\begin{lemma} \label{proposition3:liminf}
		Suppose that the global-local metamodel-assisted OvS approach proposed in Algorithm~\ref{Global} is used to solve the two-stage optimization problem~(\ref{eq.intro1}) with finite first- and second-stage decision spaces, i.e., $|\mathcal{X}|<\infty$ and $|\mathcal{Y}(\mathbf{x})|<\infty$ for any $\mathbf{x}\in \mathcal{X}$. Then we have
	$\lim \inf_{k \rightarrow \infty} \bar{G}(\widehat{\mathbf{x}}^{\star (k)}) > G(\mathbf{x}^\star) - \nu$ for any $\nu >0$ w.p.1. It's equivalent to show
	\[
	\mbox{Pr}\{\bar{G}(\widehat{\mathbf{x}}^{\star (k)}) <  G(\mathbf{x}^\star) - \nu \quad \mbox{i.o.}\}=0.
	\]
	\end{lemma}
	
\begin{proof} 
Following the proof of Lemma (3) in \citep{Sun_etal_2014}, at the end of each iteration, for any $\mathbf{x}\in \mathcal{X}$, we add additional scenarios to $\mathbf{x}$ such that the number of scenarios at $\mathbf{x}$ is at least $N^{(k)}(\widehat{\mathbf{x}}^{\star(k)})$. Let $n^{(k)}(\mathbf{x}) = \max_{\mathbf{x}\in\mathcal{X}}\left(N^{(k)}(\widehat{\mathbf{x}}^{\star(k)}),N^{(k)}(\mathbf{x}) \right)$ and denote $\bar{G}$ with additional scenarios as $\bar{G}^{a}(\mathbf{x})= \frac{1}{n^{(k)}(\mathbf{x})} \sum_{j=1}^{n^{(k)}(\mathbf{x})} \{ \breve{G}_j(\mathbf{x})-\mbox{E}[\delta(\mathbf{x},\pmb{\xi}_{j})|\pmb{\xi}_{j}] \}$.

Given finite decision space $|\mathcal{X}|<\infty$, when we use the Stochastic kriging global metamodel searching for the optimal first-stage decision, Lemma~\ref{proposition2:samplepath_revisit} states that $n^{(k)}(\mathbf{x}) \rightarrow  \infty$ as the iteration $k \rightarrow \infty$ or the simulation budget $C\rightarrow \infty$. Since $n^{(k)}(\mathbf{x}) = \ceil{n_0 g^{t(\mathbf{x})-1}}$, the number of revisits $t(\mathbf{x}) \rightarrow \infty$ as $k \rightarrow \infty$, which further implies $\alpha^{(k)}(\mathbf{x}) \rightarrow 0$ as $k \rightarrow \infty$. %Notice that under this specific construction, all first-stage decision are forced to be observed and thus we can apply the corresponding proof in
By applying Lemma~\ref{lemma2:gapconv0}, we can show 
	$\delta^{(k)}(\mathbf{x}, \pmb{\xi}_{j}) \stackrel{a.s.}\rightarrow 0$ as $k\rightarrow \infty$.
	
\begin{sloppypar}
	Let $\bar{G}^c(\mathbf{x}) \equiv \frac{1}{n^{(k)}(\mathbf{x})}\sum_{j=1}^{n^{(k)}(\mathbf{x})}G_j(\mathbf{x})$ with $G_j(\mathbf{x})=c_0(\mathbf{x})+q(\mathbf{x},\mathbf{y}^\star,\pmb{\xi}_j)$. 
	For any $\nu>0$, since
	$\bar{G}^{a}(\mathbf{x})$, ${G}(\mathbf{x})$ and $\bar{G}^c(\mathbf{x})$ are all finite,
	by applying the triangle inequality, we have
	\begin{eqnarray}
	\lefteqn{|\bar{G}^a(\mathbf{x}) - {G}(\mathbf{x}) | =|\bar{G}^{a}(\mathbf{x}) -\bar{G}^c(\mathbf{x})+\bar{G}^c(\mathbf{x})- {G}(\mathbf{x}) | }
	\nonumber \\ 
	& < & |\bar{G}^{a}(\mathbf{x})-\bar{G}^c(\mathbf{x})| +|\bar{G}^c(\mathbf{x})- {G}(\mathbf{x}) | 
	\nonumber\\
	& < &\nu,  \mbox{ as $k\rightarrow \infty$.}
	\label{eq.mid1} 
	\end{eqnarray}
	%as $k \rightarrow \infty$. 
	The inequality in (\ref{eq.mid1}) follows because: (1) $|\bar{G}^c(\mathbf{x})- {G}(\mathbf{x}) |<\nu/2$ holds by applying $n^{(k)}(\mathbf{x}) \rightarrow  \infty$ and the strong law of large numbers; and (2) $|\bar{G}^{a}(\mathbf{x}) -\bar{G}^c(\mathbf{x})|<\nu/2$ holds by applying $\delta^{(k)}(\mathbf{x}, \pmb{\xi}) \rightarrow 0$.
	That means $\bar{G}^{a}(\mathbf{x}) \rightarrow {G}(\mathbf{x})$ w.p.1 as $k\rightarrow \infty$. Since ${G}(\mathbf{x}^\star)=\min_{\mathbf{x}\in \mathcal{X}}{G}(\mathbf{x})\leq {G}(\mathbf{x})$, we have
	\[
	\mbox{Pr}\{ \bar{G}^{a}(\mathbf{x}) < {G}(\mathbf{x}^\star) -\nu \quad \mbox{i.o.}\}=0
	\]
	as $k \rightarrow \infty$.
	Hence, with $|\mathcal{X}|<\infty$ and $\min_{\mathbf{x} \in \mathcal{X}} \bar{G}^{a}(\mathbf{x}) \leq  \mbox{Pr}\{ \bar{G}(\widehat{\mathbf{x}}^{\star (k)})$, we have
	\begin{eqnarray}
	\lefteqn{ \mbox{Pr}\{ \bar{G}(\widehat{\mathbf{x}}^{\star (k)}) < G(\mathbf{x}^\star)- \nu  \quad \mbox{i.o.} \} 
	\leq  \mbox{Pr}\{ \min_{\mathbf{x} \in \mathcal{X}} \bar{G}^{a}(\mathbf{x}) < G(\mathbf{x}^\star) - \nu \quad \mbox{i.o.} \} } \nonumber \\
	&\leq & \sum_{\mathbf{x} \in \mathcal{X}}\mbox{Pr}\{\bar{G}^{a}(\mathbf{x}) < G(\mathbf{x}^\star) - \nu \quad \mbox{i.o.} \}  = 0. \nonumber
	\end{eqnarray} 
	\end{sloppypar}
\end{proof}

\raggedbottom
%\end{screenonly}

\section{Proof of Lemmas~\ref{lemma1:revisit} and~\ref{lemma2:gapconv0}} 
%\begin{screenonly}
\begin{proof} (Lemma~\ref{lemma1:revisit})  
It suffices to prove $\mbox{Pr}\{\lim_{k\rightarrow\infty}N^{(k)}(\mathbf{x}) < N\} = 0 \ \forall \mathbf{x} \in \mathcal{X}, \forall N > 0$. By Lemma~\ref{proposition3:liminf}, for all $\mathbf{x}\in \mathcal{X}$, 
$$\mbox{Pr}\{\lim_{k\rightarrow\infty}N^{(k)}(\mathbf{x}) < N\} = \mbox{Pr}\{\lim_{k\rightarrow\infty}N^{(k)}(\mathbf{x})<N,\lim \inf_{k \rightarrow \infty} \bar{G}(\widehat{\mathbf{x}}^{\star (k)}) > G(\mathbf{x}^\star) - \nu\}$$
for any $\nu >0$. Let $\Omega(\mathbf{x}) = \{\lim_{k\rightarrow\infty}N^{(k)}(\mathbf{x})<N,\lim \inf_{k \rightarrow \infty} \bar{G}(\widehat{\mathbf{x}}^{\star (k)}) > G(\mathbf{x}^\star) - \nu\}$. It suffices to show that $\mbox{Pr}\{\Omega(\mathbf{x})\}=0$ for all $\mathbf{x}\in \mathcal{X}$.

Without loss of generality, we can follow Sun et al.~\citep{Sun_etal_2014} to set a predetermined large positive constant $\overline{M}$ such that $|\bar{G}(x)|< \overline{M}$. %By Equation~(\ref{eq.predictor}),
Thus, we have $|\widehat{\mu}_0|=[\mathbf{1}^\prime (\Sigma+V)^{-1}
	\mathbf{1}]^{-1}\mathbf{1}^\prime (\Sigma+V)^{-1}|\bar{\mathbf{G}}_{\mathcal{P}_o}| \leq \overline{M}$. Since $\Sigma$ is symmetric positive definite and $V$ is diagonal positive definite, we can apply eigendecomposition to $\Sigma$ such that $\Sigma=Q\Lambda Q^\prime$, $Q^{-1}=Q^\prime$. Furthermore, by Woodbury identity~\cite{GoluVanl96}, $[\Lambda+Q^{-1}VQ]^{-1}=\Lambda^{-1}-\Lambda^{-1}Q^\prime(V^{-1}+Q\Lambda^{-1}Q^{-1})^{-1}Q\Lambda^{-1}$ holds. Then, we have the following inequality,
\begin{eqnarray}
\mathbf{1}^\prime [\Sigma+V]^{-1}\mathbf{1}&=&\mathbf{1}^\prime [Q\Lambda Q^\prime+V]^{-1}\mathbf{1} \nonumber \\
&=& \mathbf{1}^\prime Q[\Lambda+Q^{-1}VQ]^{-1}Q^\prime\mathbf{1}\nonumber\\
&=& \mathbf{1}^\prime Q\left[\Lambda^{-1} - \Lambda^{-1}Q^\prime (V^{-1} + Q\Lambda^{-1} Q^{-1})^{-1}Q\Lambda^{-1}\right] Q^\prime\mathbf{1}
\label{eq.205} \\
 &=& \mathbf{1}^\prime Q\Lambda^{-1}Q^\prime\mathbf{1} -\mathbf{1}^\prime Q \Lambda^{-1}Q^\prime (V^{-1} + Q\Lambda^{-1} Q^{-1})^{-1}Q\Lambda^{-1}Q^\prime\mathbf{1}
\nonumber\\
&\leq& \mathbf{1}^\prime Q\Lambda^{-1}Q^\prime\mathbf{1}
\label{eq.206}\\
&\leq& \mathbf{1}^\prime [Q\Lambda Q^\prime]^{-1}\mathbf{1} \nonumber \\
&\leq& \mathbf{1}^\prime \Sigma^{-1}\mathbf{1}.
\label{eq.207}
\end{eqnarray} 
Step~(\ref{eq.205}) follows by applying  Woodbury identity.
Step~(\ref{eq.206}) follows because $(V^{-1} + Q\Lambda^{-1} Q^{-1})^{-1}$ is symmetric positive definite. Thus we also have
	\begin{eqnarray}
	\widehat{G}(\mathbf{x})&=&\widehat{\mu}_0
	+ \Sigma(\mathbf{x},\cdot)^\prime [\Sigma+V]^{-1}
	(\bar{\mathbf{G}}_{\mathcal{P}_o}-\widehat{\mu}_0\cdot \mathbf{1})
	\nonumber \\
	& \leq &  \overline{M} +2\cdot\Sigma(\mathbf{x},\cdot)^\prime [\Sigma+V]^{-1}\mathbf{1}\cdot\overline{M}
	\nonumber \\
	& \leq& \overline{M}(1+2\tau^2\mathbf{1}^\prime \Sigma^{-1}\mathbf{1}). \nonumber
	\label{eq.meanupperbound}
	\end{eqnarray}
	The last step follows by applying (\ref{eq.207}).
Notice that $\tau^2 \Sigma^{-1}$ is correlation matrix defined by the set of first-stage design points $\mathcal{P}_o$ with $|\mathcal{P}_o| = K$. Denote this correlation matrix as $$R_{\mathcal{P}_o\subseteq \mathcal{X}}\equiv \begin{pmatrix}
    r(\mathbf{x}_1,\mathbf{x}_1) & r(\mathbf{x}_1,\mathbf{x}_2)  & \dots  & r(\mathbf{x}_1,\mathbf{x}_K) \\
    r(\mathbf{x}_2,\mathbf{x}_1) & r(\mathbf{x}_2,\mathbf{x}_2)  & \dots  & r(\mathbf{x}_2,\mathbf{x}_K) \\
    \vdots & \vdots & \vdots & \ddots \\
    r(\mathbf{x}_K,\mathbf{x}_1) & r(\mathbf{x}_K,\mathbf{x}_2)  & \dots  & r(\mathbf{x}_K,\mathbf{x}_K) \\
\end{pmatrix}. $$ By $|\mathcal{X}|<\infty$ from Assumption~\ref{assumption:discrete}, there is only finite number ($2^{|\mathcal{X}|} - 1$) of those correlation matrices. Then, let $d_{max} = \min_{\mathcal{P}_o \subseteq \mathcal{X}}\{1 + 2\mathbf{1}^\prime R_{\mathcal{P}_o}\mathbf{1}\}$, we have $\widehat{G}(\mathbf{x}) \leq \overline{M} d_{max}$ for any $\mathbf{x} \in \mathcal{X}$ and $k>0$.

	For any first-stage design point $\mathbf{x}$ that has been observed, without loss of generality, we can follow Sun et al.~\citep{Sun_etal_2014} to assign a small positive variance, say $\sigma_{min}^2$, to bound the prediction uncertainty, $s^2(\mathbf{x})\geq\sigma_{min}^2$ to guarantee the exploration for those design points. For any first-stage design point $\mathbf{x}$ that has not been observed, since the feasible set $\mathcal{X}$ is discrete, by applying Proposition~\ref{proposition:errorlowerbound}, we get $s^2(\mathbf{x})\geq \widehat{\tau}^2[1-r(\underline{d})]$, where $\underline{d}=\min_{\mathbf{x},\mathbf{x}^\prime \in \mathcal{X}}\{|\mathbf{x}-\mathbf{x}^\prime|\}$. By Assumptions~\ref{assumption:spatialvariane} and \ref{assumption:consistency},  $\widehat{\tau}^2$ is the ML estimate of $\tau^2>0$ and it is consistent. Thus, when $k$ is large, there is $\epsilon>0$ such that $\widehat{\tau}^2 \geq\tau^2-\epsilon>0$. Let $\underline{s}^2=(\tau^2-\epsilon)[1-r(\underline{d})]$ and $\underline{\sigma}\equiv \min(\sigma_{min},\underline{s})$. Then, we have $s(\mathbf{x}) \geq\underline{\sigma}$, depending only on $\mathcal{X}$, $\pmb{\phi}$, correlation function $r$.

Therefore, for any $x\in \mathcal{X}$ and $\omega \in \Omega(\mathbf{x})$, there exists $K(\omega)$ such that $\forall k >K(\omega)$, $\bar{G}(\widehat{\mathbf{x}}^{\star (k)}) > G(\mathbf{x}^\star) - \nu$. For any $k > K(\omega)$, let $c=\frac{G(\mathbf{x}^\star) -\nu - \overline{M} d_{max}}{\underline{\sigma}}$. Since $c<\frac{\bar{G}(\widehat{\mathbf{x}}^\star)-\widehat{G}(\mathbf{x})}{	s(\mathbf{x})}$, we have
	\begin{equation}
	\label{eq:posteriorlowerbound}
		\mbox{Pr}\{\widetilde{G}(\mathbf{x})
	<\bar{G}(\widehat{\mathbf{x}}^\star) \} =\mbox{Pr}\left\{\frac{\widetilde{G}(\mathbf{x})-\widehat{G}(\mathbf{x})}{	s(\mathbf{x})}
	< \frac{\bar{G}(\widehat{\mathbf{x}}^\star)-\widehat{G}(\mathbf{x})}{	s(\mathbf{x})}\right\} \geq \mbox{Pr}\left\{\frac{\widetilde{G}(\mathbf{x})-\widehat{G}(\mathbf{x})}{	s(\mathbf{x})}
	<c\right\} = \Phi(c) > 0\nonumber
	\end{equation}
where $\Phi(\cdot)$ is the CDF of a
standard normal random variable. Then we have 
	\begin{equation}
	f^{(k)}(\mathbf{x})\geq\frac{1}{|\mathcal{X}|}\Phi(c)>0.
	\label{eq: sampling_lowerbound}\nonumber
	\end{equation}
By applying Lemma~\ref{proposition1:revisit}, we can get $\mbox{Pr}\{\Omega(\mathbf{x})\}=0$ for all $\mathbf{x}\in \mathcal{X}$.
\end{proof}

\begin{proof} (Lemma~\ref{lemma2:gapconv0})  

 Lemma~(\ref{lemma1:revisit}) states that $N^{(k)}(\mathbf{x}) \rightarrow \infty$ for all $\mathbf{x}\in\mathcal{X}$ as the iteration $k \rightarrow \infty$ or the simulation budget $C\rightarrow \infty$. Notice that $N^{(k)}(\mathbf{x}) = \ceil{n_0 g^{t(\mathbf{x})-1}}$, the number of revisits $t(\mathbf{x}) \rightarrow \infty$
	as $k \rightarrow \infty$.
	That means $\alpha^{(k)}(\mathbf{x}) \rightarrow 0$ as $k\rightarrow \infty$. 
	
	Let $\Theta_{j}^{(k)}(\mathbf{x}) \equiv \{y_{j}\in \mathcal{Y}(\mathbf{x})\mid \mathbf{y}_{j} \mbox{ has been visited at iteration $k$}\}$ denote a set including all visited points in the local search for $j$-th scenario $\pmb{\xi}_{j}$ in $k$-th iteration. For any $\mathbf{y}_j \in \mathcal{Y}(\mathbf{x})$ if $\mathbf{y}_{j} \notin \Theta_{j}^{(k)}(\mathbf{x})$,
	the expected improvement is strictly positive,
	$$\mbox{E}_{\widetilde{q}_{\mathbf{x}}(\mathbf{y}_{j},\pmb{\xi}_{j})}\left[\mathrm{I}(\mathbf{y}_{j}, \pmb{\xi}_{j})\right] \equiv	\mathrm{E}_{\widetilde{q}_{\mathbf{x}}(\mathbf{y}_{j},\pmb{\xi}_{j})}\left[ \max \left(
 	q_{\mathbf{x}}(\widehat{\mathbf{y}}^\star_{j},\pmb{\xi}_{j})-\widetilde{q}_{\mathbf{x}}(\mathbf{y}_{j},\pmb{\xi}_{j})
 	,0\right)\right]>0,$$
 	otherwise
		$\mbox{E}_{\widetilde{q}_{\mathbf{x}}(\mathbf{y}_{j},\pmb{\xi}_{j})}\left[\mathrm{I}(\mathbf{y}_{j}, \pmb{\xi}_{j})\right]=0.$  Notice that $0<q_{\mathbf{x}}(\mathbf{y}_{j}, \pmb{\xi}_{j}) \leq q_{\max} <\infty$ for any $\mathbf{x}$ and $\mathbf{y}_{j}$ ($q_{\max}$ exists since $q(\mathbf{x},\mathbf{y},\pmb{\xi})$ is finite continuous function). From Assumption~\ref{assumption:bounded}, we have $\left\Vert q\right\Vert_{\mathcal{H}_\phi(\mathcal{X})}\leq \overline{M}$. Let $\kappa(x)\equiv x\Phi(x)+\varphi(x)$. By applying Assumption~\ref{assumption:discrete} and Proposition~\ref{proposition:errorlowerbound}, we have  $s^2_{\mathbf{x}}(\mathbf{y}_{j}, \pmb{\xi}_{j}) \geq \sigma^2[1 - R(\underline{d})] >0$ for $\mathbf{y}_j \notin \Theta_{j}^{(k)}(\mathbf{x})$, where $\underline{d}=\min_{\mathbf{x} \in \mathcal{X};\mathbf{y},\mathbf{y}^\prime\in \mathcal{Y}(\mathbf{x})}\{| \mathbf{y}- \mathbf{y}^\prime |\}$ is the smallest distance between two second stage decision points. By Assumption~\ref{assumption:consistency}, under some regularity conditions, MLE $\widehat{\tau}^2$ is consistent. Thus when $k$ is large enough, there is $\epsilon$ such that $\widehat{\sigma}^2 \geq \sigma^2-\epsilon>0$ by Assumption~\ref{assumption:spatialvariane}.
		Then we have  $\underline{s}^2=(\sigma^2-\epsilon)[1 - R(\underline{d})]$ depending only on $\pmb{\xi}_j$, $R$, $\mathcal{X}$ and $\mathcal{Y}(x)$. By applying Lemma~8 in \cite{Bull:2011:CRE:1953048.2078198}, we have 
		$$\mbox{E}_{\widetilde{q}_{\mathbf{x}}(\mathbf{y}_{j},\pmb{\xi}_{j})}\left[\mathrm{I}(\mathbf{y}_{j}, \pmb{\xi}_{j})\right]\geq \sigma\kappa(-\overline{M}/\sigma)s_{\mathbf{x}}(\mathbf{y}_{j},\pmb{\xi}_{j})\geq\sigma\kappa(-\overline{M}/\sigma)\underline{s} >0.$$
% 		\frac{\kappa(-\overline{M}/\sigma)}{\kappa(\overline{M}/\sigma)}\mathrm{I}\geq \frac{\kappa(-\overline{M}/\sigma)}{\kappa(\overline{M}/\sigma)}\mathrm{I} >0
		Thus, since $\alpha^{(k)}(\mathbf{x}) \rightarrow 0$ as $k\rightarrow \infty$, $\exists k_0>0$ such that for any $\mathbf{y}_j \in\mathcal{Y}(\mathbf{x})/\Theta_{j}^{(k)}(\mathbf{x})$ and $k>k_0$ we have 
	$$\mbox{E}_{\widetilde{q}_{\mathbf{x}}(\mathbf{y}_{j},\pmb{\xi}_{j})}\left[\mathrm{I}(\mathbf{y}_{j}, \pmb{\xi}_{j})\right] \equiv	\mathrm{E}_{\widetilde{q}_{\mathbf{x}}(\mathbf{y}_{j},\pmb{\xi}_{j})}\left[ \max \left(
 	q_{\mathbf{x}}(\widehat{\mathbf{y}}^\star_{j},\pmb{\xi}_{j})-\widetilde{q}_{\mathbf{x}}(\mathbf{y}_{j},\pmb{\xi}_{j})
 	,0\right)\right]>\alpha^{(k)}(\mathbf{x})q_{max}.$$

 	According to the stopping criteria for second-stage optimization \eqref{eq.condAlpha},	
%  	\begin{equation}
% 	\frac{  \max_{\mathbf{y}_{j} \in \mathcal{Y}(\mathbf{x} )} \mbox{E}_{\widetilde{q}_{\mathbf{x}}(\mathbf{y}_{j},\pmb{\xi}_{j})}\left[\mathrm{I}(\mathbf{y}_{j}, \pmb{\xi}_{j})\right] }
% 	{ |     q_{\mathbf{x}}(\widehat{\mathbf{y}}^\star_{j},\pmb{\xi}_{j})            |}  \leq \alpha^{(k)}(\mathbf{x}). 
% 	\label{eq.condAlpha}
% 	\end{equation}, 
% 	all $\mathbf{y}_j\in\mathcal{Y}(\mathbf{x})/\Theta_{j}^{(k)}(\mathbf{x})$ don't satisfy the stopping criteria. Thus
	we eventually run simulations at all $\mathbf{y}_j \in\mathcal{Y}(\mathbf{x})/\Theta_{j}^{(k)}(\mathbf{x})$ as $\alpha^{(k)}(\mathbf{x})\rightarrow 0$.
% 	by Assumption~\ref{assumption:discrete}, which means we visit all solutions $\mathbf{y}_j$ in the finite feasible set $\mathcal{Y}(\mathbf{x})$ for $j$-th scenario $\pmb{\xi}_j$ in $k$-th iteration, i.e. $\Theta_{j}^{(k)}(\mathbf{x}) = \mathcal{Y}(x)$.
	Then by the definition of second-stage optimality gap, for $k >k_0$, we have zero optimality gap for any $\pmb{\xi}_{j}$,
	$$
	\delta^{(k)}(\mathbf{x},\pmb{\xi}_{j}) = q_{\mathbf{x}}(\widehat{\mathbf{y}}^\star_{j}, \pmb{\xi}_{j})-q_{\mathbf{x}}(\mathbf{y}^\star_{j}, \pmb{\xi}_{j})=\min_{\mathbf{y} \in \Theta_{j}^{(k)}(\mathbf{x})}q_{\mathbf{x}}(\mathbf{y}, \pmb{\xi}_{j}) - \min_{\mathbf{y} \in \mathcal{Y}(\mathbf{x})}q_{\mathbf{x}}(\mathbf{y}, \pmb{\xi}_{j}) = 0
	$$ for $j=1,2,\ldots,N(\mathbf{x})$. It also implies
 	$\delta(\widehat{\mathbf{x}}, \pmb{\xi}_{j}) \stackrel{a.s.}\rightarrow 0$ as $k\rightarrow \infty$ because
	$$\mbox{Pr}\{\lim_{k\rightarrow \infty}\delta(\widehat{\mathbf{x}}, \pmb{\xi}_{j})=0\}=1.$$
\end{proof}

%\end{screenonly}

%\medskip
\section{Proof of Theorem~\ref{thm:conv}}
%\begin{screenonly}
% ======================= Theorem 1 =================================

\begin{proof} %(Theorem~\ref{thm:conv})  
Recall $G(\cdot)$ is the objective function, 
	 $
	 G(\mathbf{x}_i)
	={c_0(\mathbf{x}_i)}+\mbox{E}_{\pmb{\xi}}
	\left[
	 {q}(\mathbf{x}_i, \mathbf{y}^\star(\mathbf{x}_i,\pmb{\xi}), \pmb{\xi})\right]$; $\bar{G}^c(\mathbf{x}_i)$ is the SAA of $G(\mathbf{x})$ without the optimality gap,
	$
\bar{G}^c(\mathbf{x}_i)
	= {c_0(\mathbf{x}_i)}+
	\frac{1}{N(\mathbf{x}_i)} \sum_{j=1}^{N(\mathbf{x}_i)}
	 {q}(\mathbf{x}_i, \mathbf{y}^\star(\mathbf{x}_i,\pmb{\xi}_{ij}), \pmb{\xi}_{ij});
$
and $\bar{G}(\mathbf{x}_i)$ is the SAA of $G(\mathbf{x})$ with the optimality gap, $
    \bar{G}(\mathbf{x}_i) = c_0(\mathbf{x}_i) +
    \frac{1}{N(\mathbf{x}_i)} \sum_{j=1}^{N(\mathbf{x}_i)} \{ q(\mathbf{x}_i,\widehat{y}^\star(\mathbf{x}_i,\pmb{\xi}_{ij}),\pmb{\xi}_{ij})-\mbox{E}[\delta(\mathbf{x},\pmb{\xi}_{ij})] \}.  $
	For any $\nu>0$, since
	$\bar{G}(\mathbf{x})$, ${G}(\mathbf{x})$ and $\bar{G}^c(\mathbf{x})$ are all finite,
	by applying the triangle inequality, we have
	\begin{eqnarray}
	\lefteqn{|\bar{G}(\mathbf{x}) - {G}(\mathbf{x}) | =|\bar{G}(\mathbf{x}) -\bar{G}^c(\mathbf{x})+\bar{G}^c(\mathbf{x})- {G}(\mathbf{x}) | }
	\nonumber \\ 
	& < & |\bar{G}(\mathbf{x}) -\bar{G}^c(\mathbf{x})| +|\bar{G}^c(\mathbf{x})- {G}(\mathbf{x}) | 
	\nonumber\\
	& < &\nu  \label{eq.triangle} 
	\end{eqnarray}
	as $k \rightarrow \infty$. 
	For the last inequality in (\ref{eq.triangle}), as $k\rightarrow \infty$, $|\bar{G}^c(\mathbf{x})- {G}(\mathbf{x}) |<\nu/2$ holds by applying Lemma~\ref{lemma1:revisit}, i.e. $N^{(k)}(\mathbf{x}) \rightarrow \infty$ and the strong law of large numbers, while $|\bar{G}(\mathbf{x}) -\bar{G}^c(\mathbf{x})|<\nu/2$ holds by applying Lemma~\ref{lemma2:gapconv0}.
	Thus, for any $\mathbf{x} \in \mathcal{X}$ w.p.1., we have
	\begin{equation}
	\mbox{Pr}\{ |\bar{G}(\mathbf{x}) - {G}(\mathbf{x}) | > \nu   \quad \mbox{i.o.} \} =0
	\label{eq.mid11}
	\end{equation}
	as $k \rightarrow \infty$.
	
	After that, we want to bound $|\bar{G}(\widehat{\mathbf{x}}^{\star (k)}) - G(\mathbf{x}^\star)|$. 
	Two cases can happen: either $\bar{G}(\widehat{\mathbf{x}}^{\star (k)}) \leq  G(\mathbf{x}^\star)$ or $\bar{G}(\widehat{\mathbf{x}}^{\star (k)}) >  G(\mathbf{x}^\star)$.
	When $\bar{G}(\widehat{\mathbf{x}}^{\star (k)}) \leq  G(\mathbf{x}^\star)$, notice 
	$G(\mathbf{x}^\star) \leq G(\widehat{\mathbf{x}}^{\star (k)})$
	since $G(\mathbf{x}^\star)$ is the optimal objective value. 
	Then, we have
	\[
	\bar{G}(\widehat{\mathbf{x}}^{\star (k)}) \leq  G(\mathbf{x}^\star) \leq G(\widehat{\mathbf{x}}^{\star (k)})
	\]
	which leads to 
	\begin{equation}
	\begin{aligned}
	|\bar{G}(\widehat{\mathbf{x}}^{\star (k)}) - G(\mathbf{x}^\star)| & \leq |\bar{G}(\widehat{\mathbf{x}}^{\star (k)}) - G(\widehat{\mathbf{x}}^{\star (k)})|
	& \leq   \max_{\mathbf{x} \in \mathcal{X}} | \bar{G}(\mathbf{x}) -G(\mathbf{x}) | 
	\end{aligned}
	\label{eq.mid2}
	\end{equation}
	
	When $\bar{G}(\widehat{\mathbf{x}}^{\star (k)}) >  G(\mathbf{x}^\star)$, notice
	$
	\bar{G}(\mathbf{x}^\star) \geq \bar{G}(\widehat{\mathbf{x}}^{\star (k)})
	$
	since $\bar{G}(\widehat{\mathbf{x}}^{\star (k)})$ is the current optimal estimate.
	Then, we have
	\[
	G(\mathbf{x}^\star) < \bar{G}(\widehat{\mathbf{x}}^{\star (k)}) \leq \bar{G}(\mathbf{x}^\star)
	\]
	which leads to
	\begin{equation}
	\begin{aligned}
	|\bar{G}(\widehat{\mathbf{x}}^{\star (k)}) - G(\mathbf{x}^\star)|  & \leq |\bar{G} (\mathbf{x}^{\star}) - G(\mathbf{x}^\star)|
	& \leq   \max_{\mathbf{x} \in \mathcal{X}} | \bar{G}(\mathbf{x}) -G(\mathbf{x}) | 
	\end{aligned}
	\nonumber
	\end{equation}
	Hence, we have
	\begin{equation}
	|\bar{G}(\widehat{\mathbf{x}}^{\star (k)}) - G(\mathbf{x}^\star)| \leq \max_{\mathbf{x} \in \mathcal{X}} | \bar{G}(\mathbf{x}) -G(\mathbf{x}) |
	\label{eq.mid3}
	\end{equation}
	
	By applying (\ref{eq.mid2}) and (\ref{eq.mid3}), for any $\nu >0$, we have
	\begin{eqnarray}
	\mbox{Pr}\{ |\bar{G}(\widehat{\mathbf{x}}^{\star (k)}) - G(\mathbf{x}^\star)| > \nu  \quad \mbox{i.o.} \}  &\leq& \mbox{Pr}\{ \max_{\mathbf{x} \in \mathcal{X}} | \bar{G}(\mathbf{x}) -G(\mathbf{x}) | > \nu
		\quad \mbox{i.o.} \} 
	\nonumber \\
	& \leq &\sum_{\mathbf{x} \in \mathcal{X}}  \mbox{Pr}\{ | \bar{G}(\mathbf{x}) -G(\mathbf{x}) | > \nu  \quad \mbox{i.o.} \} 
	\nonumber \\
	& =& 0 
	\label{eq.mid6}
	\end{eqnarray}
	as $k \rightarrow \infty$. By Assumption~\ref{assumption:discrete}, the inequality~(\ref{eq.mid6}) holds by applying Equation~(\ref{eq.mid11}).
	Thus, the convergence follows, $\bar{G}(\widehat{\mathbf{x}}^{\star (k)}) {\rightarrow} G(\mathbf{x}^\star)$ w.p.1 as the budget $C \rightarrow \infty$ or the iteration $k\rightarrow \infty$. 
\end{proof} 
  
%\end{screenonly}

% Bibliography
\bibliographystyle{unsrt} 
\bibliography{paper2stage}

\begin{thebibliography}{10}

\bibitem{Beale_1955}
E.~M.~L. Beale.
\newblock On minimizing a convex function subject to linear inequalities.
\newblock {\em Journal of the Royal Statistical Society. Series B
  (Methodological)}, 17(2):173--184, 1955.

\bibitem{Dantzig_1955}
George~B. Dantzig.
\newblock Linear programming under uncertainty.
\newblock {\em Management Science}, 1(3/4):197--206, 1955.

\bibitem{Shapiro_Philpott_2007}
Alexander Shapiro and Andy Philpott.
\newblock A tutorial on stochastic programming, 2007.

\bibitem{shapiro_2009}
Alexander Shapiro, Darinka Dentcheva, and Andrzej Ruszczynski.
\newblock {\em Lectures on Stochastic Programming: Modeling and Theory}.
\newblock SIAM, Philadelphia, 2009.

\bibitem{Liu_etal_2011}
Jingang Liu, Chihui Li, Feng Yang, Hong Wan, and Reha Uzsoy.
\newblock Production planning in semiconductor manufacturing via simulation
  optimization.
\newblock In S.~Jain, R.~R. Creasey, J.~Himmelspach, K.~P. White, and M.~Fu,
  editors, {\em Proceedings of the 2011 Winter Simulation Conference}, pages
  3617 -- 3627, Piscataway, New Jersey, 2011. Institute of Electrical and
  Electronics Engineers, Inc.

\bibitem{Ekin_Polson_Soyer_2014}
Tahir Ekin, Nicholas~G. Polson, and Refik Soyer.
\newblock Augmented markov chain monte carlo simulation for two-stage
  stochastic programs with recourse.
\newblock {\em Decision Analysis}, 11(4):250--264, 2014.

\bibitem{Frauendorfer_1988}
Karl Frauendorfer.
\newblock Solving slp recourse problems with arbitrary multivariate
  distributions: The dependent case.
\newblock {\em Mathematics of Operations Research}, 13(3):377--394, 1988.

\bibitem{Sen_Higle_1991}
Julia~L. Higle and Suvrajeet Sen.
\newblock Stochastic decomposition: {A}n algorithm for two-stage linear
  programs with recourse.
\newblock {\em Mathematics of Operations Research}, 16(3):650--669, 1991.

\bibitem{Higle_Sen_1994}
Julia~L. Higle and Suvrajeet Sen.
\newblock Finite master programs in regularized stochastic decomposition.
\newblock {\em Mathematical Programming}, 67(1):143--168, 1994.

\bibitem{Ahmed_Shapiro_2002}
Shabbir Ahmed, Alexander Shapiro, and Er~Shapiro.
\newblock The sample average approximation method for stochastic programs with
  integer recourse.
\newblock {\em SIAM J. Optim.}, 12:479--502, 2002.

\bibitem{Spall_1992}
James~C. Spall.
\newblock Multivariate stochastic approximation using simultaneous perturbation
  gradient approximation.
\newblock {\em IEEE Trans. Autom. Control}, 37(3):332--341, 1992.

\bibitem{Jones_1998}
Donald~R. Jones, Matthias Schonlau, and William~J. Welch.
\newblock Efficient global optimization of expensive black-box functions.
\newblock {\em J. Glob. Optim.}, 13(4):455--492, 1998.

\bibitem{Hong_Nelson_2006}
L.~Jeff Hong and Barry~L. Nelson.
\newblock Discrete optimization via simulation using compass.
\newblock {\em Operations Research}, 54(1):115--129, 2006.

\bibitem{Huang_etal_2006G}
D.~Huang, T.~T. Allen, W.~I. Notz, and N.~Zeng.
\newblock Global optimization of stochastic black-box systems via sequential
  kriging meta-models.
\newblock {\em J. Glob. Optim.}, 34(3):441--466, 2006.

\bibitem{Sun_etal_2014}
Lihua Sun, L.~Jeff Hong, and Zhaolin Hu.
\newblock Balancing exploitation and exploration in discrete optimization via
  simulation through a gaussian process-based approach.
\newblock {\em Operations Research}, 62(6):1416--1438, 2014.

\bibitem{Henderson_Nelson_2006}
Shane~G. Henderson and Barry~L. Nelson.
\newblock {\em Handbooks in Operations Research and Management Science},
  volume~13.
\newblock Elsevier, 2006.

\bibitem{Fu_2014}
Michael~C. Fu, editor.
\newblock {\em Handbook of Simulation Optimization}.
\newblock Springer, New York, NY, 2014.

\bibitem{Takriti_etal_1996}
Samer Takriti, John~R. Birge, and Erik Long.
\newblock A stochastic model for the unit commitment problem.
\newblock {\em IEEE Trans. Power Syst.}, 11(3):1497--1508, 1996.

\bibitem{Ruiz_etal_2009}
Pablo~A. Ruiz, C.~Russ Philbrick, and Peter~W. Sauer.
\newblock Wind power day-ahead uncertainty management through stochastic unit
  commitment policies.
\newblock In {\em 2009 IEEE/PES Power Systems Conference and Exposition}, pages
  1--9, 2009.

\bibitem{Tuohy_etal_2009}
Aidan Tuohy, Peter Meibom, Eleanor Denny, and Mark O'Malley.
\newblock Unit commitment for systems with significant wind penetration.
\newblock {\em IEEE Trans. Power Syst.}, 24(2):592 -- 601, 2009.

\bibitem{Dempster_etal_1981}
M.~A.~H. Dempster, M.~F. Fisher, L.~Jansen, and A.~H.~G. Rinnooy~Kan.
\newblock Analytical evaluation of hierarchical planning systems.
\newblock {\em Operations Research}, 29(4):707--716, 1981.

\bibitem{Birge_Dempster_1996}
John~R. Birge and M.~A.~H. Dempster.
\newblock Stochastic programming approaches to stochastic scheduling.
\newblock {\em J. Glob. Optim.}, 9(3):417--451, 1996.

\bibitem{Bailey_Jensen_Morton_1999}
T.~Glenn Bailey, Paul~A. Jensen, and David~P. Morton.
\newblock Response surface analysis of two-stage stochastic linear programming
  with recourse.
\newblock {\em Naval Research Logistics}, 46(7):753--776, 1999.

\bibitem{Slyke_Wets_1969}
R.~M. Van~Slyke and Roger Wets.
\newblock L-shaped linear programs with applications to optimal control and
  stochastic programming.
\newblock {\em SIAM Journal on Applied Mathematics}, 17(4):638--663, 1969.

\bibitem{Louveaux_1980}
François~V. Louveaux.
\newblock A solution method for multistage stochastic programs with recourse
  with application to an energy investment problem.
\newblock {\em Operations Research}, 28(4):889--902, 1980.

\bibitem{Birge_Louveaux_1988}
John~R. Birge and François~V. Louveaux.
\newblock A multicut algorithm for two-stage stochastic linear programs.
\newblock {\em European Journal of Operational Research}, 34(3):384--392, 1988.

\bibitem{Amaran_etal_2016}
Satyajith Amaran, Nikolaos~V. Sahinidis, Bikram Sharda, and Scott~J. Bury.
\newblock {Simulation optimization: {A} review of algorithms and applications}.
\newblock {\em Annals of Operations Research}, 240(1):351--380, 2016.

\bibitem{Nelson_2016}
Barry~L. Nelson.
\newblock `some tactical problems in digital simulation' for the next 10 years.
\newblock {\em Journal of Simulation}, 10(1):2--11, 2016.

\bibitem{Zabinsky_2009}
Zelda~B. Zabinsky.
\newblock {\em Random Search Algorithms}.
\newblock Wiley, 2009.

\bibitem{Xu_Nelson_Hong_2010}
Jie Xu, Barry~L. Nelson, and L.~Jeff Hong.
\newblock Industrial strength compass: A comprehensive algorithm and software
  for optimization via simulation.
\newblock {\em ACM Trans. Model. Comput. Simul.}, 20(1):3:1--3:29, 2010.

\bibitem{DANTZIG_GLYNN_1990}
George~B. Dantzig and Peter~W. Glynn.
\newblock Parallel processors for planning under uncertainty.
\newblock {\em Annals of Operations Research}, 22(1):1--21, 1990.

\bibitem{Shapiro_Homem-de-Mello_1998}
Alexander Shapiro and Tito Homem-de Mello.
\newblock A simulation-based approach to two-stage stochastic programming with
  recourse.
\newblock {\em Mathematical Programming}, 81(3):301--325, 1998.

\bibitem{Shapiro_Homem-de-Mello_2000}
Alexander Shapiro and Tito Homem-de Mello.
\newblock On the rate of convergence of optimal solutions of monte carlo
  approximations of stochastic programs.
\newblock {\em SIAM J. Optim.}, 11(1):70--86, 2000.

\bibitem{Shapiro_Nemirovski_2005}
Alexander Shapiro and Arkadi Nemirovski.
\newblock {\em Continuous Optimization}, volume~99 of {\em Applied
  Optimization}.
\newblock Springer, Boston, MA, 2005.

\bibitem{Quan_etal_2013}
Ning Quan, Jun Yin, Szu~Hui Ng, and Loo~Hay Lee.
\newblock Simulation optimization via kriging: {A} sequential search using
  expected improvement with computing budget constraints.
\newblock {\em IIE Transactions}, 45(7):763--780, 2013.

\bibitem{ankenman_nelson_staum_2010}
Bruce Ankenman, Barry~L. Nelson, and Jeremy Staum.
\newblock Stochastic kriging for simulation metamodeling.
\newblock {\em Operations Research}, 58(2):371--382, 2010.

\bibitem{Xie_Yi_2016}
Wei Xie and Yuan Yi.
\newblock A simulation-based prediction framework for two-stage dynamic
  decision making.
\newblock In T.~M.~K. Roeder, P.~I. Frazier, R.~Szechtman, E.~Zhou, T.~Huschka,
  and S.~E. Chick, editors, {\em Proceedings of the 2016 Winter Simulation
  Conference}, pages 2304--2315, Piscataway, New Jersey, 2016. Institute of
  Electrical and Electronics Engineers, Inc.

\bibitem{WarrenPowell_2014}
Warren Powell.
\newblock Clearing the jungle of stochastic optimization.
\newblock {\em INFORMS Tutorials in Operations Research}, pages 109--137, 2014.

\bibitem{Shapiro_2008}
Alexander Shapiro.
\newblock Stochastic programming approach to optimization under uncertainty.
\newblock {\em Mathematical Programming}, 112(1):183--220, 2008.

\bibitem{Chien_etal_2011}
{Chen Fu} Chien, St{\'e}phane Dauz{\`e}re-P{\'e}r{\`e}s, Hans Ehm, {John W.}
  Fowler, Zhibin Jiang, Shekar Krishnaswamy, {Tae Eog} Lee, Lars M{\"o}nch, and
  Reha Uzsoy.
\newblock Modelling and analysis of semiconductor manufacturing in a shrinking
  world: Challenges and successes.
\newblock {\em European Journal of Industrial Engineering}, 5(3):254--271,
  2011.

\bibitem{Ma_2011}
Yao Ma.
\newblock {\em Risk Management in Biopharmaceutical Supply Chains}.
\newblock PhD thesis, University of California, Berkeley, 2011.

\bibitem{Kaminsky_Wang_2015}
Philip Kaminsky and Yang Wang.
\newblock Analytical models for biopharmaceutical operations and supply chain
  management: A survey of research literature.
\newblock {\em Pharmaceutical Bioprocessing}, 3(1):61--73, 2015.

\bibitem{xie_nelson_staum_2010}
Wei Xie, Barry~L. Nelson, and Jeremy Staum.
\newblock The influence of correlation functions on stochastic kriging
  metamodels.
\newblock In {B. Johansson, S. Jain, J. Montoya-Torres, J. Hugan, and E.
  Yucesan}, editor, {\em Proceedings of the 2010 Winter Simulation Conference},
  pages 1067--1078, Piscataway, New Jersey, 2010. Institute of Electrical and
  Electronics Engineers, Inc.

\bibitem{Sacks_Welch_Mitchell_Wynn_1989}
Jerome Sacks, William~J. Welch, Toby~J. Mitchell, and Henry~P. Wynn.
\newblock Design and analysis of computer experiments.
\newblock {\em Statistical Science}, 4(4):409--423, 1989.

\bibitem{Kim_Nelson_2007}
Seong-Hee Kim and Barry~L. Nelson.
\newblock Recent advances in ranking and selection.
\newblock In {\em Proceedings of the 2007 Winter Simulation Conference}, pages
  162--172, Piscataway, New Jersey, 2007. Institute of Electrical and
  Electronics Engineers, Inc.

\bibitem{Tsai_Barry_Staum_2009}
Shing~Chih Tsai, Barry~L. Nelson, and Jeremy Staum.
\newblock {\em Combined screening and selection of the best with control
  variates}, volume 133 of {\em International Series in Operations Research \&
  Management Science}.
\newblock Springer, New York, 2009.

\bibitem{Lesnevski_Nelson_Staum_2007}
Vadim Lesnevski, Barry~L. Nelson, and Jeremy Staum.
\newblock Simulation of coherent risk measures based on generalized scenarios.
\newblock {\em Management Science}, 53(11):1756--1769, 2007.

\bibitem{Lesnevski_Nelson_Staum_2008}
Vadim Lesnevski, Barry~L. Nelson, and Jeremy Staum.
\newblock An adaptive procedure for estimating coherent risk measures based on
  generalized scenarios.
\newblock {\em Journal of Computational Finance}, 11(4):1--31, 2008.

\bibitem{Wald1949}
Abraham Wald.
\newblock Note on the consistency of the maximum likelihood estimate.
\newblock 20(4):595--601, 1949.

\bibitem{Birge:2011:ISP:2031490}
John~R. Birge and Franois Louveaux.
\newblock {\em Introduction to Stochastic Programming}.
\newblock Springer Publishing Company, Incorporated, 2nd edition, 2011.

\bibitem{GoluVanl96}
Gene~H. Golub and Charles~F. Van~Loan.
\newblock {\em Matrix Computations}.
\newblock The Johns Hopkins University Press, third edition, 1996.

\bibitem{Bull:2011:CRE:1953048.2078198}
Adam~D. Bull.
\newblock Convergence rates of efficient global optimization algorithms.
\newblock {\em J. Mach. Learn. Res.}, 12:2879--2904, November 2011.

\end{thebibliography}
% Sample .bib file with references that match those in
% the 'Specifications Document (V1.5)' as well containing
% 'legacy' bibs and bibs with 'alternate codings'.
% Gerry Murray - March 2012
%
%% History dates
%\received{February 2007}{March 2009}{June 2009}
\end{document}